\documentclass[11pt,reqno]{amsproc}
\usepackage[margin=1in]{geometry}
\usepackage{amsmath, amsthm, amssymb}
\usepackage{hyperref}
\hypersetup{colorlinks=true, pdfstartview=FitV, linkcolor=blue,citecolor=blue, urlcolor=blue}
\usepackage[abbrev,lite,nobysame]{amsrefs}
\usepackage{times}
\usepackage[usenames,dvipsnames]{color}
\usepackage{bbm}
\usepackage{bm}
\usepackage{subcaption}
\usepackage{mathtools}
\usepackage{pdfpages}
\usepackage{comment}
\definecolor{labelkey}{rgb}{0,0,1}

\newcommand{\R}{\mathbb{R}}

\newcommand{\N}{\mathbb{N}}

\newcommand{\D}{\mathcal{D}}

\providecommand{\R}{\mathbb{R}}

\providecommand{\N}{\mathbb{N}}

\newcommand{\pv}{\operatorname{p.\!v.}}

\renewcommand{\leq}{\leqslant}
\renewcommand{\geq}{\geqslant}

\renewcommand{\div}{\operatorname{div}}

\newcommand{\dist}{\operatorname{dist}}

\DeclareMathOperator{\supp}{supp}

  \newtheorem{thm}{Theorem}[section]
  \newtheorem{cor}[thm]{Corollary}
  \newtheorem{Lemma}[thm]{Lemma}
  
  \theoremstyle{definition}
    \newtheorem{Definition}[thm]{Definition}
  \theoremstyle{remark}
  \newtheorem{remark}[thm]{Remark}

      \newtheorem{Claim}[thm]{Claim}

\usepackage{etoolbox}
\patchcmd{\subsubsection}{\itshape}{\itshape\bfseries}{}{} 

\title[]{A few representation formulas for solutions of fractional Laplace equations}

\author[]{Sidy M. Djitte and Franck Sueur}

\address[S. M. Djitte]{Friedrich-Alexander-Universit\"at Erlangen-N\"urnberg, Department of Mathematics, Chair for Dynamics, Control, Machine Learning and Numerics (Alexander von Humboldt Professorship), Cauerstr. 11, 91058 Erlangen, Germany. 
}
\email{sidy.m.djitte@fau.de}
\address[F. Sueur]{IMB, Universit\'e de Bordeaux, France   $\&$ Institut  Universitaire de France}
\email{franck.sueur@math.u-bordeaux.fr}


\begin{document}
\maketitle

\begin{abstract}
This paper is devoted to the Laplacian operator of fractional order $s\in (0,1)$ in several dimensions.
We first establish a representation formula for the partial derivatives of the solutions of the homogeneous Dirichlet problem. 
 Along the way, we obtain a Pohozaev-type identity for the fractional Green function and of the fractional Robin function. The latter extends to the fractional setting a formula obtained by Br\'ezis and Peletier, see  \cite{Bresiz}, in the classical case of the Laplacian.
As an application we consider the particle system extending the classical point vortex system to the case of a fractional Laplacian. We observe that, for a single particle in a bounded domain, the properties of the fractional Robin function are crucial for the study of the steady states.
 We also extend the classical Hadamard variational formula to the fractional Green function as well as to the shape derivative of weak solution to the  homogeneous Dirichlet problem. 
 Finally we turn to the inhomogeneous Dirichlet problem and extend a formula by J.L. Lions, see \cite{Lions}, regarding the kernel of the reproducing kernel Hilbert space of harmonic functions to the case of $s$-harmonic functions. We observe that, despite the order of the operator is not $2$, this formula looks like the  Hadamard variational formula, answering in a negative way to an open question raised in \cite{ELPL}.
\end{abstract}

\setcounter{tocdepth}{3}
\tableofcontents

\section{Introduction and main results}


\subsection{Introduction}
\label{sec-intro}

The fractional Laplace operators appear in different disciplines
of mathematics (PDEs as well as probability theory), in various
 applications  (issued from biology, ecology or finance) and have numerous definitions, based on 
Fourier analysis, harmonic extensions, semigroup theory or  quadratic form, among others, see \cite{FRO,Kwas} and the references therein for more.
One quite explicit formula for the  fractional Laplace operator of order $s$, with $s\in (0,1)$, in $N\in \N^*$ dimensions, reads 
$$
(-\Delta)^su (z) := c_{N,s}\,  \pv \int_{\R^N}\frac{u(z)-u(y)}{|z-y|^{N+2s}}dy,
$$
with 
$$c_{N,s} :=\pi^{N/2}s4^s\frac{\Gamma(\frac{N+2s}{2})}{\Gamma(1-s)},$$
where $\Gamma$ is the gamma function, and where $\pv$ refers to the Cauchy principal value. 
Above $u$ is a function from a subset of $\R^N$ with real values. 
The reason for the presence of the normalization constant $c_{N,s}$
 is to match with the Fourier definition  which sets the fractional Laplace operator $(-\Delta)^s$ as the Fourier multiplier of symbol 
$|\xi|^{2s}$. 
On the other hand, the operator  $(-\Delta)^s$ is naturally associated with the bilinear form: 
$$
\mathcal E_s(u,u):=\frac{c_{N,s}}{2}\iint_{\R^N\times\R^N}\frac{(u(x)-u(y))^2}{|x-y|^{N+2s}}dxdy.
$$

Throughout this manuscript, we shall denote by $H^s(\R^N)$ the set of all those $L^2(\R^N)$ functions $u$ for which $\mathcal E_s(u,u)<\infty$ and for a given bounded open set $\Omega\subset\R^N$, we let $\mathcal H^s_0(\Omega)$ to be elements of $ H^s(\R^N)$ for which $u\equiv 0$ in $\R^N\setminus\Omega$.

\subsection{Representation formulas for partial derivatives}
\label{sec-rep}

For a given source term $f: \Omega\to\R$, the Dirichlet problem for the fractional Laplacian reads:
\begin{equation}\label{Dir-01}
(-\Delta)^s u=f\;\;\text{in} \quad\Omega\quad\&\quad 
u=0\quad 
\textrm{in }\;\;\R^N\setminus\Omega. 
\end{equation}
Under reasonable assumption on $f:\Omega\to \R$, it is well known that the problem above admits a unique solution which can be represented by \cite[Proposition 1.2.2]{Abatangelo}
\begin{equation}\label{repres-sol-Dir}
u(x)=\int_{\Omega}G_s(x,y)f(y)dy\;\;\;\text{for}\;x\in\Omega,
\end{equation}
where $G_s(x,\cdot)$ is the Green function associated to the operator $(-\Delta)^s$, that is, the solution to \begin{equation} \label{DIR}
    (-\Delta)^sG_s(x,\cdot)=\delta_x \quad\text{in}\quad \mathcal D'(\Omega)\qquad\&\qquad G_s(x,\cdot)=0\quad\text{in}\quad \R^N\setminus\Omega.
\end{equation}
Here $\delta_x$ denote the Dirac delta distribution. Formula \eqref{repres-sol-Dir} plays an important role in the qualitative and quantitative analysis of the problem \eqref{Dir-01}. For example, if $f\in L^\infty(\Omega)$, using the estimate on the Green function (see e.g \eqref{eq:greenfunct-estimate}) one easily derives from \eqref{repres-sol-Dir} that $u\in L^\infty(\Omega)$ as well with $\|u\|_{L^\infty(\Omega)}\leq C\|f\|_{L^\infty(\Omega)}$.
Further applications of the formula \eqref{repres-sol-Dir} arise in regularity theory but also in the study of symmetry and sign properties of solutions (see \cite{FallTobias}). More general results on representation formulas for solutions to fractional Dirichlet problems can be found in \cite{ASS, Abatangelo} and the references therein.

Here we are interested in representation formulas for partial derivatives of solutions in terms of the source term $f$. 
Because of the singularity of the  Green function, one cannot simply apply a derivative to the formula \eqref{repres-sol-Dir} as for a regular integral with parameter.

In the classical case and when the data are sufficiently regular, a representation formula for partial derivatives of solutions can be derived by a simple application  of the Green's third formula which we recall in what follow. Given $u\in C^2(\Omega)\cap C^1(\overline\Omega)$, we have 
\begin{equation}\label{classical-repres-formula}
    u(x)=\int_{\Omega}G_1(x,y)(-\Delta)u(y)dy-\int_{\partial\Omega}\frac{\partial G_1(x,y)}{\partial{\nu_y}}u(y)dy,\;\;\;\text{for\,\, $x\in\Omega$},
\end{equation}
where $G_1(x,\cdot)$ is the Green function of the Laplacian with singularity at $x$.
\par\;

Indeed, if $u$ is sufficiently regular, say for instance, $u\in C^3(\Omega)\cap C^2(\overline\Omega)$ so that $\partial_i u\in C^2(\Omega)\cap C^1(\overline \Omega)$, then plugging  $\partial_ju$ into \eqref{classical-repres-formula} and integrating by parts the first term, we get 
\begin{equation*}
    \partial_iu(x)= - \int_{\Omega}\frac{\partial G_1(x,y)}{\partial y_i}(-\Delta)u(y)dy-\int_{\partial\Omega}\frac{\partial G_1(x,y)}{\partial{\nu_y}}\partial_i u(y)dy,\;\;\;\text{for\,\, $x\in\Omega$}.
\end{equation*}
Thus if $f\in C^1(\overline\Omega)$, $\Omega$ is smooth enough, and $u$ is the solution of 
$$
(-\Delta)u=f\;\;\text{in}\;\;\Omega\quad\&\quad u=0\;\;\text{on}\;\;\partial\Omega,
$$
then $\partial_iu$ can be represented by 
\begin{equation}\label{repres-deriv-classical-Dir}
    \partial_iu(x)=-\int_{\Omega}\frac{\partial G_1(x,y)}{\partial y_i}f(y)dy
   -\int_{\partial\Omega}\frac{\partial G_1(x,y)}{\partial{\nu_y}}\partial_i u(y)dy,\;\;\;\text{for\,\, $x\in\Omega$}.
\end{equation}
\par\;

Our first result establishes the counterpart of the identity \eqref{repres-deriv-classical-Dir} when the Laplacian is replaced by the fractional Laplacian $(-\Delta)^s$. In other words, we shall prove the following result.
\begin{thm}\label{main-result}
Let $N\in\N_*$ with $N\geq 2s$ and let $\Omega$ be a bounded open set of $\R^N$ of class $C^{1,1}$. Let $f: \R^N\times \R\to \R, (h,q)\mapsto f(h,q)$ be a given real-valued function such that  $f\in C^\alpha_{loc}(\R^N\times \R)$ for some $\alpha\in (0,1)$ if \;$2s>1$ and $f\in C^{1,\alpha}_{loc}(\R^N\times\R)$ if\; $2s\leq 1$. Let $u\in C^s(\R^N)\cap\mathcal H^s_0(\Omega)$ be a weak solution 
\begin{equation}\label{sect3-H}
(-\Delta)^su=f(x,u)\;\;\text{in $\Omega$}\quad\&\quad u=0\;\;\text{in}\;\;\R^N\setminus\Omega.
\end{equation}
Then for all $x\in \Omega$, there holds for any $i=1,2,\cdots N$,
\begin{equation}\label{repres-partial-deriv}
\frac{\partial u}{\partial x_i}(x)= -\int_\Omega \frac{\partial G_s(x,\cdot)}{\partial y_i} f(y,u(y))\,dy-\Gamma^2(1+s)\int_{\partial \Omega} \frac{u}{\delta^s_\Omega} \frac{G_s(x,\cdot)}{\delta^s_\Omega} \nu_i\,d\sigma ,\quad \text{if}\,\, 2s>1.
\end{equation}
\par\;

In the case $2s\leq 1$, we have
\begin{equation}\label{repres-partial-deriv-1}
\frac{\partial u}{\partial x_i}(x)=  \int_\Omega  G_s(x,\cdot)\Bigg[\frac{\partial f}{\partial h_i}(y,u(y))+\frac{\partial f}{\partial q}(y,u(y))\frac{\partial u}{\partial y_i}(y)\Bigg]dy-\Gamma^2(1+s)\int_{\partial \Omega} \frac{u}{\delta^s_\Omega} \frac{G_s(x,\cdot)}{\delta^s_\Omega} \nu_i\,d\sigma
\end{equation}
where  $\nu_i=\nu\cdot e_i$ is the $i$-coordinate of the outward unit normal and $\delta_\Omega:=\dist(\cdot,\partial\Omega)$ is the distance function to the boundary $\partial\Omega$.

In particular in the case $(-\Delta)^su=\text{const}$ in $\Omega$ \;\;\&\;\; $u=0$ in $\R^N\setminus\Omega$, we have 
\begin{equation} \label{dedu}
    \frac{\partial u}{\partial x_i}(x)= - \Gamma^2(1+s)\int_{\partial \Omega} \frac{u}{\delta^s_\Omega} \frac{G_s(x,\cdot)}{\delta^s_\Omega} \nu_i\,d\sigma\quad\text{for all $s\in (0,1)$}.
\end{equation}
\end{thm}
\begin{remark}A few remarks are in order:

\begin{itemize}
\item First, by regularity theory (see \cite{RS}) for each $x\in \Omega$, the function: $\partial\Omega\to \R, z\mapsto G_s(x,z)/\delta^s_\Omega(z)$ is well defined and  any $u$ satisfying \eqref{sect3-H} is in $C^1_{loc}(\Omega)$ and $u/\delta^s_\Omega\in C^\alpha(\overline{\Omega})$ for some $\alpha\in (0,1)$. Moreover,  $G_s(x,\cdot)/\delta^s_\Omega\in L^1(\partial\Omega)$ and also $\frac{\partial G_s(x,\cdot)}{\partial z_i}\in L^1(\Omega)$ provided that $s\in (1/2,1)$ (see appendix \ref{subsec-9.1}). Finally, since $\delta^{1-s}\nabla u\in L^\infty(\Omega)$ (see e.g \cite{RS}), it is also not difficult to check that $\delta^{s-1}G_s(x,\cdot)\in L^1(\Omega)$.  Therefore all the quantities appearing in the identities above are well defined.

\item When applied to the overdetermined boundary value problem
\begin{equation}\label{over-pde}
\left\{ \begin{array}{rcll} (-\Delta)^su &=& f(x,u) &\textrm{in }\;\;\Omega \\ u&=&0& 
\textrm{in }\;\;\R^N\setminus\Omega\\ u/\delta^s_\Omega&=&0&\text{on}\;\; \partial\Omega, \end{array}\right. 
\end{equation}
Theorem \ref{main-result} gives the following nontrivial result: Let $u\in L^\infty(\Omega)$ be a bounded solution of \eqref{over-pde}, then there holds $\nabla u\in L^\infty(\Omega)$ provided that $2s>1$.

Under certain condition on $f$, it is expected that zero is the only solution of \eqref{over-pde}. The latter is known (thanks to Pohozaev identity) when $f$ is either a constant or a pure power non-linearity see e.g \cite{RO-S, DFW2023}. We believe that  the regularity result stated above may play somehow a role into proving this conjecture.

\item When the data are sufficiently regular, it is possible to derive a formula like \eqref{repres-partial-deriv-1} by using the following Green's type identity for the fractional Laplacian (see \cite[Proposition 1.2.2]{Abatangelo}): for any $u\in C^{2s+\varepsilon}_{loc}(\Omega)$ such that $\delta^{1-s}_\Omega u\in C(\overline{\Omega})$, $\|u\|_{\mathcal L^1_s(\R^N)}:=\int_{\R^N}\frac{|u(y)|}{1+|y|^{N+2s}}dy<\infty$ and $u=0$ in $\R^N\setminus\overline{\Omega}$, there holds
\begin{equation}\label{Abatangelo}
u(x)=\int_{\Omega}G_s(x,y)(-\Delta)^su(y)dy+\int_{\partial\Omega}\frac{G_s(x,\theta)}{\delta^s_\Omega(\theta)}Eu(\theta)d\theta\quad\text{for $x\in\Omega$},
\end{equation}
where $$ Eu(\theta)=\lim_{\Omega\ni x\to\theta}\frac{\delta^{1-s}_\Omega(x)u(x)}{\int_{\partial\Omega}M_\Omega(x,\theta')d\theta'}\quad\text{and}\quad M_\Omega(x,\theta')=\lim_{\Omega\ni y\to \theta'}\frac{G_s(x,y)}{\delta^s_\Omega(y)}.$$
\par\;
Indeed, if $u$ solves \eqref{sect3-H} with $f(x,u)=f(x)\in C^2(\overline{\Omega})$, then $\partial_ju\in C^{2s+\varepsilon}_{loc}(\Omega)\cap \mathcal L^1_s(\R^N)$ solves 
    
\begin{equation}\label{partial-u-pde}
\left\{ \begin{array}{rcll} (-\Delta)^s\partial_ju &=& \partial_jf &\textrm{in }\;\;\Omega \\ u&=&0&
\textrm{in }\;\;\R^N\setminus\overline\Omega\\ \delta^{1-s}_\Omega(x)\partial_ju(x)&=&s(u/\delta^s_\Omega)(x)\nu_j(x)&\text{on}\;\; \partial\Omega. \end{array}\right. 
\end{equation}
The last identity in \eqref{partial-u-pde} follows from recent regularity result (see e.g \cite[Eq (1.5)]{FallSven}). Using \eqref{Abatangelo} with $\partial_ju$ and integrating by parts in the first integral, we get
\begin{align*}\label{Abatangelo}
\partial_ju(x)&=\int_{\Omega}G_s(x,y)(-\Delta)^s(\partial_ju)(y)dy+\int_{\partial\Omega}\frac{G_s(x,\theta)}{\delta^s_\Omega(\theta)}E(\partial_ju)(\theta)d\theta,\\
&=\int_{\Omega}G_s(x,y)\partial_jf(y)dy+s\int_{\partial\Omega}\frac{G_s(x,\theta)}{\delta^s_\Omega(\theta)}\frac{u(\theta)}{\delta^s_\Omega(\theta)}E\nu_j(\theta)d\theta,
\end{align*}
where we denote 
$$
E\nu_j(\theta)=\lim_{\Omega\ni x\to\theta}\frac{\nu_j(x)}{\int_{\partial\Omega}M_\Omega(x,\theta')d\theta'}.
$$
We arrive at the desire identity once we prove that $sE\nu_j(\theta)=-\Gamma^2(1+s)\nu_j(\theta)$ holds for all $\theta\in \partial\Omega$. Unfortunately, it is not clear to us how to get the latter for an arbitrary bounded open set $\Omega$.
\end{itemize}
\end{remark}
\par\;

To prove the identity \eqref{repres-partial-deriv}, we use instead the following identity (see e.g \cite[Theorem 1.3]{DFW2023} or \cite[Theorem 1.9]{RO-S}).
\begin{equation}\label{RS-DFW}
\int_\Omega \frac{\partial v}{\partial z_i}(-\Delta)^sw\,dz =
-\int_\Omega \frac{\partial w}{\partial z_i}(-\Delta)^sv\,dz-\Gamma^2(1+s)\int_{\partial\Omega}\frac{v}{\delta^{s}_\Omega}\frac{w}{\delta^{s}_\Omega}\,\nu_i\,d\sigma.
\end{equation}
Formally, we get \eqref{repres-partial-deriv} by using $G_s(x,\cdot)$ and $u$ as test functions into \eqref{RS-DFW}. Unfortunately, the function $G_s(x,\cdot)$ is too irregular to be admissible in \eqref{RS-DFW}. To overcome this difficulty, we approximate $G_s(x,\cdot)$ by a $C^\infty_csomehow $-function of the form $\eta_k \psi_{\mu,x}G_s(x,\cdot)$ where $\eta_k$  and $\psi_{\mu,x}$ are suitable cut-off functions which vanish near the boundary $\partial\Omega$ and near the singular point $x$ of the Green function respectively. Next we plug $\eta_k \psi_{\mu,x}G_s(x,\cdot)$ and $\eta_k u$ into \eqref{RS-DFW}. We expand by using the product rule for the fractional Laplacian, and then pass to the limit as $k\to\infty$ and $\mu\to 0^+$ respectively to arrive at the desire identity. The work for the passage to the limit w.r.t to $k$ was mainly done in \cite{DFW}. The novelty here concerns the passage to the limit w.r.t to $\mu$. We refer to Section \ref{sec4.1} below for more details.

\subsection{Representation formulas for  gradient of the fractional Robin function}
\label{sec-rep-grad}

The fractional Green function $G_s$ can be split into:
\begin{equation}\label{Eq-spliting-of-Green}
    G_s(x,\cdot)= F_s(x,\cdot)-H_\Omega(x,\cdot),
\end{equation}
where 
\begin{equation}\label{funda}
   F_s(x,\cdot) :=\frac{b_{N,s}}{|x-\cdot|^{N-2s}}, 
\end{equation}
 is the  fundamental solution of $(-\Delta)^s$ and 
$H_\Omega(x,\cdot)$ solves the equation
\begin{equation}\label{regular-part-of-Green}
\left\{ \begin{array}{rcll} (-\Delta)^s H_\Omega(x,\cdot)&=& 0  &\textrm{in }\Omega, \\ H_\Omega(x,z)&=&F_s(x,z)&
\textrm{in }\R^N\setminus\Omega. \end{array}\right. 
\end{equation}

In this section, we are interested in a representation formula for the gradient of the fractional Robin function. We recall that the fractional Robin function associated with the domain $\Omega$, denoted as $\mathcal R_s^\Omega:\Omega\to \R$ is given for $x$ in $\Omega$ by 
\begin{equation}\label{def-s-robin}
\mathcal R_s^\Omega(x):= H_\Omega(x,x),
\end{equation}
 where $H_\Omega(x,\cdot)$ is the regular part of the Green function $G_s(x,\cdot)$ as defined in \eqref{regular-part-of-Green}.  To ease the notation, from now on, we simply write $\mathcal R_s$ at the place of $\mathcal R_s^\Omega$. When $s=1$, i.e, in the classical case of the Laplacian we write $\mathcal R$ instead of $\mathcal R_1$.  The Robin function plays an important role in various fields of the mathematics such as geometric function theory, capacity theory, concentration problems (see e.g \cite{Bandle, Bresiz,Flucher} and the references therein). In particular, critical points of the Robin function (or the so-called harmonic centers of $\Omega$) appear when studying elliptic boundary value problems involving concentration of energy \cite{Flucher}. And also for elliptic problems involving critical Sobolev exponent, the number of solutions is linked to the the number of (non-degenerate) critical points of the Robin function (see e.g \cite{O.Rey}). Despite the great importance of Robin function, many questions regarding its properties are still unanswered even in the classical case. Below, we prove some basics properties (integral representation of the gradient and non-degeneracy in symmetric domains) of the fractional Robin function. First, let us recall the following result by Br\'ezis and Peletier regarding the variation of the classical Robin function with respect to the space variable $x$.
 \begin{thm}\cite{Bresiz}\label{thm-brezis}
     Let $\Omega$ be a bounded open set of $\R^N$ ($N\geq 2$) and let  $\mathcal R$ be the Robin function of $\Omega$ associated with the classical Laplacian. Then  for any $x$ in $\Omega$,
     \begin{align}\label{repres-gradient-class-Robin}
         \nabla\mathcal R(x)=\int_{\partial\Omega}(\partial_\nu G_1(x,\cdot))^2\nu d\sigma,
     \end{align}
     where $\nu$ is the outward unit normal to $\partial\Omega$.
 \end{thm}
Our first result generalizes Theorem \ref{thm-brezis} to the fractional setting. It reads as follows. 

\begin{thm}[Pohozaev-type identities for the fractional Green function]\label{Thm.2}
Let $N\in \N_*$ with $N\geq 2s$ and let $\Omega$ be a bounded open set of $\R^N$ of class $C^{1,1}$. Let $x,y\in \Omega$ with $x\neq y$. Then for all $s\in (0,1)$ there holds:
\begin{equation}\label{partial-Green}
 \frac{\partial G_s}{\partial x_i}(y,x)+\frac{\partial G_s}{\partial y_i}(x,y)=  -\Gamma^2(1+s)\int_{\partial \Omega} 
 \frac{G_s(x,\cdot)}{\delta^s_\Omega} \frac{G_s(y,\cdot)}{ \delta^s_\Omega}
 \nu_id\sigma.    
\end{equation}

Moreover,
\begin{equation}\label{repres-gradient-Robin func}
    \nabla \mathcal R_s(x)=\Gamma^2(1+s)\int_{\partial \Omega} \Bigg(\frac{G_s(x,\cdot)}{ \delta^s_\Omega}\Bigg)^2\nu\,d\sigma .
\end{equation}
where $\nu$ is the outward unit normal to $\partial\Omega$.
\end{thm}
\begin{remark}
    Formula \eqref{repres-gradient-Robin func} is consistent with \eqref{repres-gradient-class-Robin} since the latter can be obtained by formally sending $s$ to  $1^-$ in \eqref{repres-gradient-Robin func}.
\end{remark}

The proof of Theorem \ref{Thm.2} is similar to the proof of Theorem \ref{main-result}. To get the former, we use simultaneously the functions $\R^N\to\R, \;z\mapsto \psi_{\gamma,y}(z)G_s(y,z)\eta_k(z)$ and $\R^N\to\R,\;z\mapsto\psi_{\mu,x}(z)G_s(x,z)\eta_k(z)$ with $x\neq y$ as test functions into \eqref{RS-DFW} and then argue as in the proof of Theorem \ref{main-result}. We refer to Section \ref{sec4} for more details.
\par\;

As a corollary we obtain the following result in the case where the domain enjoys some symmetries. It is inspired by \cite{G} and proved,  thanks to formula \eqref{repres-gradient-Robin func},  in Section \ref{sec-robin}. 

\begin{cor}\label{non-deg-frac-Robin}
    Let $\Omega$ be a bounded open set of class $C^{1,1}$ of $\R^N$ with $N\geq 2$. 
  Assume that there is $j\in \{1,2,\cdots, N\}$ such that $\Omega$ is symmetric with respect to the hyperplane $H_j:=\big\{x\in\R^N:x_j=0\big\}$ and that $0$ is in $\Omega$. 
 Then $\partial_j  \mathcal R_s(0)=0$.
 If moreover, there is $i\in \{1,2,\cdots, N\}$, with $i \neq j$, such that $\Omega$ is also symmetric with respect to the hyperplane $H_i$, then $\partial_{ij}  \mathcal R_s(0)=0$.
\end{cor}

In the case of the Laplace operator, it is known that the classical  Robin function of a convex domain in two dimensions is convex and  admits only one non-degenerate critical point, see \cite{CF} by
Caffarelli and Friedman and \cite{G} for an alternative complex variable proof. Both  proofs use that, for simply connected domain, the classical  Robin function satisfies 
 the Liouville equation, as a consequence of the Riemann conformal mapping theorem. 
Since such an equation is not available for the fractional  Robin function, the question of whether or not a similar convexity property holds as well in the fractional setting is an interesting open problem.

In the particular case $N=1$, $\Omega=(-1,1)$ and $s=1/2$, a direct computations shows that $0$ is the only non-degenerate critical point of $\mathcal R_{1/2}$ since in this case, the Robin function is simply given by $R_{1/2}(x)=\frac{1}{\pi}\log 2(1-x^2)$ (see e.g \cite{Bucur}.\par\;

More recently, it has been proved by  Bartsch, Micheletti and Pistoia in \cite{GMP} that for a bounded domain, possibly multiply connected, in any dimensions, critical points have generically the Morse property. 
The proof makes use of the shape derivative of the classical Green function and of a Baire type argument, which can therefore be adapted to the case of the fractional Robin function.

Let us also mention that such results regarding shape derivatives are very useful in order to prove the eigenvalues are generically unique, see \cite{Albert,BMP,CKL,OZ} in the classical case and \cite{FGMP} in the fractional case.

\subsection{Application to the $s$-point vortices}
\par\;
\par\;

As an application of the previous section, we consider the particle system, so-called $s$-point vortices system, extending to the case of a Laplacian operator of fractional order $s\in (0,1)$ the  point vortex system, a classical model  which goes back to Helmholtz \cite{Helmholtz}, Kirchhoff \cite{Kirchhoff}, Poincar\'e \cite{Poincare}, Routh \cite{Routh}, Kelvin \cite{Kelvin} and Lin \cite{Lin1,Lin2}. 
In these works it was thought  as an idealized fluid model where the vorticity of an ideal incompressible two-dimensional fluid is discrete.  Although it does not constitute a solution of the incompressible Euler equations, even in the sense of distributions,  point vortices can be viewed as Dirac masses obtained as  limits of concentrated smooth vortices which evolve according to the Euler equations. 
In  the case of several vortices, including in  the case where there is also a part of the vorticity which is absolutely continuous with respect to Lebesgue measure, a justification  was given by Marchioro and Pulvirenti, see \cite{MP}. 
Let us also mention that Gallay has   proven in \cite{Gallay} that the point vortex system can also be obtained as vanishing viscosity limits of concentrated smooth vortices evolving according to the incompressible Navier-Stokes equations.
Finally, it has been recently proved in a series of works, see \cite{GS} and the reference therein, that 
the classical point vortex motion  can also be viewed as the limit of the dynamics of rigid particles in motion in an irrotational flow, when these particles 
shrink into  point massless particles, while the circulation is held fixed.
\medskip

 For  $N$ point vortices  occupying the positions 
 $x_{i}  \in  \mathbb \R^2$ with respective strength   $\gamma_i \in \R^*$, where $1 \leq i \leq N$,
 the dynamics reads:
\begin{align}
\label{dvp1}
 x'_{i} &=  \sum_{j\neq i } \gamma_j  \frac{ (x_{i} -  x_{j} )^\perp }{2\pi | x_{i} -  x_{j}   |^2 }   .
 \end{align}
For distinct initial positions there is a unique corresponding local-in-time smooth solution, as a direct consequence of the  Cauchy-Lipschitz theorem. Moreover this system can be reformulate as a  Hamiltonian system under the form
\begin{align}
\label{vp1}
\gamma_i x'_{i} &= 
 \nabla^\perp_{x_{i}} H(X),
\end{align}
where  $X = (x_1,\ldots, x_N)$ and 
$$H(X): =  \frac{1}{2\pi} \sum_{(i,j) / \, i\neq j} \, \gamma_i \gamma_j  \ln  | x_{i}  - x_{j}   |     .$$
As long as there is no collision between any pair of point vortices, it follows from the symmetries of the problem and from the Noether theorem that the Hamiltonian  $H$ itself and  the following quantities:
$$M:= \sum_{i} \gamma_i x_i  , \quad 
 I := \sum_{i} \gamma_i  | x_{i}    |^2, \quad   \text{ and } \quad
  T := \sum_{(i,j) / \, i\neq j } \gamma_i \gamma_j | x_{i} -  x_{j}   |^2 ,$$
 are conserved in time. 
 In particular, these invariance can be used to establish that, in the case where all the strengths 
 $\gamma_i $ have the same sign, then there is no collision at all, and the local smooth solution is actually global, see for instance  \cite{MP}. With these notations, the link with the incompressible Euler equations is made by setting  the fluid vorticity $\omega$ as 
$$\omega := \sum_{i=1}^N  \gamma_i\,\delta_{x_i}.$$

\medskip
In~\cite{Geldhauser_Romito_2020}, the authors suggest the following variant of the point-vortex model, for  $s$ in $(0,1)$, 
\begin{equation}\label{eq:evo alpha}
      x'_i:=\sum_{\substack{j=1\\j\neq i}}^2  \gamma_j\, 
       b_{2,s}
        \frac{ (x_{i} -  x_{j} )^\perp }{| x_{i} -  x_{j}   |^{3-2s} }  ,
\end{equation}
which we refer to as a system of $s$-point vortices.

Again, for distinct initial positions, this equation admits a unique solution by the Cauchy-Lipschitz theorem provided that the right-hand side of~\eqref{eq:evo alpha} remains bounded and Lipschitz, that is 
 as long as there is no collision. 
The system ~\eqref{eq:evo alpha}
also has a Hamiltonian structure. Indeed, in this case, we define the Hamiltonian of the system for $X = (x_1,\ldots, x_N)$  by
\begin{equation}\label{def:Hamiltonien}
    H(X):=\frac{1}{2}\sum_{(i,j) / \,i\neq j} \gamma_i\, \gamma_j\,
    b_{2,s} | x_i - x_j|^{2s-2} ,
\end{equation}
 recalling that $b_{2,s} | x_i - x_j|^{2s-2}$ is the fundamental solution of $(-\Delta)^s$ in $\R^2$.
Then it is a direct computation from~\eqref{eq:evo alpha} to check that
\eqref{vp1} again holds true. 
Moreover, with this Hamiltonian structure of the equation, it is possible to make use again of the Noether theorem (direct computations using the point-vortex equation~\eqref{eq:evo alpha} are also possible, see for instance~\cite{DGC,GC,GC2}) to determine the other quantities that are invariant by the dynamics: the invariance of the Hamiltonian with respect to the translations (respectively to rotations) of the plane
implies the preservation of 
\begin{equation}\label{def:M}
    M(X):=\sum_{i=1}^N  \gamma_i\,x_i \quad (\text{resp. } I(X):=\sum_{i=1}^N  \gamma_i\,|x_i|^2 .
\end{equation}
Again, this system \eqref{eq:evo alpha} can be connected with a PDE of importance in geophysical  fluid mechanics:  the surface quasi-geostrophic equations,  which models the evolution of a quasi-stratified fluid subject  evolving in a rapidly rotating frame. The transport equation of the vorticity is then the same as for the incompressible Euler equations. 
The difference lays in the Biot-Savart law that involves a fractional Laplace operator in this case. As the interaction kernel is more singular than the one of the classical point vortex and of the incompressible Euler equations, the rigorous justification of this derivation is more delicate, and to our knowledge still open. On the other hand such a justification should be possible in the case where $s >1$, which we choose to left aside in this paper. 

\medskip

On the other hand, another connection between, on the one hand the classical point vortex and the incompressible Euler equations, and on the other hand 
$s$-variant of the point-vortex model \eqref{eq:evo alpha} and the surface quasi-geostrophic equations of order $s$, is to consider the mean-field limit for which the number $N$ of point vortices goes to infinity while their strengths $\gamma_i$ is renormalized by a factor $1/N$. Such a limit is justified with strong quantitative estimates in \cite{DS}.

\medskip

 In the case where some point vortices occupy a $2$D bounded domain $ \Omega$, the dynamics is more intricate, as the interactions then depend on the boundary. 
 Indeed, even the self-interaction of a single classical point vortex is not trivial: the dynamics \eqref{vp1}  then holds with $H$ given by $H := \frac12 \mathcal R$ where we recall that  $\mathcal R$ is the Robin function of $\Omega$ associated with the classical Laplacian.
 In this case, Turkington  proved in  \cite{Turk} that this  dynamics  can be viewed as a limit of the dynamics of a concentrated vorticity patch, given as   solution to the  incompressible Euler equations in  $\Omega$. 
He also observed that the solution is global in time, and in particular that there is no collision of the vortex point with the external boundary $\partial \Omega$. This follows from the conservation of the $H$, since $H$ blows up at the boundary  $\partial \Omega$. 
 The stationary point vortices are the critical points of the Robin function.
Moreover the sum of the eigenvalues of the second derivative of the Robin function is positive, meaning that at least one of them is positive. So two main cases can occur : either both eigenvalues are positive, either one is positive and one is negative. In the first case, a point vortex initially close enough to  the critical point then remains indefinitely close to it; this is the stable situation. Otherwise there exists a point vortex evolving on a level set of $H$ going away from the the critical point; this is the unstable situation. 
Let us refer here to \cite{FG} for more on the subject.
\medskip

Now, in the case where the interaction kernel is given by the fractional Laplacian, the corresponding counterpart of what precedes for the dynamics of a single $s$-point vortex in a  $2$D bounded domain $ \Omega$ 
is given by the equation 
\eqref{vp1}  with $H$ given by $H := \frac12 \mathcal R_s^\Omega$
where we recall that  $\mathcal R_s^\Omega$ is the fractional Robin function
defined in \eqref{def-s-robin}.
Again, the stationary points correspond to the 
critical points of the Hamiltonian and their stability depends on the signature of the Hessian. 
Corollary \ref{non-deg-frac-Robin} and the genericity of  the Morse property of fractional Robin function hinted at the end of Section \ref{sec-rep-grad}  therefore yield precious informations on the dynamical stability of the steady states of such single $s$-point vortex in bounded domain. 

\subsection{Hadamard variational formulas for shape derivative of solutions}
\label{sec-HG}
\par\;

In this subsection, we examine Hadamard's formula for the variation
of the solution to fractional Dirichlet problem due to change of the domain. There are several reasons that motivate this. One such reasons could be the following. Thanks to the formula \eqref{repres-sol-Dir}, the study of the Dirichlet problem \eqref{Dir-01} is significantly simplified once the Green function of the domain is known. However, the computation of a Green function is prohibitively difficult except for some nice domains. To overcome this difficulty, one approach aims to replace the Green function (and hence the solution) of the problem with that of an easier problem which is, in some sense, nearby. One way of doing this is the variation of the underlying domain. In here we establish formulas for the first variation of the solution to the Dirichlet problem \eqref{Dir-01} when the source term $f$ is either a distribution or belongs to some Sobolev space to be precise later.  

One of the earliest result regarding the variation of solution to a boundary value problem is due to  Hadamard in the pioneering work \cite{Hadamard}, where he discovered the following celebrated identity which is today known in the literature as "Hadamard variational formula": let $\Omega_t$ be a perturbation of $\Omega$ whose boundary is given by $\partial\Omega_t=\Big\{y=x+t\alpha(x)\nu(x), \;x\in\partial\Omega\Big\}$, where $\alpha\in C^\infty(\partial\Omega)$ and $\nu$ is the outer unit normal. Then 
\begin{align}\label{Hadamard}
    \delta G(x,y):=\lim_{t\to 0}\frac{G_{\Omega_t}(x,y)-G_\Omega(x,y)}{t}=\int_{\partial\Omega}\frac{\partial G_\Omega(x,z)}{\partial{\nu_z}}\frac{\partial G_\Omega(y,z)}{\partial{\nu_z}}\alpha(z)d\sigma(z) .
\end{align}
\par\;

A rigorous proof of Hadamard's formula was given later on by  Garabedian who also considered more general perturbations. After these pioneering works, shape derivative computations has received significant interest and several extensions of Hadamard's formula for more general elliptic boundary values problems  were obtained in \cite{Oz}.
\par\;

In the fractional setting, first results on Hadamard variational types formula was obtained in \cite{DalibardVaret} by A. L. Dalibard and D. G\'erard-Varet followed by \cite{DFW}. In the later, the first author and co-authors computed, for the first time, Hadamard variational formula for the first variation of fractional Sobolev best constants in the general setting. The result of \cite{DFW} are the best known regarding shape derivative of nonlocal domain dependent functions involving the fractional Laplacian.

Our aim here, is to establish a Hadamard variational formula for the shape derivative of the fractional Green function as well as for weak solution of the Dirichlet problem \eqref{Dir-01}.
To state our results properly, some preliminaries are in order.  Let $(\phi_t)_{t\in (-1,-1)}$ with $\phi_t\in C^2(\R^N,\R^N)$ be a family of transformations such that
\begin{equation}\label{transformations}
t\mapsto \phi_t\in C^2(\R^N,\R^N)\;\text{is of class\; $C^2$}\quad\&\quad \phi_0=\textrm{id}_{\R^N}.
\end{equation}
Note that \eqref{transformations} implies that $\phi_t : \R^N \to \R^N$ is a global diffeomorphism if $|t|$ is small enough, see e.g. \cite[Chapter 4.1]{Zolesio}. We denote by $\phi_t^{-1}$ the inverse of $\phi_t$. Given $\D\subset \R^N$ such that $\Omega\Subset D$  and $h\in W^{1,\infty}(D)$, we shall write $\Omega_t:=\phi_t(\Omega)$ and let $u_t\in \mathcal H^s_0(\Omega_t)$ the weak solution of the Dirichlet problem:
\begin{equation}\label{purturbed-Dirich-pb-0}
\left\{ \begin{array}{rcll} (-\Delta)^s u_t&=& h  &\textrm{in }\;\;\Omega_t \\ u_t&=&0&
\textrm{in }\;\;\R^N\setminus\Omega_t. \end{array}\right. 
\end{equation}
For $t=0$, we simply write $u$ at the place of $u_0$. By a standard argument based on the implicit function theorem, one proves that the map $t\mapsto v_t:=u_t\circ \phi_t\in \mathcal H^s_0(\Omega)$ is differentiable at $0$ (see e.g \cite[Lemma 3.1]{DalibardVaret} for the case $s=1/2$ and \cite{HenrotPierre} for $s=1$). By regularity, the function 
\begin{equation}\label{shape-deriv}
u':=v'-\nabla u\cdot Y\in L^1(\Omega)\end{equation} 
is well defined for every $x\in\Omega$ and even continuous. Here $v'=\partial_t{v_t}\big|_{t=0}\in \mathcal H^s_0(\Omega)$ and $Y=\partial_t{\phi_t}\big|_{t=0}\in C^1(\R^N,\R^N)$.

For the reader's convenience we recall from \cite{Sokolowski} the following definition.
\begin{Definition}
    The function $u'$ defined in \eqref{shape-deriv} is refer to as the shape derivative of the solution $u$ to 
    the Dirichlet problem:
\begin{equation}
 (-\Delta)^s u= h  \quad \textrm{in }\;\;\Omega , \\ \quad u=0 \quad  
\textrm{in }\;\;\R^N\setminus\Omega.  ,
\end{equation}
in the direction of the vector field $Y$.
\end{Definition}
The shape derivative measures the variation of the solution $u$ to the 
    the Dirichlet problem  due to change of the domain. Note that the shape derivative depends on the family of transformations $\phi_t$
     only through the vector field $Y$.
 In particular, for a given  family of transformations $\phi_t$, the shape derivative $u'$ appears as the first order correction in the following expansion: 
 for all $x\in K\subset\Omega$, where $K$ is compact, 
  $$
  u_t(x)=u(x)+t\Bigg(\Gamma^2(1+s)\int_{\partial\Omega}\frac{G_s(x,\cdot)}{\delta^s_\Omega}(\sigma)\frac{u}{\delta^s_\Omega}(\sigma)Y\cdot\nu(\sigma)d\sigma\Bigg)+o(t)\quad\text{as $t\to 0$.}
  $$

Recall that $u\in \mathcal H^s_0(\Omega)$ is the weak solution of the problem 
\begin{equation}
\left\{ \begin{array}{rcll} (-\Delta)^s u&=& h  &\textrm{in }\;\;\Omega \\ u&=&0&
\textrm{in }\;\;\R^N\setminus\Omega. \end{array}\right. 
\end{equation}
in the sense that $\mathcal E_s(u,v)=\int_{\Omega}h(y)v(y)dy$ for all $v\in \mathcal H^s_0(\Omega)$  and  $Y=\partial_t{\phi_t}\big|_{t=0}$ where $\phi_t$ is the family of transformations given in \eqref{transformations}. Our third main result reads as follows.
\begin{thm}[Hadamard formula for shape derivative of weak solutions]\label{thm.3}
  Let $N\in \N_*$ with $N\geq 2s$ and let $\Omega$ be a bounded open set of $\R^N$ of class $C^{1,1}$. Let $u'$ be defined as in \eqref{shape-deriv}. Then for all $x\in\Omega$ we have 
  \begin{equation}\label{repres-shape-deriv}
      u'(x)=\Gamma^2(1+s)\int_{\partial\Omega}\frac{G_s(x,\cdot)}{\delta^s_\Omega}(\sigma)\frac{u}{\delta^s_\Omega}(\sigma)Y\cdot\nu(\sigma)d\sigma,
  \end{equation}
  where $\nu$ denotes the outward unit normal to the boundary. 
\end{thm}
\par \;
Our next result is the extension of the Hadamard formula for the Green function \eqref{Hadamard} to the fractional setting. It reads as follows.
\begin{thm}[Hadamard formula for the fractional Green function]\label{thm.5}Let $s\in (0,1/2)$ and let $\Omega$ be a bounded open set of class $C^{1,1}$ of $\R^N$ ($N\geq 2s$). Let $G_{\Omega_t}\big(\cdot,\cdot\big)$ be the Green function of $\Omega_t:=\phi_t(\Omega)$. For a fixed $x\in\Omega$, consider the real-valued function $A^x_t: \Omega\to \R, y\mapsto G_{\Omega_t}(\phi_t(x),\phi_t(y)) $. Assume that the map $t\mapsto A_{t,x}$ is differentiable at $0$ with 
\begin{align}
(H1)\qquad\partial_tA_{t}^x{\big|_{t=0}}\in \mathcal L^1_s(\R^N)=\Big\{u:\R^N\to\R, \int_{\R^N}\frac{|u(x)|}{1+|x|^{N+2s}}dx<\infty\Big\}.
\end{align}
Moreover, we suppose that \begin{align}
    (H2)\qquad\qquad\partial_t{H_{\Omega_t}(\phi_t(x),\phi_t(\cdot))}\big|_{t=0}\in C(\Omega).
\end{align} Let $y\in\Omega$ with $x\neq y$ and defined $$\delta G_s(x,y):=\Big(\partial_tA_{t}^x{\big|_{t=0}}\Big)(y)-\nabla_x G_s(x,y)\cdot Y(x)-\nabla_y G_s(x,y)\cdot Y(y).
$$
Under the assumptions $(H1)$ and $(H2)$, there holds: 
\begin{align}\label{var-green}
    \delta G_s(x,y)=\Gamma^2(1+s)\int_{\partial\Omega}\frac{G_s(x,\cdot)}{\delta^s_\Omega}\frac{G_s(y,\cdot)}{\delta^s_\Omega}Y\cdot\nu \;d\sigma.
\end{align}
In particular if $x,y\in K\subset \Omega$ with $x\neq y$ and $K$ is compact. Then the map $t\mapsto G_{\Omega_t}(x,y)$ is differentiable at $0$ and 
\begin{align} \label{var-green-bis}
    \partial_t G_{\Omega_t}(x,y){\Big|_{t=0}}=\Gamma^2(1+s)\int_{\partial\Omega}\frac{G_s(x,\cdot)}{\delta^s_\Omega}\frac{G_s(y,\cdot)}{\delta^s_\Omega}Y\cdot\nu \;d\sigma.
\end{align}
Furthermore, we have
\begin{align}\label{var-Robin-func}
\partial_t{\mathcal R_s^{\Omega_t}(\phi_t(x))}\Big|_{t=0}-\nabla \mathcal R_s^\Omega(x)\cdot Y(x)=-\Gamma^2(1+s)\int_{\partial\Omega}\Bigg(\frac{G_s(x,\cdot)}{\delta^s_\Omega}\Bigg)^2Y\cdot\nu \;d\sigma
\end{align}
where $\mathcal R_s^\Omega(x)=H_\Omega(x,x)$ and $H_\Omega(\cdot,\cdot)$ is the regular part of the Green function as defined in the beginning of subsection \ref{sec-rep-grad}
\end{thm}
\begin{remark}
Let $f\in C^\infty_c(\Omega)$ and let $u$ be the weak solution of $$
(-\Delta)^su=f\quad\text{in}\;\;\; \Omega\quad\&\quad u=0\;\;\text{in $\R^N\setminus\Omega$} .
$$
Let $J:\Omega \mapsto J(\Omega)\in \R$ be the energy functional associated to the equation above. That is to say $$J(\Omega)=\frac{c_{N,s}}{2}\iint_{\R^N\times\R^N}\frac{(u(x)-u(y))^2}{|x-y|^{N+2s}}dxdy=\int_{\Omega}f(y)u(y)dy.$$ 
A direct computation gives 
\begin{align}
      \partial_t J(\Omega_t){\Big|_{t=0}}&=\partial_t {\Big|_{t=0}} \int_{\Omega_t}u_t(y)f(y)dy=\int_{\Omega}u'(y)f(y)dy,\nonumber\\
      &=\Gamma^2(1+s)\int_{\Omega}f(y)\Bigg[\int_{\partial\Omega}\frac{G_s(y,z)}{\delta^s_\Omega(z)}\frac{u(z)}{\delta^s_\Omega(z)}Y\cdot\nu(z)d\sigma(z)\Bigg]dy
\end{align}
where in the last line we have used \eqref{repres-shape-deriv}. By \cite[Lemma 3.1.3]{Abatangelo} with $h=(u/\delta^s)Y\cdot\nu\in C(\partial\Omega)$,
\begin{align}
      \partial_t J(\Omega_t){\Big|_{t=0}}
      &=\Gamma^2(1+s)\int_{\Omega}f(y)\Bigg[\int_{\partial\Omega}\frac{G_s(y,z)}{\delta^s_\Omega(z)}\frac{u(z)}{\delta^s_\Omega(z)}Y\cdot\nu(z)d\sigma(z)\Bigg]dy\nonumber\\
      &=\Gamma^2(1+s)\int_{\partial\Omega}(u/\delta^s)^2 Y\cdot \nu\,d\sigma.
\end{align}
This gives an alternative way of deriving the fractional Hadamard formula proved in \cite{DFW} (see also \cite{DalibardVaret}).
\end{remark}

\subsection{Reproducing kernel Hilbert space}
\label{sec-rkhs}

As another application of representation formulas of solutions to fractional Laplace equations,  in this section, we address the issue of representing the kernel of the reproducing kernel Hilbert space of  $s$-harmonic functions by means of a boundary integral  formula, what extends the analysis initiated by  J.L. Lions, see \cite{Lions}, in the case of  the harmonic functions.
Let us first recall that a reproducing kernel Hilbert space (RKHS) is a Hilbert space of functions in which pointwise evaluations are continuous linear functionals. It then follows from Riesz' theorem that each of these functionals can be represented as an inner product with an element of this Hilbert space. One then defines a two-points kernel by considering the inner products of any pair of such elements. 

\begin{Definition}
We say that a space $H$ is RKHS as a Hilbert space $H$ of  real-valued functions defined on an open domain $\Omega \subset \R^N$ such that 
 for any $x$ in $\Omega$, the evaluation mapping $E_x$ of functions of $H$  at $x$, that is the  mapping defined for any $f$ in $H$ by $E_x (f) = f(x)$, 
is linear continuous. Then, it follows from Riesz' theorem that there exists $K_x$ in $H$ such that for any $g$ in $H$,  $E_x (f) = <f,g>$, the inner product of $f$ and $g$ in $H$. Then the  two-points kernel is defined for any pair of $x$ and $y$ in $\Omega$  by $K(x,y)= <K_x,K_y> $. It follows from the properties of the inner product that the kernel $K$ is symmetric and positive definite. 
\end{Definition}

This notion was first introduced in 1907 by Stanislaw Zaremba for boundary value problems for harmonic and biharmonic functions, and simultaneously by James Mercer in the theory of integral equations, before to be more systematically tackled by Nachman Aronszajn and Stefan Bergman. These spaces have various applications in complex analysis, harmonic analysis, quantum mechanics and statistical learning theory. Let us here only mention the famous Moore–Aronszajn theorem, see \cite{Aronszajn}, which states that, conversely, every symmetric, positive definite kernel defines a unique reproducing kernel Hilbert space. 
\medskip 

One of the last theorems by Jacques-Louis Lions was about a formula for the reproducing kernel of some function spaces associated with second order elliptic boundary value problems, see \cite{Lions}. Lions used a variational approach and obtained his formula through a penalization limit. 
One striking example  considered in \cite{Lions} is the space $\mathcal H$ of the distributions $u$ satisfying $\Delta u=0$ in $\Omega$ and such that the trace of $u$ on $\partial \Omega$ is in the Lebesgue space  $L^2 (\partial \Omega)$. Observe that the definition of the trace is already not trivial, in particular for domains which have only poor regularity. We restrict here to nice domains for simplicity. 
This space $\mathcal H$ is endowed with the inner product: 
 $$<u,v> :=  \int_{\partial \Omega} u(z)v(z) d\sigma(z).$$ 
 Then, Lions proved the following result, see \cite{Lions}.
\begin{thm} \label{Lions-theo}
The space $H$ is RKHS and its  two-points kernel is given,  for any pair of $x$ and $y$ in $\Omega$,  by 
 \begin{equation} \label{l-formu}
 K(x,y)=  \int_{\partial\Omega} (\partial_\nu G_1 (x,\cdot))( \partial_\nu G_1 (y,\cdot) )\,d\sigma .
 \end{equation}
\end{thm}
Above, we recall that $G_1(x,\cdot)$ is the Green function of the classical Laplacian with singularity at $x$.
Indeed,  another, more direct, approach to Theorem \ref{Lions-theo}, was proposed later in \cite{ELPL}, together with some extensions. We reproduce here their proof of \eqref{l-formu} as a preparation for our extension to the case of fractional Laplace operators. Let us highlight that \eqref{l-formu} is the most simple form of Lions formula, which holds for more general topology and operators, as can be seen in  \cite{Lions} and \cite{ELPL}.
\begin{proof}
First, a consequence of the  representation formula 
\eqref{classical-repres-formula} is the  Poisson formula: for any $u$ in  $\mathcal H$, for any $x$ in $\Omega$, 
$u(x)=<u, \partial_\nu G_1 (x,\cdot) > $.
Because for any $x$ in $\Omega$, that is away from the boundary, the trace of $\partial_\nu G_1 (x,\cdot) $ on $\partial\Omega$ is in $L^2 (\partial \Omega)$, 
 the evaluation mapping $E_x$ from $\mathcal H$ to $\R$ which maps $f$ to $E_x (f) = f(x)$
is linear continuous and  $\mathcal H$ is a Hilbert space.
Then it follows from Riesz' theorem that there exists $K_x$ in $H$ such that for any $u$ in $H$, 
$u(x) = < u , K_x >$.
 Then by identification, 
 we conclude that 
 $K_x = \partial_\nu G_1 (x,\cdot)$ and the formula \eqref{l-formu} follows.
\end{proof}

An interesting observation in \cite{ELPL} is the resemblance of \eqref{l-formu} with the Hadamard variation formula \eqref{Hadamard} for the Green function of the classical Laplacian. 
In \cite{ELPL}, it is said that the authors  "are of the opinion that the connection between Hadamard’s varation formula and the reproducing kernel just indicated, in the classical harmonic case, seems to be an isolated phenomenon peculiar to the second order case." 
In this section, we establish that actually, a similar connection between Lions formula for RKHS and Hadamard’s varation formula 
occurs as well for  fractional Laplace operators. 
Indeed, let us define the space $\mathcal H_s$ of the distributions $u$ satisfying 
$(-\Delta)^s u=0  $ in $\Omega$ and such that $\frac{u}{\delta^{s-1}_\Omega}$ is in $L^2 (\partial \Omega)$. 
Then this space $\mathcal H_s$ is endowed with the inner product: 
 $$<u,v>_s :=  \int_{\partial \Omega}   \frac{u}{\delta^{s-1}_\Omega}  \frac{v}{\delta^{s-1}_\Omega}  d\sigma.$$

\begin{thm} \label{s-Lions-theo}
The space $\mathcal H_s$ is RKHS and its  two-points kernel, which we still denote $K$, is given,  for any pair of $x$ and $y$ in $\Omega$,  by 
 \begin{equation} \label{l-formu-s}
 K(x,y)=  \Gamma^2(s) \Gamma^2(s+1)  \int_{\partial \Omega}   \frac{G_s(x,\cdot)}{\delta^{s}_\Omega}  \frac{G_s(y,\cdot)}{\delta^{s}_\Omega}  d\sigma.
 \end{equation}
\end{thm}
The proof of Theorem \ref{s-Lions-theo} is greatly simplified by some recent progresses on the the nonhomogeneous boundary value problem for the fractional laplacian.
Actually, it has been recently studied through Fourier technics by Grubb in the series of paper \cite{Grubb-2014,Grubb-2015,Grubb-2016,Grubb-2017,Grubb-2020,Grubb-2021-en,Grubb-2022} with, in particular, the following outcomes: 
the nonhomogeneous Dirichlet problem:
\begin{equation}\label{eq-intro}
\left\{ \begin{array}{rcll} 
(-\Delta)^s u&=&0  &\textrm{in }\Omega \\ 
 \frac{u}{\delta_\Omega^{s-1} }&=&g& \textrm{on } \partial \Omega, \end{array}\right. 
\end{equation}
is well-posed in some appropriate Sobolev-type spaces (at least when  the integrability index is $p=2$ otherwise the natural trace spaces are rather Besov spaces). Also, Grubb has obtained an integration by parts formula for functions  $v$ and $w$ satisfying some nonhomogeneous boundary conditions in the sense that
 $\frac{v}{\delta_\Omega^{s-1}}$ and  $\frac{w}{\delta_\Omega^{s-1}}$ are not zero on $ \partial \Omega$. In the case where only one of them is not zero, that is the case which we need below, say $w$, this formula reads:
\begin{equation}
\int_\Omega  v(-\Delta)^sw\,dz =
\int_\Omega  w(-\Delta)^sv\,dz
 - \Gamma(s) \Gamma(s+1) \int_{\partial\Omega}\frac{v}{\delta^{s}_\Omega}\frac{w}{\delta^{s-1}_\Omega} \,d\sigma.
\end{equation}
We can now embark on the proof of Theorem \ref{s-Lions-theo}.
\begin{proof}
By applying the latter formula to the Green's functions  $G_s(x,\cdot)$, we obtain for any function $u$ in the space $\mathcal H_s$, the representation formula:  for any $x$ in $\Omega$,
\begin{equation}\label{eq:po1s}
u(x) =  < u, \Gamma(s) \Gamma(s+1) \frac{G_s(x,\cdot)}{\delta_\Omega}>_s .
\end{equation}
Since for any $x$ in $\Omega$, the function $\frac{G_s(x,\cdot)}{\delta^{s}_\Omega}$ is in $L^2 (\partial \Omega)$, we deduce that  the space  $\mathcal H_s$ is  a Hilbert space and that the evaluation mappings 
 $E_x$ from $\mathcal H$ to $\R$ which maps $f$ to $E_x (f) = f(x)$
are linear continuous. Then it follows from Riesz' theorem that  for any $x$ in $\Omega$, there exists $K_x$ in $\mathcal H_s$ such that for any $u$ in $\mathcal H_s$, $u(x) =<u,K_x>_s$.
 Then by identification,
 we conclude that $K_x = \Gamma(s) \Gamma(s+1) \frac{G_s(x,\cdot)}{\delta_\Omega}$, hence the formula  \eqref{l-formu-s} for the  $2$-points kernel $K$.
\end{proof}

Thus let us stress that  the formula \eqref{l-formu} is  very similar to the Hadamard variation formula for the Green function $G_s$, see \eqref{var-green-bis}, proving that the connection in \cite{ELPL} 
 between the classical Hadamard variation formula and the reproducing kernel associated with the classical Laplace operator and for a large class of second order elliptic systems, also holds true for some elliptic operators of fractional order $2s$, invalidating the conjecture made in  \cite{ELPL}.

One may also easily check that 
 the formula  \eqref{l-formu-s} is consistent with the classical Lions formula  \eqref{l-formu} for the Laplace operator, since as $s \rightarrow 1^-$, 
 $\frac{G_s(y,\cdot)}{\delta^{s}_\Omega} \rightarrow \partial_\nu G(y,\cdot)$ and 
 $\Gamma^2(s) \Gamma^2(s+1) \rightarrow \Gamma(1)^2 \Gamma(2)^2 = 1$.

Let us mention that it could be useful for various applications to determine the boundary behaviour of the reproducing kernel thanks, what could, perhaps, be done by using the knowledge of  the boundary behaviour of the Green function $G_s$ and the fractional Lions formula \eqref{l-formu-s}.

\begin{remark}
As a complementary remark regarding the similarity, or not, of Lions' type formula for two-points kernel of  RKHS with the Hadamard variation formula, let us highlight that there are some second order systems for which the two formula do not like quite the same. Indeed, consider the space $\mathfrak H$ of the traces on $\Omega$ of solutions to the 
the steady Stokes problem:
\begin{equation} \label{eq_stoke}
\left\{
\begin{array}{rcl}
- \Delta u + \nabla p &=& 0 \, \\
\operatorname{div} u &= & 0 \,.
\end{array}
\right.
\quad \text{ on $\Omega$} ,
\end{equation}
endowed with 
$$<u,v> :=  \int_{\partial \Omega}  u \cdot v  \, d\sigma.$$
We  recall that  the system  \eqref{eq_stoke} aims  at describing steady, incompressible fluids with zero-Reynolds number. %
In particular the second equation encodes the incompressibility of the fluid velocity field $u$, which is vector-valued, and this incompressibility constraint  translates into the presence  of the gradient of the fluid pressure in the first equation, the fluid pressure $p$ being itself scalar-valued.  We refer to \cite{Galdi}  for more.
\begin{equation} \label{eq_Newt2}
\Sigma(u,p) = 2 D(u) - p \mathbb{I}_3 \quad   \text{ where } 2 D(u):= \nabla u + (\nabla u)^T . 
\end{equation}
We also recall the following classical Poisson formula: for any solution $(u,p)$ of the steady Stokes problem \eqref{eq_stoke} we have 
\begin{equation} \label{eq_Green-St}
u(x)
= \int_{\partial \Omega} ( \Sigma (G(x,\cdot),P(x,\cdot))  n) \cdot u \,d\sigma.
\end{equation}
where $(G(x,\cdot),P(x,\cdot))$ is the Green function associated with 
 the Stokes system and Dirac mass at position $x$, in the first equation of \eqref{eq_stoke}.
 A reasoning similar to the ones above leads to the conclusion that the space $\mathfrak H$ is a RHKS with the following
 Lions type formula for its $2$-points kernel $K$: 
$$K(x,y) =   \int_{\partial \Omega}  \Sigma(G(x,\cdot),P(x,\cdot))n \cdot \Sigma(G(y,\cdot),P(y,\cdot))n d\sigma. $$

 On the other hand, the Hadamard variation formula reads, see \cite{Simon,Ushikoshi}: 
 $$\delta G(x,y) =   \int_{\partial \Omega}  \partial_n G(x,\cdot) \cdot \partial_n G(y,\cdot) n d\sigma. $$
\end{remark}


\subsection{Organization of the rest of the paper}

The paper is organised as follow: in Section \ref{sec4.1} we give the proof of Theorem \ref{main-result}. Section \ref{sec4}, Section  \ref{sec-robin}, Section \ref{sec-thm4} and Section  \ref{HGatlast} are respectively devoted to the proofs of Theorem  \ref{Thm.2}, of  Corollary \ref{non-deg-frac-Robin}, of Theorem \ref{thm.3} and of  Theorem \ref{thm.5}.
 In the appendix we collect some technical lemmas that are used in the proof of the main results.

\section{Proof of Theorem \ref{main-result}}\label{sec4.1}

This section is devoted to the proof of Theorem \ref{main-result}.
In Section \ref{sec-scheme} we give the scheme of the proof of Theorem \ref{main-result}.
Two intermediate results are necessary in the course of the proof: Lemma \ref{Lem4.2} and Lemma \ref{Lem4.3} which are respectively proved in Section \ref{sec-pr-l22} and in Section \ref{sec8}.

\subsection{Scheme of the proof}\label{sec-scheme}

The starting point of the proof of Theorem \ref{main-result} is the following identity:
\begin{equation}\label{eq:pohozaev.1}
\int_\Omega \frac{\partial v}{\partial z_i}(-\Delta)^sw\,dz =
-\int_\Omega \frac{\partial w}{\partial z_i}(-\Delta)^sv\,dz-\Gamma^2(1+s)\int_{\partial\Omega}\frac{v}{\delta^{s}_\Omega}\frac{w}{\delta^{s}_\Omega}\,\nu_i\,d\sigma.
\end{equation}
which has been established for regular enough functions $v$
and $w$ in \cite[Theorem 1.3]{DFW2023} and \cite[Theorem 1.9]{RO-S}.
\par \,

However,  for any $x\in \Omega$, the function $G_s(x,\cdot)$  is too irregular to be admissible in \eqref{eq:pohozaev.1} in both these results.
To overcome this difficulty, we approximate $G_s(x,\cdot)$ by a $C^\infty_c(\Omega)$-function of the form $\eta_k \psi_{\mu,x}G_s(x,\cdot)$ where $\eta_k$  and $\psi_{\mu,x}$ are suitable cut-off functions which vanish near the boundary $\partial\Omega$ and near the singular point $x$ of the Green function respectively.
More precisely, let $\rho\in C^\infty_c(-2,2)$ such that $0\leq \rho\leq 1$ and $\rho=1$ in $(-1,1)$.
We have the following result regarding a double integral which involves the  radial cut-off associated with the function $\rho$. 
This  lemma is used below in the proof of Lemma \ref{Lem4.2} and of Lemma \ref{lem-lim-G-x-y}. Its proof is given in Section \ref{ser-rad}.
\begin{Lemma}\label{Lem6.1}
Let  $\rho\in C^\infty_c(-2,+2)$ such that $\rho\equiv 1$ in $(-1,+1)$. Then for all $s\in (0, 1)$ we have 

$$
\iint_{\R^N\times \R^N}\frac{\big(\rho(|y|^2)-\rho(|z|^2)\big)\big(|z|^{2s-N}-|y|^{2s-N}\big)}{|z-y|^{N+2s}}dydz<\infty.
$$
Moreover,
\begin{equation}\label{sec7-15}
b_{N,s}c_{N,s}\iint_{\R^N\times \R^N}\frac{\big(\rho(|y|^2)-\rho(|z|^2)\big)\big(|z|^{2s-N}-|y|^{2s-N}\big)}{|z-y|^{N+2s}}dydz
=-2.
\end{equation}
\end{Lemma}
Now, fix $\delta\in C^{1,1}(\R^N)$ such that $\delta$ coincides with the signed distance function $\delta_\Omega(x)=\dist(x,\R^N\setminus\Omega)-\dist(x,\Omega)$ near the boundary of $\Omega$. Moreover, we assume that $\delta$ is positive in $\Omega$ and negative in $\R^N\setminus\Omega$. Next for any $k\in \mathbb N^*$ we define 
\begin{equation}\label{eta-k}
    \eta_k: \R^N\to \R,\; y\mapsto \eta_k(y)=1-\rho(k\delta(y)) .
\end{equation}
Let us also from now on fixed $x\in \Omega$ , and for any $\mu\in(0,1)$, we also define  $\psi_{\mu,x}$ by 
\begin{equation}\label{psi-mu}
     \psi_{\mu,x}: \R^N\to \R,\; y\mapsto \psi_{\mu,x}(y)=1-\rho\Big(\frac{8}{\delta^2_\Omega(x)}\frac{|x-y|^2}{\mu^2}\Big).
\end{equation}
Note that thanks to the normalising constant $8/\delta^2_\Omega(x)$, the function
\begin{equation}\label{rho-mu}
\rho_{\mu,x}(\cdot):=\rho\Big(\frac{8}{\delta^2_\Omega(x)}\frac{|x-\cdot|^2}{\mu^2}\Big),
\end{equation}
is $C^\infty$ \underline{with} compact support in $B_{\frac{\delta_\Omega(x)}{2}}(x)\Subset\Omega$. This will be important to ensure that $\psi_{\mu,x}(\cdot)G_s(x,\cdot)$ satisfies the hypothesis of Lemma \ref{DFW} in further below. The following lemma contains some important properties of the cut-off function $\psi_{\mu,x}$. These properties will be used repeatedly throughout this manuscript.

\begin{Lemma}\label{lem-ps-mu-x}Let $\psi_{\mu,x}$ be defined as above. Then the following properties holds true:
\begin{itemize}
    \item For all $z\in\R^N$ we have 
    \begin{equation}\label{rescapsi}
    (-\Delta)^s\psi_{\mu,x}(z)=-\widetilde\mu^{-2s}(-\Delta)^s\big(\rho\circ |\cdot|^2\big)\big(\frac{z-x}{\widetilde\mu}\big)\;\;\text{where\;\;  $\widetilde\mu:=\frac{\delta_\Omega(x)}{2\sqrt{2}}\mu$}.
\end{equation}
    \item Let $G_s(x,\cdot)$ be the Green with singularity at $x$. Then $$(-\Delta)^s\big(\psi_{\mu,x}(\cdot)G_s(x,\cdot)\big)\in L^\infty(\Omega).$$ 
    \item Let $\mathcal I_s[\cdot,\cdot]$ be the bilinear form defined in \eqref{def-I-s} and let $\eta_k$ be given by \eqref{eta-k}. Then,  for all $\varepsilon>0$, we have 
\begin{gather} \label{psi-mu-equ}
    \lim_{\mu\to 0^+}
    \lim_{k\to\infty}\int_{\Omega}\eta_k^2(z)\partial_iu(z)\Bigg[G_s(x,\cdot)(-\Delta)^s\psi_{\mu,x}-\mathcal I_s\big[\psi_{\mu,x},G_s(x,\cdot)\big]\Bigg]dz 
    \\ \nonumber   =\lim_{\mu\to 0^+}\lim_{k\to\infty}\int_{B_\varepsilon(x)}\eta_k^2(z)\partial_i u(z)\Bigg[G_s(x,\cdot)(-\Delta)^s\psi_{\mu,x}-\mathcal I_s\big[\psi_{\mu,x},G_s(x,\cdot)\big]\Bigg]dz.
\end{gather}
\end{itemize}
Here we used the notation $\partial_i=\frac{\partial}{\partial z_i}$.
\end{Lemma}

\begin{proof}
The first item is a simple consequence of the integral definition of the operator $(-\Delta)^s$. To check the second item, we write 
 $$
 \psi_{\mu,x}(z)G_s(x,z)=b_{N,s} \frac{\psi_{\mu,x}(z)}{|x-z|^{N-2s}}-H_s(x,z)-\rho_{\mu,x}(z)H_s(x,z)
 $$
 where we recall the notation \eqref{rho-mu}. It is clear that $\frac{\psi_{\mu,x}(z)}{|x-z|^{N-2s}}\in C^\infty(\R^N)\cap L^\infty(\R^N)$ and therefore $(-\Delta)^s\big(\psi_{\mu,x}(\cdot)|x-\cdot|^{-N-2s}\big)\in L^\infty(\Omega)$. We also have $(-\Delta)^s H_s(x,\cdot)=0$ in $\Omega$ by definition. Finally, since $\rho_{\mu,x}(\cdot)H_s(x,\cdot)\in C^\infty_c(B_{\frac{\delta(x)}{2}}(x))$, one easily check that $(-\Delta)^s\big(\rho_{\mu,x}(\cdot)H_s(x,\cdot)\big)\in L^\infty(\Omega)$. To check the last item, we expand and examine, separately, the two terms corresponding to the integrals on the complement set $\Omega\setminus B_\varepsilon(x)$.
First, 
 if $|x-z|>\varepsilon$ for some $\varepsilon>0$, then for $\mu>0$ sufficiently small, we have $\psi_{\mu,x}(z)=1$. Thus for $\mu>0$ sufficiently small, we get 
\begin{align}\label{Eq-Delta-s-psi-1}
  \frac{1}{c_{N,s}}\big|(-\Delta)^s\psi_{\mu,x}&(z)\big|=\pv\int_{\R^N}\frac{1-\psi_{\mu,x}(y)}{|z-y|^{N+2s}}dy=\int_{B_{\widetilde\mu}(x)}\frac{\rho\Big(\frac{8}{\delta^2_\Omega(x)}\frac{|x-y|^2}{\mu^2}\Big)}{|z-y|^{N+2s}}dy\nonumber\\
  &\leq \int_{B_{\varepsilon/4}(x)}\frac{\rho\Big(\frac{8}{\delta^2_\Omega(x)}\frac{|x-y|^2}{\mu^2}\Big)}{|x-y|^{N+2s}}dy\nonumber\\
  &\leq C(\varepsilon)\int_{B_{\varepsilon/4}(x)}\rho\Big(\frac{8}{\delta^2_\Omega(x)}\frac{|x-y|^2}{\mu^2}\Big)dy\leq C(\varepsilon).
\end{align}
where $\widetilde\mu=\frac{\delta_\Omega(x)}{2}\mu$. Moreover, by \eqref{Eq-Delta-s-psi-1} we have $\lim_{\mu\to 0}(-\Delta)^s\psi_{\mu,x}(z)=0$. Next, since $G_s(x,\cdot)\partial_iu\in L^1(\Omega\setminus B_\varepsilon(x))$, we deduce by the dominated convergence theorem that 
\begin{align}\label{sec7-3}
    \lim_{\mu\to 0}\lim_{k\to\infty}\int_{\Omega\setminus B_\varepsilon(x)}\eta_k^2G_s(x,\cdot)\partial_iu(-\Delta)^s\psi_{\mu,x}dz=0.
\end{align}

Similarly, if $|x-z|>\varepsilon$, then for $\mu>0$ sufficiently small we have 
\begin{align*}\label{sec7-4}
    \big|\mathcal I_s[\psi_{\mu,x},G_s(x,\cdot)](z)\big|&=\Big|\pv\int_{\R^N}\frac{(1-\psi_{\mu,x}(y))(G_s(x,z)-G_s(x,y))}{|y-z|^{N+2s}}dy\Big|\nonumber\\
    &=\Big|\int_{B_{\widetilde\mu}(x)}\frac{\rho\Big(\frac{8}{\delta^2_\Omega(x)}\frac{|x-y|^2}{\mu^2}\Big)\Big(G_s(x,z)-G_s(x,y)\Big)}{|y-z|^{N+2s}}dy\Big|\nonumber\\
    &\leq C(\varepsilon)\int_{B_{\varepsilon/4}(x)}\rho\Big(\frac{8}{\delta^2_\Omega(x)}\frac{|x-y|^2}{\mu^2}\Big)\Big(1+|G_s(x,y)|\Big)dy\leq C(\varepsilon),
\end{align*}
and $\mathcal I_s[\psi_{\mu,x},G_s(x,\cdot)](z)\to 0$ as $\mu\to 0^+$. Since $\partial_iu\in L^1(\Omega)$, we deduce by the dominated convergence theorem that
\begin{equation}\label{sec7-5}
    \lim_{\mu\to 0^+}\lim_{k\to\infty}\int_{\Omega\setminus B_\varepsilon(x)}\eta_k^2(z)\partial_i u(z)\mathcal I_s[\psi_{\mu,x},G_s(x,\cdot)](z)dz=0.
\end{equation}
Claim \eqref{sec7-2} follows by combining  \eqref{sec7-3} and \eqref{sec7-5}.
\end{proof}

From now on, let also the integer $i$ be fixed equal to $1$ or $2$ or $\cdots $ or $N$.
 For any $k\in \mathbb N^*$, 
we use \eqref{eq:pohozaev.1} with
$\eta_k u\in C^1_c(\Omega)$ and $\eta_k\psi_{\mu,x}G_s(x,\cdot)\in C^\infty_c(\Omega)$ instead of $v$ and $w$; we arrive at the identity: 
\begin{align}\label{key-id}
     \int_{\Omega}\partial_i(\eta_k u)(-\Delta)^s\Big(\eta_k\psi_{\mu,x}G_s(x,\cdot)\Big)dz=-\int_{\Omega}\partial_i\big(\eta_k \psi_{\mu,x}G_s(x,\cdot)\big)(-\Delta)^s(\eta_ku)\,dz,
 \end{align}
 where we write $\partial_i=\frac{\partial }{\partial z_i}$.    Let us highlight that there is no boundary term since $\eta_k u$ is compactly supported.  Let us now recall the fractional product rule: 
 
 \begin{equation}\label{pl}
(-\Delta)^s (uv) =  v(-\Delta)^s u + u (-\Delta)^s v  -\mathcal I_s [u,v],
\end{equation}
where
\begin{align}\label{def-I-s}
\mathcal I_s [u,v] :=  c_{N,s} \, \pv\int_{\R^N} \frac{(u(\cdot)-u(y)) (v(\cdot)-v(y))}{|\cdot-y|^{N+2s}}dy .
\end{align}
 On the one hand we apply the product law \eqref{pl} to 
 $u= \eta_k$ and $v=G_s(x,\cdot)\psi_{\mu,x} $ to obtain that 
 the left hand side of \eqref{key-id} also reads as
 %
 \begin{align}
     &\int_{\Omega}\partial_i(\eta_k u)(z)(-\Delta)^s\Big(\eta_k\psi_{\mu,x}G_s(x,\cdot)\Big)(z)dz\nonumber\\
     &=\int_{\Omega}\partial_i (\eta_ku)(z)\Bigg(\psi_{\mu,x}G_s(x,\cdot)(-\Delta)^s\eta_k-\mathcal I_s\big[\eta_k, G_s(x,\cdot)\psi_{\mu,x}\big]\Bigg)dz\nonumber\\
     &\quad+\int_{\Omega}\eta_k(z)\partial_i (\eta_ku)(z)(-\Delta)^s\Big(\psi_{\mu,x}G_s(x,\cdot)\Big)dz . \label{rhs1}
 \end{align}
 On the other hand, we  apply the product law \eqref{pl} to 
 $u$ and $v= \eta_k$ to obtain that 
 the opposite of the right hand side of \eqref{key-id} also reads as
 \begin{align}
     \int_{\Omega}\partial_i\Big(\eta_k \psi_{\mu,x}G_s(x,\cdot)\Big)(-\Delta)^s(\eta_ku)\,dz\nonumber
     &=\int_{\Omega}\partial_i\Big(\eta_k \psi_{\mu,x}G_s(x,\cdot)\Big)\Big(u(-\Delta)^s\eta_k-\mathcal I_s(u,\eta_k)\Big)dz\nonumber\\
     &\quad+\int_{\Omega}\eta_k\partial_i\Big(\eta_k \psi_{\mu,x}G_s(x,\cdot)\Big)(-\Delta)^su(z)dz . \label{rhs2}
 \end{align}

Let us set for any $\mu\in(0,1)$,  for any $k\in \mathbb N*$, 
\begin{align*}
 \mathcal A_{k,\mu}(x) &:= \int_{\Omega}\eta_k(z)\partial_i\big(\eta_k u\big)(z)(-\Delta)^s\Big(\psi_{\mu,x}G_s(x,\cdot)\Big)(z)dz, 
\\ \mathcal  B_{k,\mu} (x)&:=  \int_{\Omega}\partial_i (\eta_ku)(z)\Bigg(\psi_{\mu,x}G_s(x,\cdot)(-\Delta)^s\eta_k-\mathcal I_s\big[\eta_k, \psi_{\mu,x} G_s(x,\cdot)\big]\Bigg)dz,
\\ \mathcal C_{k,\mu}(x) &:= \int_{\Omega}\eta_k\partial_i\Big(\eta_k \psi_{\mu,x}G_s(x,\cdot)\Big)(-\Delta)^su(z)dz,
\\ \mathcal D_{k,\mu}(x) &:=  \int_{\Omega}\partial_i\Big(\eta_k \psi_{\mu,x}G_s(x,\cdot)\Big)\Big(u(-\Delta)^s\eta_k-\mathcal I_s[u,\eta_k]\Big)dz . \end{align*}
 Therefore, combining \eqref{key-id},  \eqref{rhs1} and  \eqref{rhs2}, we arrive at 
\begin{equation} \label{all}
   \mathcal  A_{k,\mu}(x) +\mathcal  C_{k,\mu}(x) = - \mathcal B_{k,\mu}(x) -\mathcal  D_{k,\mu}(x)  .
\end{equation}
Note that all the quantities above are well defined. The rest of the proof of Theorem \ref{main-result}  consists in passing to the limit, as $\mu\to 0^+$ and $k\to\infty$, in all the terms of \eqref{all}.

For the first terms in these r.h.s. we apply the following result from  \cite{DFW, DFW2023}. 

\begin{Lemma}\label{DFW}
Let $v$ and $w$ be such that $v\equiv 0\equiv w$ in $\R^N\setminus\Omega$. Assume moreover that $w\in C^s(\R^N)$, $v/\delta_\Omega^s, w/\delta_\Omega^s\in C^\alpha(\overline{\Omega})$ and 
 \begin{equation}\label{Gredient-estimate-FallJarohs}
\delta^{1-\alpha}_\Omega\nabla\big(v/\delta_\Omega^s\big) \;\;\text{is bounded near the boundary $\partial\Omega$}\; \text{for some}\; \alpha\in (0,1).
\end{equation}
 Then for all $Y\in C^{0,1}(\R^N,\R^N)$, there holds
 \begin{align}\label{DFW-id}
     &\lim_{k\to\infty}\int_{\Omega}\nabla(\eta_k w)\cdot Y \Big[w(-\Delta)^s\eta_k-\mathcal I_s\big[\eta_k, w\big]\Big]dz
     =\frac{\Gamma^2(1+s)}{2}\int_{\partial\Omega}\frac{v}{\delta_\Omega^s}\frac{w}{\delta_\Omega^s}Y\cdot\nu d\sigma
 \end{align}
where $\nu$ is the outward unit normal to the boundary.
\end{Lemma}

We apply this result to both the couples 
$(v,w) =( u ,\psi_{\mu,x}G_s(x,\cdot) )$
and 
$(v,w) =( \psi_{\mu,x}G_s(x,\cdot) , u)$ with $Y=e_i=(0,\cdots,1,0,\cdots,0)$, and then pass to the limit as $\mu\to 0^+$
to obtain that %
\begin{align} \label{B}
\lim_{\mu\to 0^+}\lim_{k\to\infty} \mathcal B_{k,\mu}(x) &= \frac{\Gamma^2(1+s)}{2}\int_{\partial \Omega} \frac{u(z)}{\delta^s_\Omega (z)} \frac{G_s(x,z)}{\delta^s_\Omega (z)} \nu_i(z)d\sigma(z) ,
    \\ \label{D}
    \lim_{\mu\to 0^+}\lim_{k\to\infty} \mathcal D_{k,\mu}(x) &= \frac{\Gamma^2(1+s)}{2}\int_{\partial \Omega} \frac{u(z)}{\delta^s_\Omega (z)} \frac{G_s(x,z)}{\delta^s_\Omega (z)} \nu_i(z)d\sigma(z) .
\end{align}
\par\;

We note that the couple $(u,\psi_{\mu,x}G_s(x,\cdot))$ satisfies the hypothesis of the Lemma above by regularity theory. Indeed, by \cite{RS, FallSven} to get that a function $w:\R^N\to \R$ with $w\equiv 0$ in $\R^N\setminus\Omega$ satisfies the hypothesis of the Lemma, it suffices to check that $(-\Delta)^sw\in L^\infty(\Omega)$. By assumption, it is clear that $(-\Delta)^su\in L^\infty(\Omega)$ and also by Lemma \ref{lem-ps-mu-x} we also have $(-\Delta)^s\Big(G_s(x,\cdot)\psi_{\mu,x}\Big)\in L^\infty(\Omega)$.
\par\;

Now, regarding 
 the passage to the limit of $\mathcal A_{k,\mu}(x)$ and $\mathcal C_{k,\mu}(x)$ 
 we have the following pair of results which are respectively proved in Section \ref{sec-pr-l22} and in Section \ref{sec8}.
 \begin{Lemma}\label{Lem4.2}
For all $x\in \Omega$, we have 
  \begin{equation} \label{Lem4.2-id}
\lim_{\mu\to 0^+}\lim_{k\to\infty} \mathcal A_{k,\mu}(x)
=
\frac{\partial u}{\partial x_i}(x).
\end{equation}
 \end{Lemma}
\begin{Lemma}\label{Lem4.3}
  For  all $x\in\Omega$, we have 
    \begin{align}\label{Lem4.3-id-1}
  \lim_{\mu\to 0^+}\lim_{k\to\infty} 
 \mathcal  C_{k,\mu}(x)
  =
  \int_{\Omega}\frac{\partial G_s}{\partial z_i}(x,\cdot) (-\Delta)^su\,dy=\int_{\Omega}\frac{\partial G_s}{\partial z_i}(x,y) f(y,u(y))\,dy,\quad\text{if}\;\; 2s>1,
  \\\label{Lem4.3-id-2}
  \lim_{\mu\to 0^+}\lim_{k\to\infty} 
 \mathcal  C_{k,\mu}(x)
  = -\int_\Omega  G_s(x,\cdot)\Bigg[\frac{\partial f}{\partial h_i}(y,u(y))+\frac{\partial f}{\partial q}(y,u(y))\frac{\partial u}{\partial y_i}(y)\Bigg]dy,\quad\text{if}\;\; 2s\leq 1.
  \end{align}
  In the case $f\equiv \text{const}$, we have 
  \begin{align}\label{Lem4.3-id-3}
     \lim_{\mu\to 0^+}\lim_{k\to\infty} 
 \mathcal  C_{k,\mu}(x)=0\quad\text{for all}\, s\in (0,1).
  \end{align}
\end{Lemma}

 Therefore, combining \eqref{all},  \eqref{B}, \eqref{D}, \eqref{Lem4.2-id},  \eqref{Lem4.3-id-1} and \eqref{Lem4.3-id-2}, we arrive at \eqref{repres-partial-deriv} and \eqref{repres-partial-deriv-1}. Up to the
proofs of Lemma \ref{Lem6.1}, Lemma \ref{Lem4.2} and Lemma \ref{Lem4.3}
the proof of Theorem \ref{main-result} is therefore done.


\subsection{Proof of Lemma \ref{Lem6.1}}
\label{ser-rad}

Let $B_r$ be the centered Euclidean ball of radius $r>0$. We decompose the integral into two parts, depending on whether both $x$ and $y$ belong to the ball $B_4$ or if only one of them does:
\begin{align}
  &\iint_{\R^N\times \R^N}\frac{\big(\rho(|y|^2)-\rho(|z|^2)\big)\big(|z|^{2s-N}-|y|^{2s-N}\big)}{|z-y|^{N+2s}}dydz\nonumber\\
  &=\iint_{B_4\times B_4}\frac{\big(\rho(|y|^2)-\rho(|z|^2)\big)\big(|z|^{2s-N}-|y|^{2s-N}\big)}{|z-y|^{N+2s}}dydz\label{Lem.3.2-1}\\ 
  &\;\;+ 2 \int_{B_2}\rho(|z|^2) \Big(\int_{\R^N\setminus B_4}\frac{|z|^{2s-N}-|y|^{2s-N}}{|z-y|^{N+2s}}dy \Big) dz\nonumber
\end{align}
Since 
$$\int_{\R^N}\frac{|z|^{2s-N}}{1+|z|^{N+2s}}dz<\infty,$$
it is clear that the last integral is finite. To continue the proof, we distinguish two cases. 
\par\;
\underline{Case 1}: $2s< 1$. In this case, we estimate
\begin{align}
   &\Bigg| \iint_{B_4\times B_4}\frac{\big(\rho(|y|^2)-\rho(|z|^2)\big)\big(|z|^{2s-N}-|y|^{2s-N}\big)}{|z-y|^{N+2s}}dydz\Bigg|\nonumber\\
   &\leq C \iint_{B_4\times B_4}\frac{|z|^{2s-N}+|y|^{2s-N}}{|z-y|^{N+2s-1}}dydz<\infty.
\end{align}
\underline{Case 2}: $2s\geq 1$.
In this case we decompose the integral in \eqref{Lem.3.2-1} into two parts, depending on whether both $x$ and $y$ do not belong to the ball $B_1$ or if only one of them does:
\begin{align}
  &\iint_{B_4\times B_4}\frac{\big(\rho(|y|^2)-\rho(|z|^2)\big)\big(|z|^{2s-N}-|y|^{2s-N}\big)}{|z-y|^{N+2s}}dydz\nonumber\\
  &=\iint_{(B_4\setminus B_1)\times (B_4\setminus B_1)}\frac{\big(\rho(|y|^2)-\rho(|z|^2)\big)\big(|z|^{2s-N}-|y|^{2s-N}\big)}{|z-y|^{N+2s}}dydz \label{Lem.3.2-2}\\ \label{Lem.3.2-3}
  &\;\;+2\iint_{B_1\times (B_4\setminus B_1)}\frac{\big(\rho(|y|^2)-\rho(|z|^2)\big)\big(|z|^{2s-N}-|y|^{2s-N}\big)}{|z-y|^{N+2s}}dydz.
\end{align}
Since $z\mapsto|z|^{2s-N}$ is $C^\infty(B_4\setminus B_1)$, by Taylor expansion we easily check that the integral in \eqref{Lem.3.2-2} is finite.  Next, we recall the following elementary estimate: 
for any $a, b\geq 0$ and $a\neq 0$, we have 
\begin{equation}\label{elem-id}
|a^\alpha-b^\alpha|\leq \frac{a^{\alpha-1}}{\alpha}\max\big(1,(b/a)^\alpha\big)|a-b|.
\end{equation}
Apply this with $\alpha=2s-N$, $a=|z|$ and $b=|y|$ yields
\begin{equation}\label{es-1}
  \big||z|^{2s-N}-|y|^{2s-N}\big|\leq C \frac{1}{|y|^{N-2s+1}}|z-y|\quad\text{for $y\in B_1\;\;\&\;\;z\in B_4\setminus B_1$.}  
\end{equation}
Now using \eqref{es-1} and that $\rho\in C^\infty_c(\R)$, it is straightforward to check the integral in \eqref{Lem.3.2-3} is finite provided that $2s>1$. It remains the case 
 $s=1/2$. For this we used the estimate: for $b>1$ and $\alpha\in (0,1)$ there holds:  
\begin{align*}
    \frac{1-b^{1-N}}{N-1}=\int_1^b t^{-N}dt&\leq \Big(\int_1^bdt\Big)^\alpha\Big(\int_1^b t^{-\frac{N}{1-\alpha}}dt\Big)^{1-\alpha}\nonumber\\
    &\leq \big(b-1\big)^\alpha\Big(\int_1^\infty t^{-\frac{N}{1-\alpha}}\Big)^{1-\alpha}\nonumber\\
    &\leq \big(b-1\big)^\alpha\Big(\frac{1-\alpha}{N+\alpha-1}\Big)^{1-\alpha}.
\end{align*}
Apply this to $\frac{|z|}{|y|}$ with $y\in B_1$ and $z\in B_4\setminus B_1$ gives 
\begin{align*}
    \Big||y|^{1-N}-|z|^{1-N}\Big|=|y|^{1-N}\Big|1-\Big(\frac{|z|}{|y|}\Big)^{1-N}\Big|&\leq  (N-1)|y|^{1-N}\Big(\frac{|z|}{|y|}-1\Big)^\alpha\Big(\frac{1-\alpha}{N+\alpha-1}\Big)^{1-\alpha}\nonumber\\
    &\leq C|y|^{1-N-\alpha}|z-y|^\alpha.
\end{align*}
It follows that 
\begin{align}
&\iint_{B_1\times (B_4\setminus B_1)}\frac{\big(\rho(|y|^2)-\rho(|z|^2)\big)\big(|z|^{1-N}-|y|^{1-N}\big)}{|z-y|^{N+1}}dydz\nonumber\\
&\leq C\int_{B_1}|y|^{1-N-\alpha}\int_{B_4\setminus B_1}\frac{dz}{|y-z|^{N-\alpha}}dy< \infty.
\end{align}
In conclusion we have 
\begin{align}
   \iint_{\R^N\times \R^N}\frac{\big(\rho(|y|^2)-\rho(|z|^2)\big)\big(|z|^{2s-N}-|y|^{2s-N}\big)}{|z-y|^{N+2s}}dydz<\infty.
\end{align}
Then, by Fubini's theorem,  we have that 
\begin{align}
&b_{N,s}c_{N,s}\iint_{\R^N\times \R^N}\frac{\big(\rho(|y|^2)-\rho(|z|^2)\big)\big(|z|^{2s-N}-|y|^{2s-N}\big)}{|z-y|^{N+2s}}dydz\nonumber\\
&=-2\int_{\R^N}b_{N,s}|z|^{2s-N}(-\Delta)^s(\rho\circ|\cdot|^2)dz=-2,
\end{align}
by using  an integration by parts over the full space $\R^N$,  that 
 $b_{N,s}|z|^{2s-N}$ is the fundamental solution of $(-\Delta)^s$ in $\R^N$ and the value of $\rho(0)$.

\subsection{Proof of Lemma \ref{Lem4.2}}
\label{sec-pr-l22}

We shall need the following result in the proof of Lemma \ref{Lem4.2} below.
\begin{Lemma}\label{ess-convergence-result} Let $w\in C(\Omega)$ and fix $x\in\Omega$. Let $G_s(x,\cdot)$ be the Green function with singularity at $x$ and $\psi_{x,\mu}$ be defined as in \eqref{psi-mu}. Then for all $\varepsilon>0$ such that $\overline{B_{2\varepsilon}(x)}\subset\Omega$, we have 
\begin{align}
    \lim_{\mu\to 0^+}\int_{B_\varepsilon(x)}w(p)G_s(x,p)(-\Delta)^s\psi_{\mu,x}(p)dp=-w(x),\label{fl}\\
    \lim_{\mu\to 0^+}\int_{B_\varepsilon(x)}w(p)\mathcal I_s\big[G_s(x,\cdot), \psi_{\mu,x}\big](p)dp=-2w(x).\label{sl}
\end{align}
\end{Lemma}
\par\;
\begin{remark}\label{rem-lem-2.6}
Up to choosing $\varepsilon>0$ sufficiently small, the assumption $w\in C(\Omega)$ in the lemma above can be released to the continuity of  $w$ at the point $x$.
\end{remark}
To ease the notation, we write $\partial_i:=\frac{\partial}{\partial x_i}$. Recall that $x\in\Omega$ is given and that for any $\mu\in(0,1)$,  and any $k\in \mathbb N*$, 
$$ \mathcal A_{k,\mu}(x) := \int_{\Omega}\eta_k(z)\partial_i\big(\eta_k u\big)(z)(-\Delta)^s\Big[\psi_{\mu,x}G_s(x,\cdot)\Big](z)dz. $$
We start with the following simple observation. We have 
    \begin{align}\label{Lemid}
\mathcal A_{k,\mu}(x) &= \int_{\Omega}\eta_k^2\partial_i u\Bigg(G_s(x,\cdot)(-\Delta)^s\psi_{\mu,x}-\mathcal I_s\big(\psi_{\mu,x},G_s(x,\cdot)\big)\Bigg)dz\nonumber\\
&\quad\quad  +k\int_{\Omega}u\eta_k(z)\eta'(k\delta)(z)\partial_i\delta(z)(-\Delta)^s\Big(\psi_{\mu,x}G_s(x,\cdot)\Big)(z)dz .
  \end{align}
Indeed, by Leibniz' rule, we have 
\begin{align*}
   \mathcal A_{k,\mu}(x)
    &=\int_{\Omega}\eta_k^2(z)\partial_i u(z)(-\Delta)^s\Big[\psi_{\mu,x}G_s(x,\cdot)\Big](z)dz\nonumber\\
    &\quad+k\int_{\Omega}u\eta_k(z)\eta'(k\delta)(z)\partial_i\delta(z)(-\Delta)^s\Big[\psi_{\mu,x}G_s(x,\cdot)\Big](z)dz .
\end{align*}
Next, using the product rule \eqref{pl} with
 $u=\psi_{\mu,x} $, $v=G_s(x,\cdot) $, and using also that $(-\Delta)^s G_s(x,\cdot)=0$ in $\Omega\setminus \overline{B_\mu(x)}$, we get
     \begin{align*}
    \small\text{$\int_{\Omega}\eta_k^2\partial_i u(-\Delta)^s\Big[\psi_{\mu,x}G_s(x,\cdot)\Big]dz
    =\int_{\Omega}\eta_k^2\partial_i u\Bigg[G_s(x,\cdot)(-\Delta)^s\psi_{\mu,x}-\mathcal I_s\big[\psi_{\mu,x},G_s(x,\cdot)\big]\Bigg]dz.$}
\end{align*}
The claim \eqref{Lemid} follows by combining the two identities above.
The rest of the proof of Lemma \ref{Lem4.2}
consists in passing to the limit, as $\mu\to 0^+$ and $k\to\infty$, in the two terms in the right hand side of \eqref{Lemid}. 
We start with a result regarding the double-limit of the first term, which by Lemma \ref{lem-ps-mu-x} can be reduced to the one of the same integral on an arbitrarily small ball. That is,
\begin{gather}
    \label{sec7-2}
    \lim_{\mu\to 0^+}
    \lim_{k\to\infty}\int_{\Omega}\eta_k^2(z)\partial_iu(z)\Bigg[G_s(x,\cdot)(-\Delta)^s\psi_{\mu,x}-\mathcal I_s\big[\psi_{\mu,x},G_s(x,\cdot)\big]\Bigg]dz 
    \\ \nonumber   =\lim_{\mu\to 0^+}\lim_{k\to\infty}\int_{B_\varepsilon(x)}\eta_k^2(z)\partial_i u(z)\Bigg[G_s(x,\cdot)(-\Delta)^s\psi_{\mu,x}-\mathcal I_s\big[\psi_{\mu,x},G_s(x,\cdot)\big]\Bigg]dz.
    \end{gather} 
\par\;
In view of \eqref{Lemid} and \eqref{sec7-2} and applying Lemma \ref{ess-convergence-result} with $w=\partial_iu\in C(\Omega)$, we get 
\begin{align*}
    &\lim_{\mu \to 0^+}\lim_{k\to\infty}\mathcal A_{k,\mu}(x)\nonumber\\
    &=\lim_{\mu\to 0^+}\lim_{k\to\infty}\int_{B_\varepsilon(x)}\eta_k^2(z)\partial_i u(z)\Bigg[G_s(x,\cdot)(-\Delta)^s\psi_{\mu,x}-\mathcal I_s\big[\psi_{\mu,x},G_s(x,\cdot)\big]\Bigg]dz\nonumber\\
    &\quad\quad  +\lim_{\mu\to 0^+}\lim_{k\to\infty}k\int_{\Omega}u\eta_k(z)\eta'(k\delta)(z)\partial_i\delta(z)(-\Delta)^s\Big(\psi_{\mu,x}G_s(x,\cdot)\Big)dz\nonumber\\
    &=\lim_{\mu\to 0^+}\int_{B_\varepsilon(x)}\partial_i u(z)\Bigg[G_s(x,\cdot)(-\Delta)^s\psi_{\mu,x}-\mathcal I_s\big[\psi_{\mu,x},G_s(x,\cdot)\big]\Bigg]dz\nonumber\\
    &\quad\quad  +\lim_{\mu\to 0^+}\lim_{k\to\infty}k\int_{\Omega}u\eta_k(z)\eta'(k\delta)(z)\partial_i\delta(z)(-\Delta)^s\Big(\psi_{\mu,x}G_s(x,\cdot)\Big)dz\nonumber\\
    &=\partial_iu(x)+\lim_{\mu\to 0^+}\lim_{k\to\infty}k\int_{\Omega}u\eta_k(z)\eta'(k\delta)(z)\partial_i\delta(z)(-\Delta)^s\Big(\psi_{\mu,x}G_s(x,\cdot)\Big)dz.
\end{align*}
Thus to deduce Lemma \ref{Lem4.2} it suffices to  prove that
\begin{align}\label{pf}
   \lim_{\mu\to 0^+}\lim_{k\to\infty}k\int_{\Omega}u\eta_k(z)\eta'(k\delta)(z)\partial_i\delta(z)(-\Delta)^s\Big(\psi_{\mu,x}G_s(x,\cdot)\Big)dz=0. 
\end{align}
\par\;
To see \eqref{pf}, we note that since $(-\Delta)^s\Big[\psi_{\mu,x}G_s(x,\cdot)\Big]\in L^\infty(\Omega)$ (see Lemma \ref{lem-ps-mu-x}) and $|u(z)|\leq C\delta^s(z)$ we have 
\begin{align}\label{sec7-19}
    &\Bigg|k\int_{\Omega}u\eta_k(z)\eta'(k\delta)(z)\partial_i\delta(z)(-\Delta)^s\Big[\psi_{\mu,x}G_s(x,\cdot)\Big](z)dz\Bigg|\nonumber\\
    &\leq C(\mu)k\int_{\Omega_{\frac{2}{k}}\setminus\Omega_{\frac{1}{k}}}|\rho'(k\delta(z))|\delta^s(z)dz\nonumber\\
    &\leq C(\mu)\int_{\Omega_{\frac{2}{k}}\setminus\Omega_{\frac{1}{k}}}\delta^{s-1}(z)|\rho'(k\delta(z))|dz\to 0\;\;\text{as}\;\;k\to\infty
\end{align}
by the dominated convergence theorem. This ends the proof of \eqref{pf}. In the third line, we used that if $z\in \Omega_{\frac{2}{k}}\setminus\Omega_{\frac{1}{k}}$ then $1\leq k\delta(z)\leq 2$. The proof of Lemma \ref{Lem4.2} is therefore finished.
\par\;

\subsection{Proof of Lemma \ref{ess-convergence-result}}
In this subsection we prove Lemma \ref{ess-convergence-result}.
We start with the proof of \eqref{fl}. First \eqref{rescapsi} is straightforward. 
Using it with a translation,  we get

\begin{align}\label{sec7-6}
   &-\int_{B_\varepsilon(x)}w(z)G_s(x,z)(-\Delta)^s\psi_{\mu,x}(z)dz\nonumber\\
   &=\widetilde\mu^{-2s}\int_{B_\varepsilon(0)}G_s(x,x-z)w(x-z)(-\Delta)^s\big(\rho\circ |\cdot|\big)\big(\frac{z}{\widetilde\mu}\big)dz\nonumber\\
   &=\int_{B_{\frac{\varepsilon}{\widetilde\mu}}(0)}\widetilde\mu^{N-2s}G_s(x,x-\widetilde\mu z)w(x-\widetilde\mu z)(-\Delta)^s\big(\rho\circ |\cdot|^2\big)(z)dz.
\end{align}

Since $w\in L^\infty_{loc}(\Omega)$, we have $\big|1_{B_{\frac{\varepsilon}{\widetilde\mu}}(0)}(z)w(x-\widetilde\mu z)\big|\leq C$.  Moreover $\big|\widetilde\mu^{N-2s}G_s(x,x-\widetilde\mu z)\big|\leq C|z|^{2s-N}$ (see e.g  \eqref{eq:greenfunct-estimate}). Now since $w$ is continuous at $x$ and 
$$ 
\int_{\R^N}|z|^{2s-N}(-\Delta)^s\big(\rho\circ |\cdot|^2\big)(z)dz<\infty,
$$
we have by the Lebesgue dominated convergence theorem that
\begin{align}\label{sec7-9}
&\lim_{ \mu\to 0^+}\int_{B_{\frac{\varepsilon}{\widetilde\mu}}(0)}\widetilde\mu^{N-2s}G_s(x,x-\widetilde\mu z)w(x-\widetilde\mu z)(-\Delta)^s\big(\rho\circ |\cdot|^2\big)dz\nonumber\\&=w(x)\int_{\R^N}b_{N,s}|z|^{2s-N}(-\Delta)^s\big(\rho\circ |\cdot|^2\big)dz.
\end{align}
Here we used that 
$\widetilde\mu^{N-2s}G_s(x,x-\widetilde\mu z)=\frac{b_{N,s}}{|z|^{N-2s}}-\widetilde\mu^{N-2s}H_s(x,x-\widetilde\mu z)\to b_{N,s}|z|^{2s-N}$ as $\mu\to 0^+$ since $H_s(x,\cdot)\in C(\Omega)$.  
Finally, as we have already seen above, since $b_{N,s}|z|^{2s-N}$ is the fundamental solution of $(-\Delta)^s$ in $\R^N$ and $\rho\circ|\cdot|^2\in C^\infty_c(\R^N)$, we deduce by an integration by parts over the full space $\R^N$ that:
\begin{equation}\label{sec7-8}
\int_{\R^N}b_{N,s}|z|^{2s-N}(-\Delta)^s\big(\rho\circ |\cdot|^2\big)dz=\rho(0)=1.
\end{equation}
Combining \eqref{sec7-8} and \eqref{sec7-9} we obtain \eqref{fl}. The proof of \eqref{sl} is somehow similar to the proof of \eqref{fl} with some minor differences but for completeness we give the full details of the argument. 

We first observe that
\begin{align}\label{sec7-11}
    &\lim_{\mu\to 0^+}\lim_{k\to\infty}\int_{B_\varepsilon(x)}w(z)\mathcal I_s\big[\psi_{ \mu,x},G_s(x,\cdot)\big](z)dz\nonumber\\
    &=\lim_{\mu\to 0^+}c_{N,s}\int_{B_\varepsilon(x)}w(z)\int_{B_{2\varepsilon}(x)}\frac{(\psi_{\mu,x}(y)-\psi_{\mu,x}(z))(G_s(x,y)-G_s(x,z))}{|y-z|^{N+2s}}dydz.
\end{align}
Indeed, for $\mu\in(0,1)$ sufficiently small and $|x-y|>\varepsilon$, we have $\psi_{\mu,x}(y)=0$ and thus recalling \eqref{eq:greenfunct-estimate}, we have 
\begin{align}
    &\Bigg|\int_{B_\varepsilon(x)}w(z)\int_{\R^N\setminus B_{2\varepsilon}(x)}\frac{(\psi_{\mu,x}(y)-\psi_{\mu,x}(z))(G_s(x,y)-G_s(x,z))}{|y-z|^{N+2s}}dydz\Bigg|\nonumber\\
    &\leq C(\varepsilon,\Omega, N,s)\int_{B_\varepsilon(x)}\big(1+|G_s(x,z)|\big)\big|w(z)\big|\rho\Big(\frac{8}{\delta^2_\Omega(x)}\frac{|x-z|^2}{\mu^2}\Big)\int_{\R^N\setminus B_{2\varepsilon}(x)}\frac{dy}{1+|y|^{N+2s}}dz\label{Eq9/7}
\end{align}
The r.h.s of \eqref{Eq9/7} converges to zero as $\mu\to 0^+$ which proves the claim \eqref{sec7-11}.
\par\;

Next applying the change of variables: $\overline{z}=\frac{x-z}{\widetilde\mu}$ and $\overline{y}=\frac{x-y}{\widetilde\mu}$, we get 
\begin{align}\label{sec7-12}
  &\int_{B_\varepsilon(x)}w(z)\int_{B_{2\varepsilon}(x)}\frac{(\psi_{\widetilde\mu,x}(y)-\psi_{\widetilde\mu,x}(z))(G_s(x,y)-G_s(x,z))}{|y-z|^{N+2s}}dydz \\ \label{issy2}
  &= \int_{B_{\frac{\varepsilon}{\widetilde\mu}(0)}  }w(x-\widetilde\mu z)\int_{B_{\frac{2\varepsilon}{\widetilde\mu}}(0)}\frac{\Big(\rho(|z|^2)-\rho(|y|^2)\Big)\Big(\widetilde\mu^{N-2s}G_s(x,x-\widetilde\mu y)-\widetilde\mu^{N-2s}G_s(x,x-\widetilde\mu z)\Big)}{|y-z|^{N+2s}}dydz.
\end{align}
We are going to split the integral in the right hand side above into two parts according to the decomposition:
\begin{align*}
   &\widetilde\mu^{N-2s}G_s(x,x-\widetilde\mu y)-\widetilde\mu^{N-2s}G_s(x,x-\widetilde\mu z)\\
   &=b_{N,s}\big(|y|^{2s-N}-|z|^{2s-N}\big)-\widetilde\mu^{N-2s}\big(H_s(x,x-\widetilde\mu y)-H_s(x,x-\widetilde\mu z)\big).
\end{align*}
Since $H_s(x,\cdot)$ is $s$-harmonic, we know $\nabla_yH_s(x,y)$ is locally bounded in $\Omega$. Consequently 
$$\big|H_s(x,x-\widetilde\mu y)-H_s(x,x-\widetilde\mu z)\big|\leq C(x,\varepsilon)\widetilde\mu|y-z|\quad \text{for all}\,z\in B_{\frac{\varepsilon}{\widetilde\mu}}(0)\;\;\&\;\;\text{ for all}\;\;y\in B_{\frac{2\varepsilon}{\widetilde\mu}}(0) .$$
Thus, we have
\begin{align}
&\Bigg|\iint_{B_{\frac{\varepsilon}{\widetilde\mu}(0)}\times B_{\frac{2\varepsilon}{\widetilde\mu}(0)}}\frac{\Big(\rho(|z|^2)-\rho(|y|^2)\Big)\Big(\widetilde\mu^{N-2s}H_s(x,x-\widetilde\mu y)-\widetilde\mu^{N-2s}H_s(x,x-\widetilde\mu z)\Big)}{|y-z|^{N+2s}}dydz\Bigg|\nonumber\\  \label{sec7-13}
&\leq C(x)\widetilde\mu^{N-2s+1}\iint_{B_{\frac{\varepsilon}{\widetilde\mu}(0)}\times B_{\frac{2\varepsilon}{\widetilde\mu}(0)}}\frac{\big|\rho(|z|^2)-\rho(|y|^2)\big||z-y|}{|y-z|^{N+2s}}dydz\;\;\to\;\; 0\;\;\text{as}\;\; \mu\to 0^+.
\end{align}
On the other hand, by Lemma \ref{Lem6.1} and the Lebesgue dominated convergence theorem, 
\begin{align}\label{sec7-14}
&\lim_{\mu\to 0^+}\iint_{B_{\frac{\varepsilon}{\widetilde\mu}(0)}\times B_{\frac{2\varepsilon}{\mu}(0)}}w(x-\widetilde\mu z)\frac{\Big(\rho(|z|^2)-\rho(|y|^2)\Big)\Big(b_{N,s}|y|^{2s-N}-b_{N,s}|z|^{2s-N}\Big)}{|y-z|^{N+2s}}dydz  \nonumber\\
&= w(x)b_{N,s}\iint_{\R^N\times\R^N}\frac{\Big(\rho(|z|^2)-\rho(|y|^2)\Big)\Big(|y|^{2s-N}-|z|^{2s-N}\Big)}{|y-z|^{N+2s}}dydz =- \frac{ 2 }{c_{N,s }} w(x).
\end{align}
Thus, combining \eqref{sec7-13} and \eqref{sec7-14} we get
\begin{align}
\label{issy}
  &\small\text{$\lim_{\mu\to 0^+}c_{N,s}\int_{B_{\frac{\varepsilon}{\widetilde\mu}(0)}}w(x-\widetilde\mu z)\int_{B_{\frac{2\varepsilon}{\widetilde\mu}}}\frac{\Big(\rho(|z|^2)-\rho(|y|^2)\Big)\Big(\widetilde\mu^{N-2s}G_s(x,x-\widetilde\mu y)-\widetilde\mu^{N-2s}G_s(x,x-\widetilde\mu z)\Big)}{|y-z|^{N+2s}}dydz$}\nonumber\\
  &\quad\quad\quad\quad=-2w(x).
\end{align}
Combining 
\eqref{sec7-11},  \eqref{issy2} and 
\eqref{issy} we arrive at  \eqref{sl}.

\subsection{Proof of Lemma \ref{Lem4.3}}
\label{sec8}

 Let us recall that 
\begin{align}
\mathcal  C_{k,\mu}(x):= \int_{\Omega}\eta_k\partial_i\Big(\eta_k \psi_{\mu,x}G_s(x,\cdot)\Big)(-\Delta)^su(z)dz\quad\text{for $x\in\Omega$}.
 \end{align}
 We consider the two different cases separately.
 \par\;

 \underline{Case 1}: $2s>1$. By Leibniz rule, we have 
\begin{align}\label{sec8-1}
\mathcal C_{k,\mu}(x)&=\int_{\Omega}\eta_k^2(z)\psi_{\mu,x}(z)\partial_i G_s(x,z)(-\Delta)^su(z)\,dz\nonumber\\
 &\;\;+\int_{\Omega}\eta_k(z)G_s(x,z)\partial_i\big(\eta_k\psi_{\mu,x}\big)(z)(-\Delta)^su(z)dz
\end{align}
Using that $(-\Delta)^su\in L^\infty(\Omega)$ and recalling \eqref{eq:greenfunct-estimate}, we estimate 
\begin{align}\label{sec8-2}
&\lim_{\mu\to 0^+} \lim_{k\to\infty}\Big|\int_{\Omega}\eta_k^2(z)\partial_i\psi_{\mu,x}G_s(x,z)(-\Delta)^su(z)dz\Big|\nonumber\\
&\leq C(x)\lim_{\mu\to 0^+}\int_{B\big(x, \frac{\delta(x)}{2}\mu\big)\setminus B\big(x, \frac{\delta(x)}{2\sqrt{2}}\mu\big)}\Bigg|\frac{(x-z)\cdot e_i}{\mu^2}\rho'\Big(\frac{8}{\delta^2(x)}\frac{|x-z|^2}{\mu^2}\Big)\Bigg||x-z|^{2s-N}dz\nonumber\\
&\leq C(x) \lim_{\mu\to 0^+}\int_{B\big(x, \frac{\delta(x)}{2}\mu\big)\setminus B\big(x, \frac{\delta(x)}{2\sqrt{2}}\mu\big)}  \frac{1}{\mu}\Bigg|\rho'\Big(\frac{8}{\delta^2(x)}\frac{|x-z|^2}{\mu^2}\Big)\Bigg||x-z|^{2s-N}dz\nonumber\\
&\leq C(x)\lim_{\mu\to 0^+} \mu^{2s-1}
\int_{ \frac{\delta(x)}{2\sqrt{2}} }^{\frac{\delta(x)}{2}} 
r^{2s-1} |\rho'(r^2)|dr=0 ,
\end{align}
since $2s>1$. A similar argument as in \eqref{pf} gives 
\begin{align}\label{sec8-3}
\lim_{\mu\to 0^+}\lim_{k\to\infty}\int_{\Omega}\eta_k(z)\psi_{\mu,x}\partial_i\eta_k(z)G_s(x,z)(-\Delta)^su(z)dz=0.
\end{align}
The first identity in the Lemma follows by taking the limit in \eqref{sec8-1} and using \eqref{sec8-2} and \eqref{sec8-3}.
\par\;

\underline{Case 2}: $2s\leq 1$. In this case, we have by an integration by parts
\begin{align}\label{cases-fl}
\mathcal C_{k,\mu}(x):&= \int_{\Omega}\eta_k\partial_i\Big(\eta_k \psi_{\mu,x}G_s(x,\cdot)\Big)(-\Delta)^su(z)dz\nonumber\\
 &= \int_{\Omega}\partial_i\Big(\eta_k \psi_{\mu,x}G_s(x,\cdot)\Big)(z)\eta_k(z)f(z,u(z))dz\nonumber\\
 &=-\int_{\Omega}\eta_k \psi_{\mu,x}G_s(x,\cdot) \partial_i\big(\eta_kf(\cdot, u(\cdot))\big)(z)dz\nonumber\\
 &=-\int_{\Omega}\eta_k^2\psi_{\mu,x}G_s(x,\cdot)\Bigg[\frac{\partial f}{\partial h_i}(z,u(z))+\frac{\partial f}{\partial q}(z,u(z))\frac{\partial u}{\partial y_i}(z)\Bigg]dz\nonumber\\
 &\quad\quad-k\int_{\Omega}\eta_k(z)\psi_{\mu,x}(z)G_s(x,z)\rho'(k\delta(z))\partial_i\delta(z) f(z,u(z))dz
 \end{align}
In the one hand, a similar argument to the proof of \eqref{pf} gives 
\begin{align}\label{case2-sl}
    \lim_{k\to\infty}k\int_{\Omega}\eta_k(z)\psi_{\mu,x}(z)G_s(x,z)\rho'(k\delta(z))\partial_i\delta(z) f(z,u(z))dz=0.
\end{align}
On the other hand by Lebesgue dominated convergence theorem we have 
\begin{align}\label{case2-tl}
 &\lim_{\mu\to 0^+}\lim_{k\to\infty}\int_{\Omega}\eta_k^2\psi_{\mu,x}G_s(x,\cdot)\Bigg[\frac{\partial f}{\partial h_i}(z,u(z))+\frac{\partial f}{\partial q}(z,u(z))\frac{\partial u}{\partial y_i}(z)\Bigg]dz\nonumber\\
 &\quad\quad \quad \quad=  \int_{\Omega}G_s(x,z)\Bigg[\frac{\partial f}{\partial h_i}(z,u(z))+\frac{\partial f}{\partial q}(z,u(z))\frac{\partial u}{\partial y_i}(z)\Bigg]dz.
\end{align}
In view of \eqref{case2-sl} and \eqref{case2-tl}, we get the desired identity by taking the limit as $k\to\infty$ and as $\mu\to 0^+$ into \eqref{cases-fl}.  Finally, to see \eqref{Lem4.3-id-3} we use integration by parts to write
\begin{align*}
 \lim_{\mu\to 0^+}\lim_{k\to\infty} \mathcal C_{k,\mu}(x)&=\lim_{\mu\to 0^+}\lim_{k\to\infty} \int_{\Omega}\eta_k\partial_i\Big(\eta_k \psi_{\mu,x}G_s(x,\cdot)\Big)(-\Delta)^su(z)dz\nonumber\\
 &=c\lim_{\mu\to 0^+} \int_{\Omega}\partial_i\Big(\psi_{\mu,x}G_s(x,\cdot)\Big)+c\lim_{\mu\to 0^+}\lim_{k\to\infty}k\int_{\Omega}\eta_k\rho'(k\delta)\partial_i\delta\psi_{\mu,x}G_s(x,\cdot)dz.
 \end{align*}
Since $\psi_{\mu,x}G_s(x,\cdot)=0$ on $\partial\Omega$, it is clear that the first limit in the identity above is zero by the divergence theorem. We also know from above that the second limit is zero as well. Consequently, we have $ \lim_{\mu\to 0^+}\lim_{k\to\infty}\mathcal C_{k,\mu}(x)=0$ which prove \eqref{Lem4.3-id-3}.
The proof of Lemma \ref{Lem4.3} is therefore finished.


\section{Proof of Theorem \ref{Thm.2}}\label{sec4}
This section is devoted to the proof of Theorem  \ref{Thm.2}. 
We start with the proof of  \eqref{partial-Green}, for which an intermediate result: Lemma \ref{lem-lim-G-x-y} is necessary. Its proof is postponed to Section \ref{sec9}. Finally we  turn to the proof of  \eqref{repres-gradient-Robin func}.

\subsection{Proof of  \eqref{partial-Green}}

First, we use  the functions $\R^N\to\R, \;z\mapsto \psi_{\gamma,y}(z)G_s(y,z)\eta_k(z)$ and $\R^N\to\R,\;z\mapsto\psi_{\mu,x}(z)G_s(x,z)\eta_k(z)$ with $x\neq y$ as test functions into \eqref{eq:pohozaev.1} and expand to get
\begin{equation} \label{all-2}
   \mathcal  E_{k,\mu,\gamma}(x,y) + \mathcal G_{k,\mu,\gamma}(x,y) = -\mathcal  F_{k,\mu,\gamma}(x,y) -\mathcal  H_{k,\mu,\gamma}(x,y)  
\end{equation}
with 
\begin{align*}
 \mathcal E_{k,\mu,\gamma}(x,y) &:= \int_{\Omega}\eta_k(z)\partial_i\big(\eta_k \psi_{\gamma,y}G_s(y,\cdot)\big)(z)(-\Delta)^s\Big[\psi_{\mu,x}G_s(x,\cdot)\Big](z)dz, 
\\ \mathcal F_{k,\mu,\gamma} (x,y)&:=  \int_{\Omega}\partial_i (\eta_k\psi_{\mu,x}G_s(x,\cdot))(z)\Big[\psi_{\gamma,y}G_s(y,\cdot)(-\Delta)^s\eta_k-\mathcal I_s\big[\eta_k\psi_{\gamma,y}, G_s(y,\cdot)\big]\Big]dz ,
\\ \mathcal G_{k,\mu,\gamma}(x,y) &:= \int_{\Omega}\eta_k\partial_i\Big(\eta_k \psi_{\mu,x}G_s(x,\cdot)\Big)(-\Delta)^s\Big[\psi_{\gamma,y}G_s(y,\cdot)\Big](z)dz ,
\\ \mathcal H_{k,\mu,\gamma}(x,y) &:=  \int_{\Omega}\partial_i\Big(\eta_k \psi_{\mu,x}G_s(x,\cdot)\Big)\Big[\psi_{\gamma,y}G_s(y,\cdot)(-\Delta)^s\eta_k-\mathcal I_s[\psi_{\mu,x}G_s(x,\cdot),\eta_k]\Big]dz .
 \end{align*}
On the one hand, we applied Lemma \ref{DFW} to the couple $\big(\psi_{\gamma,y}G_s(y,\cdot),\psi_{\mu,x}G_s(x,\cdot)\big)$ and then pass to the limit, first as $\mu\to 0^+$ and then as $\gamma\to 0^+$ to get that
\begin{align}\label{lim-F-x-y}
    \lim_{\gamma\to 0^+}\lim_{\mu\to 0^+}\lim_{k\to\infty}\mathcal F_{k,\mu,\gamma} (x,y)=\frac{\Gamma^2(1+s)}{2}\int_{\partial \Omega}\frac{G_s(x,z)}{\delta^s_\Omega (z)} \frac{G_s(y,z)}{\delta^s_\Omega (z)} \nu_i(z)d\sigma(z),
\end{align}
\begin{align}\label{lim-H-x-y}
   \lim_{\gamma\to 0^+} \lim_{\mu\to 0^+}\lim_{k\to\infty}\mathcal H_{k,\mu,\gamma} (x,y)=\frac{\Gamma^2(1+s)}{2}\int_{\partial \Omega} \frac{G_s(x,z)}{\delta^s_\Omega (z)} \frac{G_s(y,z)}{\delta^s_\Omega (z)} \nu_i(z)d\sigma(z)
\end{align}
where $\nu$ is the outward unit normal to the boundary. On the other hand, Lemma \ref{Lem4.2} applied with $u=\psi_{\gamma,y}G_s(y,\cdot)$ yields
\begin{align*}
    \lim_{\mu\to 0^+}\lim_{k\to\infty}\mathcal E_{k,\mu,\gamma}(x,y) &=\frac{\partial(\psi_{\gamma,y}G_s(y,\cdot))}{\partial x_i}(x)\nonumber\\
    &=\psi_{\gamma,y}(x)\frac{\partial G_s(y,x)}{\partial x_i}-\frac{16}{\delta^2(x)}\frac{(x-y)\cdot e_i}{\gamma^2}\rho'\Big(\frac{8}{\delta^2(x)}\frac{|x-y|^2}{\gamma^2}\Big).
\end{align*}
Passing to the limit as $\gamma\to 0^+$, and observing that 
$\rho'\Big(\frac{8}{\delta^2(x)}\frac{|x-y|^2}{\gamma^2}\Big)$ vanishes  when $\gamma>0$ is sufficiently small, 
we get
\begin{align}\label{lim-E-x-y}
    \lim_{\gamma\to 0^+} \lim_{\mu\to 0^+}\lim_{k\to\infty}\mathcal E_{k,\mu,\gamma}(x,y)=\frac{\partial G_s(y,x)}{\partial x_i}.
\end{align}

In view of \eqref{all-2}, \eqref{lim-F-x-y}, \eqref{lim-H-x-y}, \eqref{lim-E-x-y}, the formula announced in \eqref{partial-Green} follows once we prove the following result.
\begin{Lemma}\label{lem-lim-G-x-y}
Let $s\in (0,1)$. Then for any $x,y\in\Omega$ with $x\neq y$. 
there holds: 
\begin{equation}
\lim_{\gamma\to 0^+}  \lim_{\mu\to 0^+}\lim_{k\to\infty} G_{k,\mu, \gamma}(x,y) = \frac{\partial G_s}{\partial y_i}(x,y).
\end{equation}
\end{Lemma}

\subsection{Proof of Lemma \ref{lem-lim-G-x-y}}\label{sec9}

First we note that since $(-\Delta)^s G_s(y,\cdot)=0$ in $\Omega\setminus\overline{B_\varepsilon(y)}$ for all $\varepsilon>0$, and we  apply the fractional product law \eqref{pl} to obtain:
\begin{align}\label{G-k-mu-gamma}
\mathcal  G_{k,\mu,\gamma}(x,y) = \int_{\Omega}\eta_k\partial_i\Big(\eta_k \psi_{\mu,x}G_s(x,\cdot)\Big) \Bigg[G_s(y,\cdot)(-\Delta)^s\psi_{\gamma,y}-\mathcal I_s\big[G_s(y,\cdot),\psi_{\gamma,y}\big]\Bigg]dz.
\end{align}
Let $x,y\in\Omega$ such that $x\neq y$ and let $G_s(x,\cdot)$ and $G_s(y,\cdot)$ be the Green functions with singularity $x$ and $y$ respectively. Recall 
$$
\psi_{y,\gamma}(z)=1-\rho\Big(\frac{8}{\delta^2(y)}\frac{|z-y|^2}{\gamma^2}\Big) \;\;\;\text{and}\;\;\; \psi_{x,\mu}(z)=1-\rho\Big(\frac{8}{\delta^2(x)}\frac{|z-x|^2}{\gamma^2}\Big)\quad\text{for \;$\gamma,\mu\in (0,1)$.}
$$
 By Leibniz rule, we have 
 \begin{equation}\label{dedede}
    \mathcal  G_{k,\mu,\gamma}(x,y) = \mathcal T^1_{k,\mu,\gamma}(x,y)+\mathcal T^2_{k,\mu,\gamma}(x,y) , 
 \end{equation}
 where 
\begin{align*}  \mathcal T^1_{k,\mu,\gamma}(x,y) &:= \int_{\Omega}\eta_k^2\partial_i\Big(\psi_{\mu,x}G_s(x,\cdot)\Big) \Bigg\{G_s(y,\cdot)(-\Delta)^s\psi_{\gamma,y}-\mathcal I_s\big[G_s(y,\cdot),\psi_{\gamma,y}\big]\Bigg\}dz , \nonumber\\
 \mathcal T^2_{k,\mu,\gamma}(x,y) :=&2k\int_{\Omega}\eta_k(z)\rho'(k\delta)\partial_i\delta(z)\Big(\psi_{\mu,x}G_s(x,\cdot)\Big) (-\Delta)^s\big(\psi_{\gamma,y}G_s(z,\cdot)\big)dz .
\end{align*}
Next, arguing as in the proof of \eqref{psi-mu-equ} in Lemma \ref{lem-ps-mu-x} we have:
\begin{align}\label{sec9-1}
    &\lim_{\gamma\to 0^+}\lim_{\mu\to 0^+}\lim_{k\to \infty}\mathcal T^1_{k,\mu,\gamma}(x,y)\nonumber\\
    &=\small\text{$\lim_{\gamma\to 0^+}\lim_{\mu\to 0^+}\lim_{k\to\infty}\int_{B_\varepsilon(y)}\eta_k^2(z)\partial_i\Big(\psi_{\mu,x}G_s(x,\cdot)\Big)(z) \Bigg\{G_s(y,\cdot)(-\Delta)^s\psi_{\gamma,y}-\mathcal I_s\big[G_s(y,\cdot),\psi_{\gamma,y}\big]\Bigg\}dz,$}
\end{align}
for all $\varepsilon>0$.  To proceed further, we pick
\begin{equation}\label{Eq-choice-of-vareps}
\varepsilon=\varepsilon_0:=\min\Big(\frac{\delta_\Omega(y)}{4},\frac{|x-y|}{2}\Big)>0\;\;\; \text{so that}\;\;\;\overline{B_{2\varepsilon_0}(y)}\subset \Omega.
\end{equation}
\par\;
By the choice of $\varepsilon_0$, we know that if $z\in B_{\varepsilon_0}(y)$, then $|x-z|\geq |x-y|/2>0$ and hence $|G_s(x,z)|\leq C(\varepsilon_0)$ for all $z\in B_{\varepsilon_0}(y)$. Consequently
\begin{align}\label{sec9-7}
 &\Bigg|\int_{B_{\varepsilon_0}(y)}G_s(x,z)\partial_i\psi_{\mu,x}(z) \Bigg\{G_s(y,\cdot)(-\Delta)^s\psi_{\gamma,y}-\mathcal I_s\big[G_s(y,\cdot),\psi_{\gamma,y}\big]\Bigg\}dz\Bigg|\nonumber\\
 &\leq C(\varepsilon_0, \Omega,\gamma)\int_{B_{\varepsilon_0}(y)}|z-y|^{2s-N}\frac{1}{\mu}\Big|\rho'\Big(\frac{8}{\delta^2(x)}\frac{|x-z|^2}{\mu^2}\Big)\Big|dz\to 0\quad\text{as}\quad\mu\to 0^+ .
\end{align}
Indeed, since by the choice of $\varepsilon_0>0$ we have $|x-z|\geq |x-y|/2>0$ for all $z\in B_{\varepsilon_0}(y)$, therefore $\rho'\Big(\frac{8}{\delta^2(x)}\frac{|x-z|^2}{\mu^2}\Big)=0$ when $\mu>0$ is sufficiently small.  In \eqref{sec9-7} we used that $(-\Delta)^s\psi_{\gamma,y}\in L^\infty(\R^N)$ and $\mathcal I_s\big[G_s(y,\cdot),\psi_{\gamma,y}\big]\in L^\infty(B_{\varepsilon_0}(y))$. In view of \eqref{sec9-1} and \eqref{sec9-7}, we have 
    \begin{align}
    &\lim_{\gamma\to 0^+}\lim_{\mu\to 0^+}\lim_{k\to \infty}\mathcal T^1_{k,\mu,\gamma}(x,y)\nonumber\\
    &=\small\text{$\lim_{\gamma\to 0^+}\lim_{\mu\to 0^+}\lim_{k\to\infty}\int_{B_{\varepsilon_0}(y)}\eta_k^2(z)\partial_i\Big(\psi_{\mu,x}G_s(x,\cdot)\Big)(z) \Bigg\{G_s(y,\cdot)(-\Delta)^s\psi_{\gamma,y}-\mathcal I_s\big[G_s(y,\cdot),\psi_{\gamma,y}\big]\Bigg\}dz$}\nonumber\\
    &=\lim_{\gamma\to 0^+}\lim_{\mu\to 0^+}\int_{B_{\varepsilon_0}(y)}\partial_i\Big(\psi_{\mu,x}G_s(x,\cdot)\Big)(z) \Bigg\{G_s(y,\cdot)(-\Delta)^s\psi_{\gamma,y}-\mathcal I_s\big[G_s(y,\cdot),\psi_{\gamma,y}\big]\Bigg\}dz\nonumber\\
    &=\lim_{\gamma\to 0^+}\int_{B_{\varepsilon_0}(y)}\partial_iG_s(x,\cdot)(z) \Bigg\{G_s(y,\cdot)(-\Delta)^s\psi_{\gamma,y}-\mathcal I_s\big[G_s(y,\cdot),\psi_{\gamma,y}\big]\Bigg\}dz. \label{findeli}
\end{align}
Using the scaling law \eqref{rescapsi} (with $\gamma$ instead of $\mu$),  we get 
\begin{align}\label{sec9-8}
   &-\int_{B_{\varepsilon_0}(y)}\frac{\partial  G_s}{\partial z_i}(x,z) G_s(y,z)(-\Delta)^s\psi_{\gamma,y}(z)dz,\nonumber\\
   &=\widetilde\gamma^{-2s}\int_{B_{\varepsilon_0}(0)}G_s(y,y-z)\frac{\partial  G_s}{\partial z_i}(x,y-z)(-\Delta)^s\big(\rho\circ |\cdot|^2\big)\big(\frac{z}{\widetilde\gamma}\big)dz,\nonumber\\
   &=\widetilde\gamma^{N-2s}\int_{B_{\frac{\varepsilon_0}{\widetilde\gamma}}(0)}G_s(y,y-\widetilde\gamma z)\frac{\partial  G_s}{\partial z_i}(x,y-\widetilde\gamma z)(-\Delta)^s\big(\rho\circ |\cdot|^2\big)(z)dz,\nonumber\\
   &=\int_{B_{\frac{\varepsilon_0}{\widetilde\gamma}}(0)}\widetilde\gamma^{N-2s}G_s(y,y-\widetilde\gamma z)\Big[\frac{\partial  G_s}{\partial z_i}(x,y-\widetilde\gamma z)-\frac{\partial  G_s}{\partial z_i}(x,y)\Big](-\Delta)^s\big(\rho\circ |\cdot|^2\big)(z)dz,\nonumber\\
   &\;\;\;+\frac{\partial  G_s}{\partial y_i}(x,y)\int_{B_{\frac{\varepsilon_0}{\widetilde\gamma}}(0)}\widetilde\gamma^{N-2s}G_s(y,y-\gamma z)(-\Delta)^s\big(\rho\circ |\cdot|^2\big)(z)dz.
\end{align}
By the choice of $\varepsilon_0$ and since that $G_s(x,\cdot)\in C^\infty_{loc}(\Omega\setminus\{x\})$, we have 
$$
\Big|\frac{\partial  G_s}{\partial z_i}(x,y-\widetilde\gamma z)-\frac{\partial  G_s}{\partial z_i}(x,y)\Big|\leq \widetilde\gamma C(\varepsilon_0)|z|\quad\forall\,z\in B_{\frac{\varepsilon_0}{\widetilde\gamma}}(0).
$$
Using this and that $\big|\widetilde\gamma^{N-2s}G_s(y,y-\gamma z)\big|\leq C|z|^{2s-N}$, we get 
\begin{align}\label{sec9-9}
   &\Bigg|\int_{B_{\frac{\varepsilon_0}{\widetilde\gamma}}(0)}\widetilde\gamma^{N-2s}G_s(y,y-\widetilde\gamma z)\Big[\frac{\partial  G_s}{\partial z_i}(x,y-\widetilde\gamma z)-\frac{\partial  G_s}{\partial z_i}(x,y)\Big](-\Delta)^s\big(\rho\circ |\cdot|^2\big)(z)dz\Bigg|\nonumber\\
   &\leq C(\varepsilon_0)\widetilde\gamma\int_{\R^N}\frac{dz}{|z|^{N-2s-1}(1+|z|^{N+2s})}\to 0\;\;\text{as}\;\; \gamma\to 0^+.
\end{align}
On the one hand, we know by \eqref{sec7-8} that
\begin{equation}\label{sec9-10}
   \lim_{\gamma\to 0^+}\int_{B_{\frac{\varepsilon_0}{\widetilde\gamma}}(0)}\widetilde\gamma^{N-2s}G_s(y,y-\widetilde\gamma z)(-\Delta)^s\big(\rho\circ |\cdot|^2\big)(z)dz= 1.
\end{equation}
Passing into the limit in \eqref{sec9-8} and taking into account \eqref{sec9-9} and \eqref{sec9-10}, we end up with 
\begin{equation}\label{sec9-11}
 \lim_{\gamma\to 0^+}\int_{B_{\varepsilon_0}(y)}\frac{\partial  G_s}{\partial z_i}(x,z)G_s(y,z)(-\Delta)^s\psi_{\gamma,y}(z)dz =-\frac{\partial G_s}{\partial y_i}(x,y).   
\end{equation}
A similar argument as in the proof of \eqref{sl} above yields also
\begin{align}\label{sec9-12}
 &\lim_{\gamma\to 0^+}\int_{\Omega}\frac{\partial  G_s}{\partial z_i}(x,z)\mathcal I_s\big[G_s(y,z),\psi_{\gamma,y}\big](z)dz\nonumber\\
 &=\lim_{\gamma\to 0^+}\int_{B_{\varepsilon_0}(y)}\frac{\partial  G_s}{\partial z_i}(x,z)\mathcal I_s\big[G_s(y,z),\psi_{\gamma,y}\big](z)dz\nonumber\\
 &=-2\frac{\partial G_s}{\partial y_i}(x,y).   
\end{align}

In view of \eqref{dedede}, \eqref{findeli}, \eqref{sec9-11} and \eqref{sec9-12}, we are done once we know that 
\begin{align*}
  &\lim_{\gamma\to 0^+}\lim_{\mu\to 0^+}\lim_{k\to\infty} \mathcal T^2_{k,\mu,\gamma}(x,y)\\
 &=\lim_{\gamma\to 0^+}\lim_{\mu\to 0^+}\lim_{k\to\infty} 2k\int_{\Omega}\eta_k(z)\rho'(k\delta)\partial_i\delta(z)\Big(\psi_{\mu,x}G_s(x,\cdot)\Big) (-\Delta)^s\big(\psi_{\gamma,y}G_s(z,\cdot)\big)dz\\
 &=0.
\end{align*}
The later follows by a similar argument as in \eqref{pf}. The proof of Lemma \ref{lem-lim-G-x-y} is therefore finished.

\subsection{Proof of  \eqref{repres-gradient-Robin func}.}

We recall the decomposition $G_s(x,y)=\frac{b_{N,s}}{|x-y|^{N-2s}}-H_s(x,y)$ with $H_s(\cdot,\cdot)\in C^\infty(\Omega\times\Omega)$. Since  
$$\frac{\partial(|x-y|^{2s-N})}{\partial x_i}+\frac{\partial(|x-y|^{2s-N})}{\partial y_i}=0,$$ 
the identity \eqref{partial-Green} becomes: for any $(x,y)$ in $\Omega\times\Omega$, with $x\neq y$,
$$
\frac{\partial H_s}{\partial x_i}(y,x)+\frac{\partial H_s}{\partial y_i}(x,y)= - \Gamma^2(1+s)\int_{\partial \Omega} \frac{G_s(x,z)}{\delta^s_\Omega (z)} \frac{G_s(y,z)}{\delta^s_\Omega (z)} \nu_i(z)d\sigma(z). 
$$
But now that we got rid of the singular part, we can let $y$ go to $x$, and since 
$$  \partial_i \mathcal R_s(x) = \frac{\partial H_s}{\partial x_i}(x,x)+\frac{\partial H_s}{\partial y_i}(x,x),$$
we arrive at \eqref{repres-gradient-Robin func}.

\section{Proof of Corollary \ref{non-deg-frac-Robin}}
\label{sec-robin}

This section is devoted to the proof of Corollary \ref{non-deg-frac-Robin}.
The proof is an adaptation of the proof given in \cite{G} regarding the corresponding problem for the classical case of the Laplacian. For $j\in \{1,2,\cdots, N\}$, 
 we  set
$$H_j:=\big\{x\in\R^N:x_j=0\big\} \quad \text{ and } \quad 
\Omega_j := H_j \cap \Omega,$$ 
 and we assume that  $\Omega$ is symmetric with respect to the hyperplane $H_j$ and that $0$ is in $\Omega$. 
We call  $T_j:\R^N\to\R^N$  the reflection with respect to the hyperplane $H_j$.
We start with the following simple lemma regarding the symmetry of the fractional Green function.
\begin{Lemma}
    Fix $j\in\{1,2,\cdots, N\}$ and let $\overline y\in \Omega_j$. Then we have 

    \begin{equation}\label{lab-claim}
    G_s(\overline y, T_j(x))=G_s(\overline y, x)\quad \text{for a.e $x\in \R^N$}.
    \end{equation}
\end{Lemma}
\begin{proof}
    We have, on the one hand, that
    \begin{align}\label{lab1}
        \int_{\Omega}G_s(\overline y, z)(-\Delta)^s\phi(z)dz=\phi(\overline y)\quad\forall\, \phi\in C^\infty_c(\Omega).
    \end{align}
    On the other hand, 
   using the symmetry of  $\Omega$ and the definition of the fractional Laplacian, we have that for any $\phi\in C^\infty_c(\Omega)$, for any $z\in \Omega$, 
   $$(-\Delta)^s\phi(T_j(z))= (-\Delta)^s(\phi\circ T_j)(z)=\phi(\overline z). $$
   By changing variables, and using \eqref{lab1}  with $\phi\circ T_j$ instead of $\phi$, 
   we deduce that 
    \begin{align}\label{lab2}
        \int_{\Omega}G_s(\overline y, T_j(z))(-\Delta)^s\phi(z)dz =\phi(\overline y) .
    \end{align}
    Comparing \eqref{lab1} and \eqref{lab2}, and by a density argument, we get the result.
\end{proof}
\par\;
Now, for all $\overline y\in \Omega_j$, by Theorem \ref{Thm.2}, 
\begin{align}
 \partial_{y_j}\mathcal R_s(\overline y)  &=\Gamma^2(1+s)\int_{\partial\Omega}\Big(\frac{G_s(\overline y,\cdot)}{\delta^s_\Omega}\Big)^2(\sigma)\nu_j(\sigma)d\sigma \label{lab-fin}    .
\end{align}
Using the symmetry of  $\Omega$, \eqref{lab-claim} and that the normal vector satisfies 
 $\nu_j(T_j(\sigma))=-\nu_j(\sigma)$, we deduce that 
 $\partial_{y_j}\mathcal R_s(\overline y) =0$.
  If moreover, there is $i\in \{1,2,\cdots, N\}$, with $i \neq j$, such that $\Omega$ is also symmetric with respect to the hyperplane $H_i$, then one may differentiate the previous identity in the direction $i$ so that $\partial_{ij}  \mathcal R_s(\overline y)=0$.
\section{Proof of Theorem \ref{thm.3} }\label{sec-thm4}

This section is devoted to the proof of Theorem \ref{thm.3}.
In Section \ref{ssec-1} we establish a preliminary result on the shape derivative. 
In Section \ref{ssec-2} we apply an integration by parts formula to 
the shape derivative $u'$ and a regularization, with the two parameters  $\mu$ and $k$, of the Green function $ G_s(x,\cdot)$. 
In Section \ref{ssec-3} we pass to the double limit, as $\mu\to 0^+$ and $k\to\infty$, in the identity above. 
This provides the result, aka Theorem \ref{thm.3}, up to the determination of a constant. 
In Section \ref{sec10.2} we establish the value of this constant by considering explicit computation in the $1$D case.

\subsection{A preliminary result on the shape derivative}\label{ssec-1}
For an arbitrary globally Lipschitz vector field $Y$, we shall denote by $\mathcal K_Y$ the deformation kernel defined by 
\begin{equation} \label{defok}
\mathcal K_Y(x,y)=\frac{c_{N,s}}{2}\Bigg\{\div Y(x)+\div Y(y)-(N+2s)\frac{(Y(x)-Y(y))\cdot(x-y)}{|x-y|^2}\Bigg\}|x-y|^{-N-2s}.
\end{equation}
We note that, since $Y\in C^{0,1}(\R^N,\R^N)$, we have $|\mathcal K_Y(x,y)|\leq C|x-y|^{-2s-N}$. For $u\in H^s(\R^N)$ we defined $g_Y[u]\in (H^s(\R^N))'$ by 
\begin{align}\label{def-gu}
    \big<g_Y[u],\phi\big>=\iint_{\R^N\times\R^N}(u(x)-u(y))(\phi(x)-\phi(y))\mathcal K_Y(x,y)dxdy\quad\text{for all $\phi\in H^s(\R^N)$}.
\end{align}
Note that we have 
\begin{equation}\label{p00}
\| g[u] \|_{ (H^s(\R^N))'} \leq C  [ u ]_{ H^s(\R^N)} ,
\end{equation}
for some positive constant $C$. We start with the following result.
\begin{Lemma}\label{App-1}
    Let $h\in W^{1,\infty}(D)$ with $\Omega\Subset D$ and $u_t$ be the weak solution of the Dirichlet problem:
\begin{equation}\label{purturbed-Dirich-pb-00}
\left\{ \begin{array}{rcll} (-\Delta)^s u_t&=& h  &\textrm{in }\;\;\Omega_t , \\ u_t&=&0&
\textrm{in }\;\;\R^N\setminus\Omega_t \end{array}\right. ,
\end{equation}
where $\Omega_t:=\phi_t(\Omega)$ and $\phi_t$ is given as in \eqref{transformations}. Defined 
\begin{equation}
u':=v'-\nabla u_0\cdot Y\;\; \text{where}\;\;v':=\partial_t(u_t\circ\phi_t)_{\big|t=0}\in \mathcal H^s_0(\Omega).
\end{equation}
Then we have  
\begin{equation}\label{equ-u'}
(-\Delta)^s u'=0\;\;\text{in}\;\;\mathcal D'(\Omega) .
\end{equation}
If moreover, $h\in C^\mu_{loc}(\Omega)$ for some $\mu\in (0,1)$, then 
\begin{equation}\label{equ-v'}
    (-\Delta)^s v'=h\div Y+\nabla h\cdot Y-g_Y[u_0] \quad\text{weakly in \;\;} \mathcal H^s_0(\Omega)
\end{equation}
where $g_Y[u_0]\in (H^s(\R^N))'$ is given as in \eqref{def-gu}. Furthermore, there holds:

\begin{equation}\label{norm-v'-estim}
  \|v'\|_{H^s(\R^N)}\leq C\|u_0\|_{H^s(\R^N)} + C \|h\|_{H^{s+1} (\R^N)}
 \end{equation}
 for some $C=C(\Omega,N,s)>0$.
\end{Lemma}
\begin{proof}
To see \eqref{equ-u'}, we first note that under the assumptions \eqref{transformations}, the function $\phi_t:\R^N\to\R^N$ is a diffeomorphism for $|t|>0$ sufficiently small. Next, let $\phi\in C^\infty_c(\Omega)$ and use $\phi\circ \phi_t^{-1}\in C^{1,1}_c(\Omega_t)$ as a test function in \eqref{purturbed-Dirich-pb-00}. This gives after changing variables 
\begin{align}
    &\frac{c_{N,s}}{2}\iint_{\R^N\times\R^N}(v_t(x)-v_t(y))(\phi(x)-\phi(y))\frac{\textrm{Jac}_{\phi_t}(x)\textrm{Jac}_{\phi_t}(y)}{|\phi_t(x)-\phi_t(y)|^{N+2s}}dxdy\nonumber\\ \label{iden}
    &=\int_{\Omega}\phi(x)(h\circ \phi_t)(x)\textrm{Jac}_{\phi_t}(x)dx
\end{align}
Here $v_t:=u_t\circ \phi_t$ and $\textrm{Jac}_{\phi_t}$ is the Jacobian of the transformation $\phi_t$. 
We introduce the symmetric $2$-point kernel:
 \begin{equation}
  \label{knono}
   \mathcal K_t(x,y) :=\frac{c_{N,s}}{2}\frac{\textrm{Jac}_{\phi_t}(x)\textrm{Jac}_{\phi_t}(y)}{|\phi_t(x)-\phi_t(y)|^{N+2s}}\quad\text{for $y\neq y\in\R^N$}.
\end{equation}
We recall that by the Liouville theorem, 
$\partial_t \textrm{Jac}_{\phi_t}(x)\big|_{t=0} = \textrm{div }Y(x)$, so that by the Leibniz identity, 
\begin{align*}
 \partial_t \mathcal K_t(x,y) 
 \big|_{t=0} = \mathcal K_Y(x,y).
\end{align*}
Using that $t\mapsto v_t\in \mathcal H^s_0(\Omega)$ is differentiable at zero and differentiating the equation \eqref{iden} we thus get that 
\begin{align}\label{App-2}
 & \iint_{\R^N\times\R^N}(u_0(x)-u_0(y))(\phi(x)-\phi(y))\mathcal K_Y(x,y)dxdy\nonumber\\
    &+\frac{c_{N,s}}{2}\iint_{\R^N\times\R^N}(v'(x)-v'(y))(\phi(x)-\phi(y))\frac{dxdy}{|x-y|^{N+2s}}\nonumber\\
    &=\int_{\Omega}\phi(x)h(x)\textrm{div}Y(x)dx +\int_{\Omega}\phi(x)\nabla h(x)\cdot Y(x)dx 
\end{align}

Let us recast the first term in the left hand side of \eqref{App-2}.
Noting that
\begin{equation}\label{App-4}
    u_0\eta_k\to u_0\;\;\text{in}\;\;\mathcal H^s_0(\Omega) \;\;\text{as}\;\;k\to\infty
\end{equation}
    (see \cite[Lemma 2.2]{DFW}) and $|\mathcal K_Y(x,y)|\leq C|x-y|^{2s-N}$, we have 
\begin{align}\label{App-5}
  &  \iint_{\R^N\times\R^N}(u_0(x)-u_0(y))(\phi(x)-\phi(y))\mathcal K_Y(x,y)dxdy\nonumber\\  
  &=\lim_{k\to\infty}\iint_{\R^N\times\R^N}(u_0(x)\eta_k(x)-u_0(y)\eta_k(y))(\phi(x)-\phi(y))\mathcal K_Y(x,y)dxdy.
\end{align}
Denoting 
$$(-\Delta)^s_\mu w(x) :=c_{N,s}\int_{|x-y|>\mu}\frac{w(x)-w(y)}{|x-y|^{N+2s}}dy,$$ 
and following closely the proof of \cite[Lemma 4.2]{DFW}, we integrate by parts to get
\begin{align}\label{App-6}
   &\iint_{\R^N\times\R^N}(u_0(x)\eta_k(x)-u_0(y)\eta_k(y))(\phi(x)-\phi(y))\mathcal K_Y(x,y)dxdy\nonumber\\
    &=-\lim_{\mu\to 0^+}\int_{\Omega}\nabla(u_0\eta_k)\cdot Y(-\Delta)^s_\mu\phi\,dx-\lim_{\mu\to 0^+}\int_{\Omega}\nabla\phi\cdot Y(-\Delta)^s_\mu(u_0\eta_k)\,dx\nonumber\\
    &=-\int_{\Omega}\nabla(u_0\eta_k)\cdot Y(-\Delta)^s\phi\,dx\nonumber\\
    &\;\;\;\;-\frac{c_{N,s}}{2}\iint_{\R^N\times\R^N}\frac{((\nabla\phi\cdot Y)(x)-(\nabla\phi\cdot Y)(y))(u_0(x)\eta_k(x)-u_0(y)\eta_k(y))}{|x-y|^{N+2s}}dxdy.
\end{align}
By definition of $\eta_k$, it is not difficult to see that 
\begin{equation}\label{App-8}
\lim_{k\to\infty}\int_{\Omega}\nabla(u_0\eta_k)\cdot Y(-\Delta)^s\phi\,dx=\int_{\Omega}\nabla u_0\cdot Y(-\Delta)^s\phi\,dx.
\end{equation}
Moreover, by \eqref{App-4}, we have 
\begin{align}\label{App-7}
 \lim_{k\to\infty}&\frac{c_{N,s}}{2}\iint_{\R^N\times\R^N}\frac{((\nabla\phi\cdot Y)(x)-(\nabla\phi\cdot Y)(y))(u_0(x)\eta_k(x)-u_0(y)\eta_k(y))}{|x-y|^{N+2s}}dxdy\nonumber\\
 &=\frac{c_{N,s}}{2}\iint_{\R^N\times\R^N}\frac{((\nabla\phi\cdot Y)(x)-(\nabla\phi\cdot Y)(y))(u_0(x)-u_0(y))}{|x-y|^{N+2s}}dxdy\nonumber\\
 &=\int_{\Omega}h(x)\nabla \phi(x)\cdot Y(x)\,dx,
\end{align}
using that $u_0$ is the weak solution of the Dirichlet problem 
\begin{equation}\label{purturbed-Dirich-pb-0000}
\left\{ \begin{array}{rcll} (-\Delta)^s u_0&=& h  &\textrm{in }\;\;\Omega , \\ u_0&=&0&
\textrm{in }\;\;\R^N\setminus\Omega , \end{array}\right. 
\end{equation}
with $\nabla\phi\cdot Y\in C^1_c(\Omega)$ as a test function.

Therefore, combining \eqref{App-5}, \eqref{App-6}, \eqref{App-8} and \eqref{App-7}, we arrive at
\begin{align}\label{Apppp}
   \iint_{\R^N\times\R^N}(u_0(x)-u_0(y))(\phi(x)-\phi(y))\mathcal K_Y(x,y)dxdy
   = - \int_{\Omega}\nabla u_0\cdot Y(-\Delta)^s\phi\,dx
 \\ \nonumber  - \int_{\Omega}h(x)\nabla \phi(x)\cdot Y(x)\,dx .
\end{align}

On the other hand, since $\phi\in C^\infty_c(\Omega)$, it is clear that the second term in the left hand side of \eqref{App-2} satisfies:
\begin{equation}\label{App-3}
\frac{c_{N,s}}{2}\iint_{\R^N\times\R^N}(v'(x)-v'(y))(\phi(x)-\phi(y))\frac{dxdy}{|x-y|^{N+2s}}=\int_{\Omega}v'(x)(-\Delta)^s\phi(x)dx.    
\end{equation}

Combining \eqref{App-2}, \eqref{Apppp} and \eqref{App-3} 
we end up with 
\begin{align*}
  &-\int_{\Omega}\nabla u_0\cdot Y(-\Delta)^s\phi\,dx- \int_{\Omega}h\nabla \phi\cdot Y\,dx+ \int_{\Omega}v'(-\Delta)^s\phi dx\nonumber\\
  &=\int_{\Omega}\phi h \textrm{div}Ydx +\int_{\Omega}\phi\nabla h\cdot Ydx .
\end{align*}
Since  $\phi\in C^\infty_c(\Omega)$, by integration by parts, 
\begin{equation*}
 - \int_{\Omega}h\nabla \phi\cdot Y\,dx =\int_{\Omega}\phi h \textrm{div}Ydx +\int_{\Omega}\phi\nabla h\cdot Ydx ,
\end{equation*}
and therefore, by substraction, we arrive at:
\begin{equation}
    \label{jeu}
    \int_{\Omega}\big(v'-\nabla u_0\cdot Y\big)(-\Delta)^s\phi\,dx=0.
\end{equation} 
This finishes the proof of \eqref{equ-u'}. 

Since $\phi\in C^\infty_c(\Omega)$, by \cite[Theorem 1.3]{DFW2023}, we have 
\begin{align*}
\int_{\Omega}\nabla u_0\cdot Y(-\Delta)^s\phi\,dx&=-\int_{\Omega}\nabla\phi\cdot Y(-\Delta)^su_0\,dx-\big<g_Y[u_0],\phi\big>
\nonumber\\
&=\int_{\Omega}(h\div Y+\nabla h\cdot Y)\phi\,dx-\big<g_Y[u_0],\phi\big>,
\end{align*}
where we used \eqref{purturbed-Dirich-pb-0000} and an integration by parts. 
Combined with \eqref{jeu}, this yields \eqref{equ-v'}.
And finally, \eqref{norm-v'-estim} follows by using $v'$ as a test function in \eqref{equ-v'} and that, by 
\eqref{p00}, we have:
\begin{equation*}
\Big|   \int_\Omega \Big( (h\div Y+\nabla h\cdot Y)-g_Y[u_0] \Big) v' \Big| 
\leq   C \big( \|h\|_{H^{s+1} (\R^N)} +  \|u_0\|_{H^s(\R^N)}  \big)  \|v'\|_{H^s(\R^N)}.
\end{equation*}
The proof of Lemma \ref{App-1} is finished.
\end{proof}

\subsection{Regularization step}\label{ssec-2}
Let start with the following observation. Let $u'\in L^1(\Omega)$ be defined as in \eqref{shape-deriv}, then $u'$ solves  
\begin{equation}\label{Eq-shape-deriv-u}
    (-\Delta)^s u'=0\;\; \text{in}\;\; \mathcal D'(\Omega)\quad\&\quad u'/\delta^{s-1}_\Omega=-s(u_0/\delta^s_\Omega)Y\cdot\nu\;\;\text{on}\;\;\partial\Omega.
\end{equation}
Indeed, by \cite{FallSven} we know  
\begin{equation}\label{key-reg-estimate}
\delta^{1-\alpha}_\Omega\nabla (u_0/\delta^s_\Omega)\;\;\text{is bounded near the boundary $\partial\Omega$ for some $\alpha\in (0,1)$.}
\end{equation}
Using this, it is straightforward to check that the second identity in \eqref{Eq-shape-deriv-u} is verified. The first equation is proved in Lemma \ref{App-1} above. 
\par\;

To get \eqref{repres-shape-deriv}, one might be tempted to use $u'$ and $G_s(x,\cdot)$ as test functions in the following identity (see e.g \cite{Grubb-2020}) which holds for sufficiently regular functions $u,v$, with $u$ satisfying a homogeneous Dirichlet boundary condition, 
\begin{equation}\label{degrubb}
\int_{\Omega}u(-\Delta)^sv\,dz+\int_{\Omega}v(-\Delta)^su\,dz=\Gamma(s)\Gamma(1+s)\int_{\partial\Omega}\frac{u}{\delta^s_\Omega}\frac{v}{\delta^{s-1}_\Omega}d\sigma.
\end{equation}
However, it is not clear that $u'=v'-\nabla u\cdot Y$ and $G_s(x,\cdot)$ are sufficiently regular to be admissible in the identity above. To get around this difficulty, we use again an approximation techniques as above. Or precisely we use $\eta_k G_s(x,\cdot)\psi_{\mu,x}\in C^\infty_c(\Omega)$ as a test function in the first equation of \eqref{Eq-shape-deriv-u}  
to get
\begin{align*}
    0&=\int_{\Omega}u'(-\Delta)^s\big[\eta_kG_s(x,\cdot)\psi_{\mu,x}\big]dz.
    \end{align*}
    Then, using the product law \eqref{pl} to 
 $ \eta_k$ and $G_s(x,\cdot)\psi_{\mu,x} $ we obtain that 
   \begin{align}\label{sec6-key-idd}
   0 &=\int_{\Omega}u'\Big\{G_s(x,\cdot)\psi_{\mu,x}(-\Delta)^s\eta_k-\mathcal I_s[\eta_k, G_s(x,\cdot]\psi_{\mu,x})\Big\}+\int_{\Omega}u'\eta_k(-\Delta)^s(G_s(x,\cdot)\psi_{\mu,x})dz.
\end{align}

\subsection{Passing to the limit}\label{ssec-3}

We now pass to the double limit, as $\mu\to 0^+$ and $k\to\infty$, in the identity above. 
On the one hand, a similar argument as in Lemma \ref{Lem4.2} gives 
\begin{align}\label{sec6-reprens-u'}
 \lim_{\mu\to 0^+}\lim_{k\to\infty} \int_{\Omega}u'\eta_k(-\Delta)^s(G_s(x,\cdot)\psi_{\mu,x})dz=u'(x).
\end{align}
We note that since $u'\in \mathcal L^1_s(\R^N)$ solves \eqref{equ-u'}, then $u'\in C(\Omega)$ by regularity theory and therefore satisfies the assumption of Lemma \ref{Lem4.2}.

On the other hand, we have the following result. 
 \begin{Lemma}\label{lemma-other}
 Let $x\in\Omega$. Then we have: 
 \begin{align} 
    &\lim_{\mu\to 0^+}\lim_{k\to\infty}\int_{\Omega}u'\Big\{G_s(x,\cdot)\psi_{\mu,x}(-\Delta)^s\eta_k-\mathcal I_s[\eta_k, G_s(x,\cdot)\psi_{\mu,x}]\Big\}dz\nonumber\\ \label{sec6-9}
 &=d_s\int_{\partial\Omega}\big(u/\delta^s_\Omega\big)(\sigma)\big(G_s(x,\cdot)/\delta^s_\Omega\big)(\sigma)Y\cdot\nu(\sigma)d\sigma.
\end{align}
where $\nu$ denotes the outward unit normal to the boundary and 
\begin{equation}
    \label{ds} d_s   :=   s\int_{0}^\infty \Big(r^s(-\Delta)^s\eta-\tilde{I}_s(r)\Big)dr,
\end{equation}
with
\begin{align}
    \tilde{I}_s(r):=c_{1,s}\int_{\R}\frac{\big(r^s_+-(r+t)^s_+\big)\big(\eta(r)-\eta(r+t)\big)}{|t|^{1+2s}}dr.
\end{align}
 \end{Lemma}
 \begin{proof}
First, arguing as in \cite[Page 16]{DFW} we have that for all $\varepsilon>0$, 
\begin{align}
    &\lim_{k\to\infty}\int_{\Omega}u'\Big\{G_s(x,\cdot)\psi_{\mu,x}(-\Delta)^s\eta_k-\mathcal I_s[\eta_k, G_s(x,\cdot)\psi_{\mu,x}]\Big\}dz\nonumber\\ \label{arguing}
    &=\lim_{k\to\infty}\int_{\Omega^\varepsilon_{+}}u'\Big\{G_s(x,\cdot)\psi_{\mu,x}(-\Delta)^s\eta_k-\mathcal I_s[\eta_k, G_s(x,\cdot)\psi_{\mu,x}]\Big\}dz,
\end{align}
where we denote 
$$
\Omega^\varepsilon_{+}:=\big\{ x\in\Omega: 0<\delta(x)<\varepsilon\big\}.
$$
Let $\Psi$ be the transformation defined by (see \cite{DFW})
\begin{align}\label{Psi-trans}
    \Psi:\partial\Omega\times (0,\varepsilon)\to \Omega^\varepsilon_+, \;(\sigma,r)\mapsto \sigma+r\nu(\sigma).
\end{align}
To be short we set
\begin{equation}\label{defire}
    g_{k,\mu,x}(z):=u'(z)\Big\{G_s(x,z)\psi_{\mu,x}(z)(-\Delta)^s\eta_k(z)-\mathcal I_s[\eta_k, G_s(x,\cdot)\psi_{\mu,x}](z)\Big\}.
\end{equation}
Changing variables by using \eqref{Psi-trans}, we get 
\begin{align}
    \lim_{k\to\infty}\int_{\Omega^\varepsilon_+}g_{k,\mu,x}(z)dz=\lim_{k\to\infty}\frac{1}{k}\int_{\partial\Omega}\int_{0}^{k\varepsilon}j_k(r,\sigma)g_{k,\mu,x}(r,\sigma)drd\sigma
\end{align}
with the slight abuse of notation
$$
g_{k,\mu,x}(r,\sigma)=g_{k,\mu,x}\big(\Psi(\frac{r}{k},\sigma)\big), \quad \mathcal M_s(x,\cdot)=\frac{G_s(x,\cdot)}{\delta_\Omega^s}\quad\text{and}\quad
j_k(r,\sigma)=\textrm{Jac}_{\Psi}(\frac{r}{k},\sigma).
$$
We decompose 
\begin{equation}  \label{decog}
   g_{k,\mu,x}(r,\sigma) 
    =g_{k,\mu,x}^1(r,\sigma)+g_{k,\mu,x}^2(r,\sigma).
\end{equation}
with
\begin{align*}
  g_{k,\mu,x}^1(r,\sigma) &=
v'(\frac{r}{k},\sigma)  \Big\{\psi_{\mu,x}(\frac{r}{k},\sigma)\mathcal M_s(x,\cdot)(\frac{r}{k},\sigma)(\frac{r}{k})^s(-\Delta)^s\eta_k(\frac{r}{k},\sigma)-\mathcal I_s[\eta_k; G_s(x,\cdot)\psi_{\mu,x}](\frac{r}{k},\sigma)\Big\}\\
   g_{k,\mu,x}^2(r,\sigma) &= \Big(-\nabla u\cdot Y\big(\frac{r}{k},\sigma\big)\Big)
\\ &\quad \times  \Big\{\psi_{\mu,x}(\frac{r}{k},\sigma)\mathcal M_s(x,\cdot)(\frac{r}{k},\sigma)(\frac{r}{k})^s(\frac{r}{k})^s(-\Delta)^s\eta_k(\frac{r}{k},\sigma)-\mathcal I_s[\eta_k, G_s(x,\cdot)\psi_{\mu,x}](\frac{r}{k},\sigma)\Big\} .
\end{align*}
\begin{Claim}There exist $\varepsilon'>0$ such that \label{spi1}
 \begin{equation*}
\lim_{k\to\infty}\frac{1}{k}\int_{\partial\Omega}\int_{0}^{k\varepsilon}j_k(r,\sigma)g^1_{k,\mu,x}(r,\sigma)drd\sigma
= 0\quad \text{for}\; \varepsilon\in (0,\varepsilon').
\end{equation*} 
 \end{Claim}
\begin{proof}
On the one hand, recalling \cite[Proposition 6.3]{DFW} we know there exists $\varepsilon'>0$ such that 
\begin{align}\label{sec6-4}
    \Big|k^{-2s}(-\Delta)^s\eta_k(\Psi(\frac{r}{k},\sigma))\Big|\leq \frac{C}{1+r^{1+2s}}\quad\text{for $k\in \N,\; 0<r<k\varepsilon'$}
\end{align}
and 
\begin{align}\label{sec6-5}
 \lim_{k\to\infty}k^{-2s}(-\Delta)^s\eta_k(\Psi(\frac{r}{k},\sigma))=  (-\Delta)^s\eta(r)\;\;\text{for $\sigma\in\partial\Omega$ and $r>0$}
\end{align}
with $\eta(r)=1-\rho(r)$. On the one hand, replacing $u$ by $G_s(x,\cdot)\psi_{\mu,x}$ into \cite[Lemma 6.8]{DFW} we also have 
\begin{align}\label{sec6-6}
 \Big|k^{-s}\mathcal I_s[\eta_k, G_s(x,\cdot)\psi_{\mu,x}](\Psi(\frac{r}{k},\sigma))\Big|  \leq \frac{C}{1+r^{1+s}}\quad\text{for $k\in \N,\; 0\leq r<k\varepsilon'$}
\end{align}
and 
\begin{align}\label{sec6-7}
\lim_{k\to\infty}k^{-s}\mathcal I_s[\eta_k, G_s(x,\cdot)\psi_{\mu,x}](\Psi(\frac{r}{k},\sigma))=    \psi_{\mu,x}(\sigma) \Big(G_s(x,\cdot)/\delta^s_\Omega\Big)(\sigma)\tilde{I}_s(r).
\end{align}
Combining \eqref{sec6-4} and \eqref{sec6-6} we get \begin{align*}
   &\frac{1}{k}\big|g_{k,\mu,x}^1(r,\sigma)\big|\nonumber\\
   &=\frac{1}{k}\small\text{$\Bigg|v'(\frac{r}{k},\sigma)\Bigg\{\psi_{\mu,x}(\frac{r}{k},\sigma)\mathcal M_s(x,\cdot)(\frac{r}{k},\sigma)(\frac{r}{k})^s(-\Delta)^s\eta_k(\frac{r}{k},\sigma)-\mathcal I_s[\eta_k, G_s(x,\cdot)\psi_{\mu,x}](\frac{r}{k},\sigma)\Bigg\}\Bigg|$}\nonumber\\
   &\;\; \leq \frac{C}{k}|v'(\frac{r}{k},\sigma)|\Bigg\{\frac{k^s}{1+r^{1+2s}}\frac{|\psi_{\mu,x}(\frac{r}{k},\sigma)|}{|\Psi(\frac{r}{k},\sigma)-x|^N}+\frac{k^s}{1+r^{1+s}}\Bigg\}\nonumber\\
   &\qquad\leq \frac{C}{k^{1-s}}\big|v'(\frac{r}{k},\sigma)\big|\Bigg\{\frac{\mu^{-N}}{1+r^{1+2s}}+\frac{1}{1+r^{1+s}}\Bigg\},
\end{align*}
where in the last line we used that $\frac{|\psi_{\mu,x}(\frac{r}{k},\sigma)|}{|\Psi(\frac{r}{k},\sigma)-x|^N}\leq C\mu^{-N}$ which follows from the definition of $\psi_{\mu,x}$. Consequently, since $|j_k(r,\sigma)|\leq C$, we have by the H\"older inequality
\begin{align*}
  \Big|\int_{\partial\Omega}\int_{0}^{k\varepsilon'}  \frac{1}{k}j_k(r,\sigma)g_{k,\mu,x}^1(r,\sigma)drd\sigma\Big|&\leq \frac{C}{k^{1-s}}\Big(\int_{\partial\Omega}\int_{0}^\infty\big|v'(\frac{r}{k},\sigma)\big|^2d\sigma\Big)^{1/2}\,dr\nonumber\\
  &\leq \frac{C}{k^{1-s}}\Big(\int_{\R^N}|v'(z)|^2dz\Big)^{1/2}\;\;\to \;\;0
\end{align*}
as $k\to\infty$ which proves Claim \ref{spi1}.\end{proof}

\begin{Claim}\label{spi2}
Let $\varepsilon'>0$ be given as in Claim \eqref{spi1}. Then there holds
\begin{align}
  &\lim_{\mu\to 0^+}\lim_{k\to\infty}\frac{1}{k}\int_{\partial\Omega}\int_{0}^{k\varepsilon'}j_k(r,\sigma)g_{k,\mu,x}^2(r,\sigma)drd\sigma
  =d_s\int_{\partial\Omega}
  \frac{u}{\delta^s_\Omega} (\sigma)
  \frac{G_s(x,\cdot)}{\delta^s_\Omega}  (\sigma)
 Y\cdot\nu(\sigma)d\sigma.
\end{align}
\end{Claim}
\begin{proof}
 First using \eqref{key-reg-estimate}, we easily check that
\begin{align*}
   k^{-1}\Big|\nabla u\cdot Y\big(r/k,\sigma\big)\Big|\leq Ck^{-s}\big(r^{s-1}+r^{s-1+\alpha}\big).
\end{align*}
Combining this with \eqref{sec6-4} and \eqref{sec6-6} we get
\begin{align}\label{sec6-estim-g-2}
  &\small\text{$\frac{1}{k}\Bigg|\nabla u\cdot Y\big(\frac{r}{k},\sigma\big)\Bigg\{\psi_{\mu,x}(\frac{r}{k},\sigma)\mathcal M_s(x,\cdot)(\frac{r}{k},\sigma)(\frac{r}{k})^s(-\Delta)^s\eta_k(\frac{r}{k},\sigma)-\mathcal I_s[\eta_k, G_s(x,\cdot)\psi_{\mu,x}](\frac{r}{k},\sigma)\Bigg\} \Bigg|$}\nonumber\\
  &\leq C\big(r^{s-1}+r^{s-1+\alpha}\big) \nonumber\\
  &\quad\quad\quad\quad\times \Bigg\{r^s\frac{|\psi_{\mu,x}(\frac{r}{k},\sigma)|}{|\Psi(\frac{r}{k},\sigma)-x|^N}\Big|k^{-2s}(-\Delta)^s\eta_k(\frac{r}{k},\sigma)\Big|+\Big|k^{-s}\mathcal I_s[\eta_k, G_s(x,\cdot)\psi_{\mu,x}](\frac{r}{k},\sigma)\Big|\Bigg\}\nonumber\\
  &\leq C\big(r^{s-1}+r^{s-1+\alpha}\big)\Bigg\{\mu^{-N}\frac{r^s}{1+r^{1+2s}}+\frac{1}{1+r^{1+s}}\Bigg\}\;\;\text{for $k\in\N, 0<r<k\varepsilon'.$}
\end{align}
In view of \eqref{sec6-estim-g-2}, we have by the dominated convergence theorem that 
\begin{align}\label{sec6-8}
    &\lim_{k\to\infty}\frac{1}{k}\int_{\partial\Omega}\int_{0}^{k\varepsilon'}j_k(r,\sigma)g_{k,\mu,x}^2(r,\sigma)drd\sigma\nonumber\\
    &=\int_{\partial\Omega}\int_{0}^{\infty}\lim_{k\to\infty}\Bigg[\frac{1}{k}j_k(r,\sigma)g_{k,\mu,x}^2(r,\sigma)\Bigg]drd\sigma\nonumber\\
    &=-s\int_{0}^\infty \Big(r^s(-\Delta)^s\eta-\tilde{I}_s(r)\Big)dr\int_{\partial\Omega}\psi_{\mu,x}(\sigma)\big(u/\delta^s_\Omega\big)(\sigma)\big(G_s(x,\cdot)/\delta^s_\Omega\big)(\sigma)Y\cdot\nu(\sigma)d\sigma
\end{align}
where we used \eqref{sec6-5}, \eqref{spi1} and that $\lim_{k\to\infty}\Big(\delta^{1-s}\nabla u\cdot Y\Big)(r/k,\sigma)=s(u_0/\delta^s_\Omega)(\sigma)Y\cdot\nu(\sigma)$.
We conclude the proof of  Claim \ref{spi2} By taking the limit w.r.t $\mu$ into \eqref{sec6-8}.
\end{proof}
Combining 
the identities \eqref{arguing} and 
\eqref{spi1}, the definition \eqref{defire},  the decomposition  \eqref{decog}, 
Claim \ref{spi1}  and Claim \ref{spi2}
 this ends the proof of
Lemma \ref{lemma-other}. 
\end{proof}

We are now ready to complete the proof of Theorem \ref{thm.3}.
    Combining \eqref{sec6-key-idd}, \eqref{sec6-reprens-u'} and \eqref{sec6-9} we end up with the identity:
\begin{align}\label{sec6-10}
u'(x)=d_s\int_{\partial\Omega}
\frac{u_0}{\delta^s_\Omega}
(\sigma) 
\frac{G_s(x,\cdot)}{\delta^s_\Omega}
(\sigma)
Y\cdot\nu(\sigma)d\sigma.
\end{align}
We conclude the proof of Theorem \ref{thm.3} by the following result.

\begin{Lemma}\label{comput}
The following identity holds true:
\begin{equation}\label{d-s}
d_s 
=\Gamma^2(1+s).    
\end{equation} 
\end{Lemma}
We give the proof of \eqref{d-s} in what follows.
\subsection{Proof of \eqref{d-s}} \label{sec10.2} 
Let $\eta=1-\rho$.
Let $N=1$, $\Omega= (-1,+1)$ and the family of transformations $\phi_t:\R\to\R$ given  by $\phi_t(x)=(1+t)x$ so that $\Omega_t = (-1-t,1+t)$ and $Y(x)=\partial_t{\phi_t}\big|_{t=0}(x)$ is simply given by $Y(x)=x$. Let $u_t\in C^s([-1-t,1+t])$ be the solution of 
$$
(-\Delta)^su_t=1\quad\text{in}\quad (-1-t,1+t)\quad \&\quad u_t=0\quad\text{in}\quad\R\setminus (-1-t,1+t).
$$ 
It is well known that $u_t$ is explicitly given by 
\begin{align*}
    u_t(x)=l_s\Big((1+t)^2-|x|^2\Big)^s_+\quad\text{with}\quad l_s :=\frac{2^{-2s}\Gamma(1/2)}{\Gamma(s+1/2)\Gamma(1+s)},
\end{align*}
where the index $+$ indicates that we consider the nonnegative part.
We also know, see e.g \cite{Bucur}, that the fractional Green function $G_s$, for $s\neq 1/2$, 
in $1D$ is explicitly given by  
$$
G_s(x,y)=\frac{\Gamma(1/2)}{2^{2s}\pi^{1/2}\Gamma^2(s)}|x-y|^{2s-1}\int_{0}^{r_0(x,y)}\frac{t^{s-1}dt}{(1+t)^{1/2}},
\, \text{with }
\, r_0(x,y)=\frac{(1-|x|^2)(1-|y|^2)}{|x-y|^2}.
$$
It is straightforward to check that 
$$\frac{u_0}{\delta^s}(y)=2^s l_s\;\; \text{for  $y= \pm 1$}, \quad 
 u_0'(x)=-2s l_s\frac{x}{(1-|x|^2)^{1-s}},  \text{for  $x \in  (-1,1)$},
$$
$$
\frac{G_s(x,\cdot)}{\delta^s}(y)=\frac{2^s}{s}\frac{\Gamma(1/2)}{2^{2s}\pi^{1/2}\Gamma^2(s)}(1-|x|^2)^s|x-y|^{-1},
$$
and
$$
v'(x)=\partial_t(u_t\circ\phi_t)(x)\big|_{t=0}=2sl_s(1-|x|^2)^s.
$$
Replacing these values into \eqref{sec6-10}, and recalling \eqref{shape-deriv}, we get 
\begin{align*}
    &2sl_s(1-x^2)^s+2sl_s\frac{x^2}{(1-x^2)^{1-s}}
    =2^sl_sd_s\frac{2^s}{s}\frac{\Gamma(1/2)}{2^{2s}\pi^{1/2}\Gamma^2(s)}(1-x^2)^s\Big(\frac{1}{1+x}+\frac{1}{1-x}\Big),
\end{align*}
that is, after simplifications, 
$d_s=s^2\Gamma^2(s)=\Gamma^s(1+s)$.
The proof of Lemma \ref{comput} is thus finished.
\section{Proof of Theorem \ref{thm.5}}
\label{HGatlast}
In this section we give a rigorous proof of Theorem \ref{thm.5}. Formally, identity \eqref{var-green} follows from \eqref{repres-shape-deriv} by taking $u_0=G_s(y,\cdot)$. But of course, this is only formal since $G_s(x,\cdot)\notin H^s(\R^N)$. This lack of Sobolev regularity of the Green function makes the proof of Theorem \ref{thm.5} much more difficult compared to the one of Theorem \ref{thm.3}. In fact, here, even the differentiability of the mapping $(-\varepsilon_0,\varepsilon_0)\ni t\mapsto G_{\Omega_t}(\phi_t(x),\cdot)\circ \phi_t\in \mathcal L^1_s(\R^N)$ is not clear to us. However, it is not unreasonable to believe that this must be the case. We therefore add it as an assumption. Here and hereafter, the weighted Lebesgue space $\mathcal L^1_s(\R^N)$ is defined as 
$$
\mathcal L^1_s(\R^N)=\Big\{u:\R^N\to\R: \int_{\R^N}\frac{|u(x)|}{1+|x|^{N+2s}}dx<\infty\Big\}.
$$
\subsection{Sketch of proof of Theorem \ref{thm.5}}

For  any $w\in C^\infty_c(\Omega)$, for any $x\in\Omega$, we  set:
\begin{align}\label{duality-pairing}
\mathcal E_Y\big(G_s(x,\cdot),w\big)&:=
\iint_{\R^N\times\R^N}\Big(G_s(x,y)-G_s(x,z)\Big)\Big(w(y)-w(z)\Big)\mathcal K_Y(y,z)dydz ,
\end{align}
where we recall that  $\mathcal K_Y(\cdot,\cdot)$ is the deformation kernel defined in \eqref{defok}. 
We note that since $\mathcal K_Y(y,z)$ is comparable to $|z-y|^{-N-2s}$, the integral in \eqref{duality-pairing} is well defined by Lemma \ref{lem-7.2}.

We start the proof with the following non-trivial result the proof of which is postponed to Subsection \ref{subsec-6.2}
\begin{Lemma}\label{lem-eq-shape-deriv-Green}
    Fix $x\in\Omega$ and assume that the mapping $(-\varepsilon_0,\varepsilon_0)\ni t\mapsto G_{\Omega_t}(\phi_t(x),\cdot)\circ \phi_t\in \mathcal L^1_s(\R^N)$ is differentiable at $0$ and assume that $s\in (0,1/2)$.
    Then we have \begin{align}\label{eq-shap-green}
(-\Delta)^s\Bigg(\partial_t G_{\Omega_t}\big(\phi_t(x),\phi_t(\cdot)\big) {\Big|_{t=0}}\Bigg)  =-F_Y\big[G_s(x,\cdot)\big]\quad \text{in } \quad \mathcal D'(\Omega)  
\end{align}
in the sense that 
$$
\int_{\Omega}\partial_t G_{\Omega_t}\big(\phi_t(x),\phi_t(\cdot)\big) {\Big|_{t=0}}(-\Delta)^sw\,dz
=-\Big<F_Y\big[G_s(x,\cdot)\big],w\Big>=\mathcal E_Y\big(G_s(x,\cdot),w\big)$$
for all $w\in C^\infty_c(\Omega)$.
\end{Lemma}
\begin{remark}
We expect that the identity \eqref{eq-shap-green} holds for all $s\in (0,1)$. However, because of lack of good --uniform in $t$--  estimate for the difference quotient $\big(H_{\Omega_t}(\phi_t(x),\phi_t(y))-H_{\Omega}(x,z)\big)/t$, we were not able to extend it for $s\in (1/2,1)$.
\end{remark}

From now on, we fixed $x,y\in\Omega$ such that $x\neq y$ and defined  
\begin{equation}
 \mathcal W_{x,y}(\cdot):=\eta_k^2(\cdot)G_s(y,\cdot)\psi_{\mu,x}(\cdot)\psi_{\gamma,y}(\cdot)   \qquad\text{for $\mu,\gamma\in (0,1)$}
\end{equation}

 where $\eta_k$ and $\psi_{\mu,z}$ are defined as in \eqref{psi-mu} and \eqref{eta-k}. We use $\mathcal W_{x,y}\in C^\infty_c(\Omega)$ as a test function into \eqref{eq-shap-green} to get
\begin{align}\label{Key-idd-02}
 \int_{\Omega}\partial_t G_t\big(\phi_t(x),\phi_t(\cdot)\big) {\Big|_{t=0}}(-\Delta)^s\Big(\eta_k^2 G_s(y,\cdot)\psi_{\mu,x}\psi_{\gamma,y}\Big)dz+\mathcal E_Y\big(G_s(x,\cdot), \mathcal W_{x,y}\big)=0.
\end{align}
Note that since $\mathcal W_{x,y}\in C^\infty_c(\Omega)$, then by assumption and thanks to Lemma \ref{lem-7.2}, the quantities in the identity above are well defined.
\par\;

For a real valued function $u$, and to positions $p$ and $q$, set $[ u] := u(p) - u(q)$. Then we observe that for any triple of functions  $u,v, w$, 
\begin{align}\label{64}
  [uv]    [w] =  [u]    [vw] +  [v]  \Big(u(q)w(p)-u(p)w(q)\Big).
  \end{align}
To establish this, that is only a matter to develop on both sides. The left hand sides then contain four terms, the right one contains twice more, but one may observe that two terms with different signs of $u(q)v(p)w(p)$ cancel out as well as two terms with different signs of $u(p)v(q)w(q)$. The remaining terms match hence the identity above. 
Therefore
 \begin{equation}
    \mathcal E_Y\big(uv,w\big)=\mathcal E_Y\big(u,vw\big)+\pv\iint_{\R^{2N}}\Big(v(p)-v(q)\Big)\Big(u(q)w(p)-u(p)w(q)\Big)\mathcal K_Y(p,q)\, dp\, dq  .
\end{equation}
Applying the identity above with
\begin{align} \label{def-u-v-w}
w=G_s(x,\cdot), \quad v=\eta_k\psi_{\mu,x},\quad \& \quad u =\eta_kG_s(y,\cdot)\psi_{y,\gamma} 
\end{align}
gives
\begin{align}
\mathcal E_Y\Big(G_s(x,\cdot);\mathcal W_{x,y}\Big)&=\mathcal E_Y\Big(G_s(x,\cdot);\; \eta_k^2 G_s(y,\cdot)\psi_{\mu,x}\psi_{\gamma,y}\Big)\nonumber\\
&=\mathcal E_Y\Big(\eta_k\psi_{\mu,x}G_s(x,\cdot);\; \eta_k G_s(y,\cdot)\psi_{\gamma,y}\Big) +\mathcal R_{k,\mu,\gamma}(x,y)
\end{align}
where we denote 
\begin{align}\label{def-R}
    \mathcal R_{k,\mu,\gamma}(x,y)=\pv\iint_{\R^{2N}}\Big(v(p)-v(q)\Big)\Big(u(q)w(p)-u(p)w(q)\Big)\mathcal K_Y(p,q)\, dp\, dq 
\end{align}
with $u,v$ and $w$ given as in \eqref{def-u-v-w}.  Since $\eta_k\psi_{\mu,x}G_s(x,\cdot)$ and $ \eta_k G_s(y,\cdot)\psi_{\gamma,y}$ are $C^\infty$ with compact support in $\Omega$, by \cite[Lemma 2.1]{DFW} we have 
\begin{align}
&\mathcal E_Y\Big(\eta_k\psi_{\mu,x}G_s(x,\cdot), \eta_k G_s(y,\cdot)\psi_{\gamma,y}\Big)\nonumber\\
&=-\int_{\Omega}Y\cdot\nabla\Big(G_s(x_1,\cdot)\eta_k\psi_{x,\mu}\Big)(-\Delta)^s\Big(\eta_k G_s(y,\cdot)\psi_{\gamma,y}\Big)dz\nonumber\\
&\quad-\int_{\Omega}Y\cdot\nabla\Big(G_s(x,\cdot)\eta_k\psi_{x,\mu}\Big)(-\Delta)^s\Big(\eta_k G_s(y,\cdot)\psi_{\gamma,y}\Big)dz\nonumber\\
&:=-\mathcal A_{k,\mu,\gamma}(x,y)-\mathcal B_{k,\mu,\gamma}(x,y).
\end{align}

Applying the product rule \eqref{pl} for the fractional Laplacian we decompose again each of the two terms above into two terms:
\begin{align}
\mathcal A_{k,\mu,\gamma}(x,y)&=    \int_{\Omega}Y\cdot\nabla\Big(G_s(x,\cdot)\eta_k\psi_{x,\mu}\Big)\Bigg\{G_s(y,\cdot)\psi_{\gamma,y}(-\Delta)^s\eta_k-\mathcal I_s[\eta_k,G_s(y,\cdot)\psi_{\gamma,y}]\Bigg\}dz\nonumber\\
&\quad +\int_{\Omega}\eta_kY\cdot\nabla\Big(G_s(x,\cdot)\eta_k\psi_{x,\mu}\Big)(-\Delta)^s\Big(G_s(y,\cdot)\psi_{x,\mu}\Big)dz\nonumber\\
&\quad \quad :=\mathcal A_{k,\mu,\gamma}^1(x,y)+\mathcal A_{k,\mu,\gamma}^2(x,y),
\end{align}
and 
\begin{align}
\mathcal B_{k,\mu,\gamma}(x,y)&=    \int_{\Omega}Y\cdot\nabla\Big(G_s(y,\cdot)\eta_k\psi_{y,\mu}\Big)\Bigg\{G_s(x,\cdot)\psi_{\gamma,x}(-\Delta)^s\eta_k-\mathcal I_s[\eta_k,G_s(x,\cdot)\psi_{\gamma,x}]\Bigg\}dz\nonumber\\
&\quad +\int_{\Omega}\eta_kY\cdot\nabla\Big(G_s(y,\cdot)\eta_k\psi_{y,\mu}\Big)(-\Delta)^s\Big(G_s(y,\cdot)\psi_{y,\mu}\Big)dz \nonumber\\
&\quad\quad :=\mathcal B_{k,\mu,\gamma}^1(x,y)+\mathcal B_{k,\mu,\gamma}^2(x,y).
\end{align}
With these notations, the identity \eqref{Key-idd-02} becomes
\begin{align}\label{Key-idd-03}
&\mathcal C_{k,\mu,\gamma}(x,y)-\mathcal A_{k,\mu,\gamma}^2(x,y)-\mathcal B_{k,\mu,\gamma}^2(x,y)+\mathcal R_{k,\mu,\gamma}(x,y)=\mathcal A_{k,\mu,\gamma}^1(x,y)+\mathcal B_{k,\mu,\gamma}^1(x,y) 
\end{align}
with 
\begin{align}\label{def-C}
 \mathcal  C_{k,\mu,\gamma}(x,y) :=  \int_{\Omega}\partial_tG_{\Omega_t}\Big(\phi_t(x),\phi_t(z)\Big)\Big|_{t=0}(-\Delta)^s\Big(\eta_k^2 G_s(y,\cdot)\psi_{\mu,x}\psi_{\gamma,y}\Big)(z)dz.
\end{align}
We know (see e.g \cite[Proposition 2.2]{DFW} or Lemma \ref{DFW} above) that 
\begin{align}
    \lim_{\gamma\to 0^+} \lim_{\mu\to 0^+}\lim_{k\to\infty}\mathcal A_{k,\mu,\gamma}^1(x,y)=\frac{\Gamma^2(1+s)}{2}\int_{\partial\Omega}\frac{G_s(x,\cdot)}{\delta^s_\Omega}\frac{G_s(y,\cdot)}{\delta^s_\Omega}Y\cdot\nu d\sigma,\label{lim-A1}\\
    \lim_{\gamma\to 0^+} \lim_{\mu\to 0^+}\lim_{k\to\infty}\mathcal B_{k,\mu,\gamma}^1(x,y)=\frac{\Gamma^2(1+s)}{2}\int_{\partial\Omega}\frac{G_s(x,\cdot)}{\delta^s_\Omega}\frac{G_s(y,\cdot)}{\delta^s_\Omega}Y\cdot\nu d\sigma, \label{lim-B1} 
\end{align}
where $\nu$ denotes the outward unit normal.  Moreover, a similar argument as in Lemma \ref{Lem4.2} yields 
\begin{align}
   \lim_{\gamma\to 0^+} \lim_{\mu\to 0^+}\lim_{k\to\infty}\mathcal A_{k,\mu,\gamma}^2(x,y)=\nabla_y G_s(x,y)\cdot Y(y),\label{lim-A2}\\
   \lim_{\gamma\to 0^+} \lim_{\mu\to 0^+}\lim_{k\to\infty}\mathcal B_{k,\mu,\gamma}^2(x,y)=\nabla_x G_s(x,y)\cdot Y(x).\label{lim-B2}
\end{align}
In view of \eqref{Key-idd-03}, \eqref{def-C}, \eqref{lim-A1}, \eqref{lim-B1}, \eqref{lim-A2} and \eqref{lim-B2}, to prove identity \eqref{var-green} it suffices to check that for all $x,y\in \Omega$ with $x\neq y$, there holds:
\begin{align}\label{toproove}
\lim_{\gamma\to 0^+}\lim_{\mu\to 0^+}\lim_{k\to\infty} \Big[\mathcal C_{k,\mu,\gamma}(x,y)+\mathcal R_{k,\mu,\gamma}(x,y)    \Big]= \partial_tG_{\Omega_t}\Big(\phi_t(x),\phi_t(y)\Big)\Big|_{t=0}.
\end{align}
\par\;

To prove \eqref{toproove}, we shall establish the following two identities:
\begin{align}\label{C1}
    &\lim_{\gamma\to 0^+}\lim_{\mu\to 0^+}\lim_{k\to\infty}\mathcal  C_{k,\mu,\gamma}(x,y)\nonumber\\
    &= \small\text{$\partial_tG_{\Omega_t}\Big(\phi_t(x),\phi_t(y)\Big)\Big|_{t=0}+b_{N,s}(N-2s)G_s(x,y)\int_{\R^N}\frac{[DY(x)\cdot z]\cdot z}{|z|^{N-2s+2}}(-\Delta)^s(\rho\circ |\cdot|^2)(z)dz$,}
\end{align}
and
\begin{align}\label{C2}
  &\lim_{\gamma\to 0^+}\lim_{\mu\to 0^+}\lim_{k\to\infty} \mathcal R_{k,\mu,\gamma}(x,y)\nonumber\\
  &=G_s(y,x)b_{N,s}\iint_{\R^N\times\R^N}\frac{\big(\rho(|q|^2)-\rho(|p|^2)\big(|p|^{2s-N}-|q|^{2s-N}\big)}{|p-q|^{N+2s}}\mathcal W_{Y,x}(p,q)\, dp\, dq  . 
\end{align}

   To deduce \eqref{toproove} from \eqref{C1} and \eqref{C2}, we shall make use of the following: for a $x\in\Omega$ fixed, and a given vector field $Y\in C^{0,1}(\R^N,\R^N)$ such that $p$ and $q$ in $\R^N$, we set
\begin{equation} \label{defWsam}
 \mathcal W_{Y,x}(p,q) := \frac{c_{N,s}}{2}\Bigg\{ 2\div Y(x)-(N+2s)\frac{[DY(x)\cdot(p-q)]\cdot(p-q)}{|p-q|^2}\Bigg\}.     
   \end{equation}
With this notation, we have:
\begin{Lemma} \label{appen}
For all $w\in C^\infty_c(\R^N)$, there holds:
   \begin{align}\label{surprise2}
   & \iint_{\R^N\times\R^N}\frac{\big(w(q)-w(p)\big(|p|^{2s-N}-|q|^{2s-N}\big)}{|p-q|^{N+2s}}\mathcal W_{Y,x}(p,q)\, dp\, dq \nonumber\\
     &= - (N-2s)\int_{\R^N}\frac{[DY(x)\cdot p]\cdot p}{|p|^{N-2s+2}}(-\Delta)^sw(p)dp    .
   \end{align}
   \end{Lemma}
The proof of Lemma  \ref{appen} is given in Section \ref{yet}.

   Adding \eqref{C1} and \eqref{C2} and using Lemma \ref{appen} with $w=\rho\circ |\cdot|^2$,  we get \eqref{toproove}. 
Then the proof of Theorem \ref{thm.5} is done up to the proofs of Lemma \ref{lem-eq-shape-deriv-Green}, of Lemma  \ref{appen}, of 
 \eqref{C1} and of \eqref{C2}, what are the issues addressed in the next sections.

\subsection{Proof of Lemma \ref{lem-eq-shape-deriv-Green}}\label{subsec-6.2}

This section is devoted to the proof of  Lemma \ref{lem-eq-shape-deriv-Green}.
The following first result is used in the proof of Lemma \ref{lem-eq-shape-deriv-Green} below.

\begin{Lemma}\label{lem-7.2}
  Let $s\in (0,1)$. Then for all $w\in C^\infty_c(\Omega)$, 
    \begin{align}
        \mathcal E_s\Big(G_s(x,\cdot), w\Big):=\frac{c_{N,s}}{2}\iint_{\R^N\times\R^N} \frac{\big(G_s(x,p)-G_s(x,q)\big)\big(w(p)-w(q)\big)}{|p-q|^{N+2s}}\, dp\, dq < \infty.
    \end{align}
\end{Lemma}
\begin{proof}
We use the decomposition \eqref{Eq-spliting-of-Green} of the fractional Green function to split the quantity at stake into 
    $$
    \mathcal E_s\Big(G_s(x,\cdot), w\Big)=
    \mathcal E_s\Big(F_s(x,\cdot) , w\Big)
    -\mathcal E_s\Big(H_s(x,\cdot), w\Big).
    $$
    Arguing as in the proof of Lemma \ref{Lem6.1} above, we check that $ \mathcal E_s\Big(F_s(x,\cdot) , w\Big)
    <\infty$. Next choose $K\Subset K'\Subset\Omega$ with $K, K'$ compact and such that $\supp w\subset K$ and write 
    \begin{align}\label{eq-lem-7.2}
      \frac{2}{c_{N,s}}\mathcal E_s\Big(H_s(x,\cdot), w\Big)&=  \iint_{K'\times  K'} \frac{\big(H_s(x,p)-H_s(x,q)\big)\big(w(p)-w(q)\big)}{|p-q|^{N+2s}}\, dp\, dq \nonumber\\
      &+2\int_{\supp(w)}w(p)\int_{\R^N\setminus K'}\frac{H_s(x,p)-H_s(x,q)}{|p-q|^{N+2s}}\, dp\, dq 
    \end{align}
    Since $H_s(x,\cdot)\in C^\infty_{loc}(\Omega)$, it is clear that the first integral in LHS of \eqref{eq-lem-7.2} is finite.  On the other hand, 
    by   the maximum principle (see e.g \cite[Lemma 3.2.1]{Abatangelo}), we have $|H_s(x,y)|\leq b_{n,s} |x-y|^{-N+2s}$ for all $x,y\in\R^N$. Consequently, 
\begin{align}
&\Bigg|\int_{\supp(w)}w(p)\int_{\R^N\setminus K'}\frac{H_s(x,p)-H_s(x,q)}{|p-q|^{N+2s}}\, dp\, dq \Bigg|\nonumber\\
&\leq C\int_{K}|w(p)|\int_{\R^N\setminus K'}\Big(\frac{1}{|x-p|^{N-2s}}+\frac{1}{|x-q|^{N-2s}}\Big)\frac{\, dp\, dq }{1+|q|^{N+2s}} <\infty.   
\end{align}
The proof of Lemma \ref{lem-7.2}  is thus finished.
\end{proof}

We now turn to the proof of Lemma \ref{lem-eq-shape-deriv-Green}.
Let $w\in C^\infty_c(\Omega)$ so that $w\circ \phi_t^{-1}\in C^\infty_c(\Omega_t)$.
We set
\begin{gather*}
    \mathcal E_{t}\Big(G_{\Omega_t}(\phi_t(x),\phi_t(\cdot)),w\Big) := \\
\iint_{\R^N\times\R^N}\Big(G_{\Omega_t}(\phi_t(x),\phi_t(y))-G_{\Omega_t}(\phi_t(x),\phi_t(z))\Big)\Big(w(y)-w(z)\Big)\mathcal K_t(y,z)dzdy ,
\end{gather*}
where we recall that  $\mathcal K_t$ is the kernel  defined in \eqref{knono}.
    We note that the integral above is finite as a consequence of Lemma \ref{lem-7.2}.
    
  Then,   we observe that for any $x$ in $\Omega$, for any small $t$, 
\begin{equation} \label{same}
w(x)=     \mathcal E_{t}\Big(G_{\Omega_t}(\phi_t(x),\phi_t(\cdot)),w\Big) .
\end{equation}
  Indeed, 
    by using the representation 
  \eqref{repres-sol-Dir} in the domain $\Omega_t$, we have 
    \begin{align}
      &w(x)=  \int_{\Omega_t}G_{\Omega_t}(\phi_t(x),y)(-\Delta)^s(w\circ \phi_t^{-1})(y)dy,\nonumber\\
      &=\frac{c_{N,s}}{2}\iint_{\R^N\times\R^N}\Big(G_{\Omega_t}(\phi_t(x),y)-G_{\Omega_t}(\phi_t(x),z)\Big)\Big(w\circ \phi_t^{-1}(y)-w\circ \phi_t^{-1}(z)\Big)|z-y|^{-2s-N}dzdy,
      \nonumber
    \end{align}
    which, by changing variables, leads to \eqref{same}.

Now, we are going to compute  in two different ways the following ratio: 
\begin{equation} \label{ratio}
    \frac{\mathcal E_t\Big(G_{\Omega_t}(\phi_t(x),\phi_t(\cdot)),w\Big)-\mathcal E_0\Big(G_s(x,\cdot),w\Big)}{t} .
\end{equation}
 The first way is very easy: it follows from 
\eqref{same} that the ratio in  \eqref{ratio} is $0$ since the left hand side of 
\eqref{same}  does not depend on $t$. 
Next, we claim that the ratio in  \eqref{ratio} is also equal to
\begin{align} \label{60}
&\iint_{\R^N\times\R^N}\Big(G_s(x,y)-G_s(x,z)\Big)\Big(w(y)-w(z)\Big)\frac{\mathcal K_t(y,z)-\mathcal K_0(y,z)}{t}\nonumber\\
&\quad\quad +\int_{\Omega}F_t(x,y)(-\Delta)^s w(y)dy +\mathcal R_t(x)
\end{align}
   where we set 
\begin{align}\label{def-R-t}
 \mathcal R_t(x,y) := \iint_{\R^N\times\R^N}\Big(F_t(x,y)-F_t(x,z)\Big)\Big(w(y)-w(z)\Big)\Big\{\mathcal K_t(y,z)-\mathcal K_0(y,z)\Big\},  
\end{align}
with 
    $$ F_t(x,y):= \frac{G_{\Omega_t}(\phi_t(x),\phi_t(y))-G_s(x,y)}{t} .$$
    Indeed, the ratio in  \eqref{ratio} is equal to
\begin{align}\label{60B}
&\frac{1}{t}\small\text{$\iint_{\R^N\times\R^N}\Bigg\{\Big(G_{\Omega_t}(\phi_t(x),\phi_t(y))-G_{\Omega_t}(\phi_t(x),\phi_t(z))\Big)\mathcal K_t(y,z)-\Big(G_s(x,y)-G_s(x,z)\Big)\mathcal K_0(y,z)\Bigg\}$}\nonumber\\
&\quad\quad\qquad\qquad\times \Big(w(y)-w(z)\Big)dydz\nonumber\\
&=\small\text{$\iint_{\R^N\times\R^N}\Big(G_s(x,y)-G_s(x,z)\Big)\Big(w(y)-w(z)\Big)\frac{\mathcal K_t(y,z)-\mathcal K_0(y,z)}{t}$}\nonumber\\
&+\iint_{\R^N\times\R^N}\Big(F_t(x,y)-F_t(x,z)\Big)\Big(w(y)-w(z)\Big)\mathcal K_t(y,z)\nonumber\\
&=\iint_{\R^N\times\R^N}\Big(G_s(x,y)-G_s(x,z)\Big)\Big(w(y)-w(z)\Big)\frac{\mathcal K_t(y,z)-\mathcal K_0(y,z)}{t}\nonumber\\
&+\iint_{\R^N\times\R^N}\Big(F_t(x,y)-F_t(x,z)\Big)\Big(w(y)-w(z)\Big)\mathcal K_0(y,z)dydz+\mathcal R_t(x),
\end{align}
what leads to \eqref{60}.
Therefore, 
\begin{align} \label{60bis}
&\iint_{\R^N\times\R^N}\Big(G_s(x,y)-G_s(x,z)\Big)\Big(w(y)-w(z)\Big)\frac{\mathcal K_t(y,z)-\mathcal K_0(y,z)}{t}\nonumber\\
&\quad\quad +\int_{\Omega}F_t(x,y)(-\Delta)^s w(y)dy +\mathcal R_t(x,y) = 0 .
\end{align}
We are now going to pass to the limit as $t \rightarrow 0$. 
We start with the first term. 
By direct computation  (see e.g \cite[Eq (3.3)]{DFW}) we have 
\begin{equation}\label{expansion-k-t}
\mathcal K_t(x,y)-c_{N,s}|x-y|^{-N-2s}-t\mathcal K_Y(x,y)=O(t^2)|x-y|^{-N-2s}
\end{equation}
uniformly for all $x\neq y\in \R^N$.
Consequently and since 
$$
\iint_{\R^N\times\R^N}\Big|G_s(x,y)-G_s(x,z)\Big|\Big|w(y)-w(z)\Big||x-y|^{-N-2s}dxdy<\infty,
$$
(see Lemma \ref{lem-7.2}), we get by the dominated convergence theorem that 
\begin{align}\label{1st-lim}
    &\lim_{t\to 0} \iint_{\R^N\times\R^N}\Big(G_s(x,y)-G_s(x,z)\Big)\Big(w(y)-w(z)\Big)\frac{\mathcal K_t(y,z)-\mathcal K_0(y,z)}{t}\nonumber\\
    &=\iint_{\R^N\times\R^N}\Big(G_s(x,y)-G_s(x,z)\Big)\Big(w(y)-w(z)\Big)\mathcal K_Y(x,y)dydz\nonumber\\
    &=
    \mathcal E_s\Big(G_s(x,\cdot), w\Big).
\end{align}
Next,
it is also clear that 
\begin{align}\label{2nd-lim}
   \lim_{t\to 0}\int_{\Omega}F_t(x,y)(-\Delta)^s w(y)dy =\int_{\Omega} \partial_tG_t(\phi_t(x),\phi_t(y)){\big|_{t=0}}(-\Delta)^sw(y)dy.
\end{align}
And finally we claim that 
\begin{align}\label{3th-limit}
\lim_{t\to 0^+}\mathcal R_t(x,y)=0 .
\end{align}
The identity \eqref{eq-shap-green}, tested with the function $w$, then follows by combining \eqref{1st-lim}, \eqref{2nd-lim} and \eqref{3th-limit}, so that, up to proving \eqref{3th-limit} the proof of Lemma \ref{lem-eq-shape-deriv-Green} is finished. 
\par\;

Now we turn our attention to the proof of \eqref{3th-limit}.
Hence by using the decomposition \eqref{Eq-spliting-of-Green} of the fractional Green function $G_s$ 
into the sum of  the  fundamental solution and of the fractional Robin function, we obtain that  
\begin{align}
&\Big|t\Big(F_t(x,y)-F_t(x,z)\Big)\Big|\nonumber\\
&=\Big|\Big(G_{\Omega_t}(\phi_t(x),\phi_t(y))-G_{\Omega_t}(\phi_t(x),\phi_t(z))\Big)-\Big(G_s(x,y)-G_s(x,z)\Big)\Big|\nonumber\\
&\leq b_{N,s}\Big|\phi_t(x)-\phi_t(y)|^{2s-N}-|x-y|^{2s-N}\Big|+\Big|H_{\Omega_t}(\phi_t(x),\phi_t(y))-H_{\Omega}(x,y)\Big|\nonumber\\
&+b_{N,s}\Big||\phi_t(x)-\phi_t(z)|^{2s-N}-|x-z|^{2s-N}\Big|+\Big|H_{\Omega_t}(\phi_t(x),\phi_t(z))-H_{\Omega}(x,z)\Big|\label{esf}\\
&\leq b_{N,s}\Big|\phi_t(x)-\phi_t(y)|^{2s-N}-|x-y|^{2s-N}\Big|+\Big|H_{\Omega_t}(\phi_t(x),\phi_t(y))\Big|+\big|H_{\Omega}(x,y)\big|\nonumber\\
&+b_{N,s}\Big||\phi_t(x)-\phi_t(z)|^{2s-N}-|x-z|^{2s-N}\Big|+\Big|H_{\Omega_t}(\phi_t(x),\phi_t(z))\Big|+\big|H_{\Omega}(x,z)\big|\nonumber\\
& \leq b_{N,s}\Big|\phi_t(x)-\phi_t(y)|^{2s-N}-|x-y|^{2s-N}\Big|+\frac{b_{N,s}}{|\phi_t(x)-\phi_t(y)|^{N-2s}}+\big|H_{\Omega}(x,y)\big|\nonumber\\
&+b_{N,s}\Big||\phi_t(x)-\phi_t(z)|^{2s-N}-|x-z|^{2s-N}\Big|+\frac{b_{N,s}}{|\phi_t(x)-\phi_t(z)|^{N-2s}}+\big|H_{\Omega}(x,z)\big|\label{rt-}.
\end{align}
Using the property of the transformation $\phi_t$, it is not difficult to see that the LHS of \eqref{rt-} is bounded by $|x-y|^{-2s-N}+|x-z|^{-2s-N}$ for $|t|>0$ sufficiently small. That is, 
\begin{equation}\label{esti-diff-quot}
\Big|t\Big(F_t(x,y)-F_t(x,z)\Big)\Big|\leq C\big(|x-y|^{-2s-N}+|x-z|^{-2s-N}\big)\quad\text{for $|t|>0$ sufficiently small.}
\end{equation}
Here we used that for $t>0$ small enough, there holds (see e.g \cite[Section 3]{DFW}):
$$
\big|\phi_t(x)-\phi_t(y)\big|^{-N-2s}=|x-y|^{-N-2s}\Big(1+2t\frac{(Y(x)-Y(y))\cdot(x-y)}{|x-y|^2}+O(t^2)\Big)^{-\frac{N+2s}{2}}
$$

In view of \eqref{expansion-k-t} and \eqref{esti-diff-quot} we have 
\begin{align}
    &\Big|t\Big(F_t(x,y)-F_t(x,z)\Big)\Big(w(y)-w(z)\Big)\frac{\mathcal K_t(y,z)-\mathcal K_0(y,z)}{t}\Big|\nonumber\\
    &\leq C\big|w(y)-w(z)\big|\big(|x-y|^{-2s-N}+|x-z|^{-2s-N}\big)|y-z|^{-2s-N}.
\end{align}
And since 
$$
\iint_{\R^N\times\R^N}\frac{\big|w(y)-w(z)\big|\big(|x-y|^{-2s-N}+|x-z|^{N+2s}\big) }{|y-z|^{-2s-N}}dydz<\infty\qquad\text{(provided that $2s<1$)}
.$$
we conclude by the Lebesgue dominated convergence theorem and \eqref{esf} that $\lim_{t\to 0}\mathcal R_t(x,y)=0$. This finishes the proof of the lemma.

\subsection{Proof of Lemma  \ref{appen}}
\label{yet}
To prove  Lemma  \ref{appen}, we first establish the following formula. 
\begin{Lemma} \label{preumlem}
Let $x\in\Omega$ and  $Y\in C^{0,1}(\R^N,\R^N)$ be fixed. Then for all $v,w\in C^\infty_c(\R^N)$, there holds:
\begin{align}\label{FS-int-01}
     &\iint_{\R^N\times\R^N}\frac{(w(q)-w(p))(v(p)-v(q))}{|p-q|^{N+2s}}\mathcal W_{Y,x}(p,q)\, dp\, dq \nonumber\\
     &=-\int_{\R^N}\nabla w(p)\cdot \Big[DY(x)\cdot p\Big](-\Delta)^sv(p)dp-\int_{\R^N}\nabla v(p)\cdot \Big[DY(x)\cdot p\Big](-\Delta)^sw(p)dp.
\end{align}
   \end{Lemma}
   We start with the proof of Lemma \ref{preumlem}.
\begin{proof}
We define the vector field $$\widetilde {Y_x}(p):=DY(x)\cdot p, $$ so that, for any $p$ and $q$ in $\R^N$, 
\begin{align}\label{re-mix}
    \frac{[DY(x)\cdot(p-q)]\cdot(p-q)}{|p-q|^2}=\frac{\big(\widetilde {Y_x}(p)-\widetilde {Y_x}(q)\big)\cdot(p-q)}{|p-q|^2},
\end{align}
and 
\begin{align}\label{div-Y}
    \div \widetilde {Y_x}(p)=\sum_{i=1}^N\Big<D\widetilde {Y_x}(p)\cdot e_i,e_i\Big>=\sum_{i=1}^N\Big<DY(x)\cdot e_i,e_i\Big>=\div Y(x),
\end{align}
where the divergence operator on the left hand side is taken with respect to the variable $p$.
In particular, 
\begin{align}
    &\iint_{\R^N\times\R^N}\frac{(w(q)-w(p))^2}{|p-q|^{N+2s}}\mathcal W_{Y,x}(p,q)\, dp\, dq \nonumber\\
    &=\frac{c_{N,s}}{2}\iint_{\R^N\times\R^N}\frac{(w(q)-w(p))^2}{|p-q|^{N+2s}}\Bigg\{ 2\div \widetilde {Y_x}(p)
    -(N+2s)\frac{\big(\widetilde {Y_x}(p)-\widetilde {Y_x}(q)\big)\cdot(p-q)}{|p-q|^2}
   \Bigg\}\, dp\, dq . \label{finalement}
\end{align}
We observe that 
\begin{align*}
    &-\frac12(N+2s) \iint_{\R^N\times\R^N}\frac{(w(p)-w(q))^2}{|p-q|^{N+2s}}\frac{\big(\widetilde {Y_x}(p)-\widetilde {Y_x}(q)\big)\cdot (p-q)}{|p-q|^2}\, dp\, dq \nonumber\\
    &=\lim_{\mu\to 0}\int_{\R^N}\Bigg[\int_{\R^N\setminus\overline{} B_\mu(p)}(w(p)-w(q))^2\nabla_q\big(|p-q|^{-2s-N}\big)\cdot \widetilde {Y_x}(q)dq\Bigg]dp.
\end{align*}
Applying the divergence theorem in $\R^N\setminus \overline{B_\gamma(p)}$, we get
\begin{align*}
    &-\frac{1}{2}(N+2s) \iint_{\R^N\times\R^N}\frac{(w(p)-w(q))^2}{|p-q|^{N+2s}}\frac{\big(\widetilde {Y_x}(p)-\widetilde {Y_x}(q)\big)\cdot (p-q)}{|p-q|^2}\, dp\, dq \nonumber\\
    &=\lim_{\mu\to 0}\int_{\R^N}\int_{\partial B_\mu(p)}\widetilde {Y_x}(q)\frac{(w(p)-w(q))^2}{|p-q|^{N+2s}}\, dq\, dp\nonumber\\
    &\quad-2\lim_{\mu\to 0}\int_{\R^N}\int_{\R^N\setminus\overline{B_\mu(p)}}\nabla w(q)\cdot \widetilde {Y_x}(q)\frac{w(p)-w(q)}{|p-q|^{N+2s}}\, dq\, dp\nonumber\\
    &\qquad\quad -\iint_{\R^N\times\R^N}\frac{(w(p)-w(q))^2}{|p-q|^{N+2s}}\div \widetilde {Y_x}\, dq\, dp.
\end{align*}
 From here we argue as in \cite[proof of Lemma 2.1]{DFW2023} to get that
\begin{align}
    &-\frac{c_{N,s} }{2}(N+2s) \iint_{\R^N\times\R^N}\frac{(w(p)-w(q))^2}{|p-q|^{N+2s}}\frac{\big(\widetilde {Y_x}(p)-\widetilde {Y_x}(q)\big)\cdot (p-q)}{|p-q|^2}\, dp\, dq \nonumber\\
    &=-2\int_{\R^N}\nabla w(p)\cdot \widetilde {Y_x}(p)(-\Delta)^sw(p)dp-
    c_{N,s} \div Y (x) \iint_{\R^N\times\R^N}\frac{(w(p)-w(q))^2}{|p-q|^{N+2s}}\, dp\, dq .\label{fina2}
\end{align}

Consequently, combining \eqref{finalement} and \eqref{fina2}, we observe that a cancellation occurs and we are left with: 
\begin{align*}
    &\iint_{\R^N\times\R^N}\frac{(w(q)-w(p))^2}{|p-q|^{N+2s}}\mathcal W_{Y,x}(p,q)\, dp\, dq =-2\int_{\R^N}\nabla w(p)\cdot \widetilde {Y_x}(p)(-\Delta)^sw(p)dp ,
\end{align*}
which also reads in terms of $Y$:
\begin{align}\label{FS-int}
    \iint_{\R^N\times\R^N}\frac{(w(q)-w(p))^2}{|p-q|^{N+2s}}\mathcal W_{Y,x}(p,q)\, dp\, dq =-2\int_{\R^N}\nabla w(p)\cdot \Big[DY(x)\cdot p\Big](-\Delta)^sw(p)dp.
\end{align}

The equation \eqref{FS-int}
 is a equality of two quadratic forms in $w$. By the polarisation identity, we deduce \eqref{FS-int-01} as the corresponding  equality for the associated symmetric bilinear forms.
\end{proof}
We are now going to deduce Lemma \ref{appen} from Lemma \ref{preumlem}. For the convenient of the reader, we recall the statement of Lemma \ref{appen} below.
\begin{Lemma}
Let $x\in\Omega$, $Y\in C^{0,1}(\R^N,\R^N)$ and let $w\in C^\infty_c(\R^N)$. Then
   \begin{align} 
   & \iint_{\R^N\times\R^N}\frac{\big(w(q)-w(p)\big(|p|^{2s-N}-|q|^{2s-N}\big)}{|p-q|^{N+2s}}\mathcal W_{Y,x}(p,q)\, dp\, dq \nonumber\\
     &= - (N-2s)\int_{\R^N}\frac{[DY(x)\cdot p]\cdot p}{|p|^{N-2s+2}}(-\Delta)^sw(p)dp    .
   \end{align}
   \end{Lemma}
\begin{proof}
Since we cannot apply directly Lemma \ref{preumlem} to the function $|\cdot|^{2s-N}$, we are going to truncate it near the origin thanks to a parameter $\mu$, and in the large  thanks to a parameter $R$, in order to apply  \eqref{FS-int-01};  and then we will let $\mu\to 0^+$ and $R\to\infty$ in the resulting identity. 
More precisely let $\rho\in C^\infty_c(B_2)$ with $\rho\equiv 1$ in $B_1$.
For $\mu\in (0,1)$ and $R>1$, 
we define $\psi_\mu(z)=1-\rho(\frac{|z|^2}{\mu^2})$ and $\widetilde{\psi}_R(z)=\rho(\frac{|z|^2}{R^2})$ for all $z\in\R^N$.
Let $a_{\mu,R}:\R^N\to\R$ be the function defined by 
\begin{align}\label{a-mu-r}
 a_{\mu,R}(z)=\psi_\mu(z)\widetilde{\psi}_R(z) |z|^{2s-N}.
\end{align}
Then we have
\begin{align}\label{conver-result}
&\iint_{\R^N\times\R^N}\frac{\big(w(q)-w(p)\big(|p|^{2s-N}-|q|^{2s-N}\big)}{|p-q|^{N+2s}}\mathcal W_{Y,x}(p,q)\, dp\, dq\nonumber\\
    &\qquad\qquad=
    \lim_{\mu\to 0^+}\lim_{R\to\infty}\iint_{\R^N\times\R^N}\frac{\big(w(q)-w(p)\big(a_{\mu,R}(p)-a_{\mu,R}(q)\big)}{|p-q|^{N+2s}}\mathcal W_{Y,x}(p,q)\, dp\, dq  .
\end{align}
We leave the proof of this to the reader.  We apply \eqref{FS-int-01} with $v=a_{\mu,R}$  to arrive at 
\begin{align}
    &\iint_{\R^N\times\R^N}\frac{\big(w(p)-w(q)\big(|p|^{2s-N}-|q|^{2s-N}\big)}{|p-q|^{N+2s}}\mathcal W_{Y,x}(p,q)\, dp\, dq \nonumber\\
    &=-\lim_{\mu\to 0}\lim_{R\to\infty}\int_{\R^N}\nabla a_{\mu,R}(p)\cdot \Big[DY(x)\cdot p\Big](-\Delta)^sw(p)dp\nonumber\\
    &\quad -\lim_{\mu\to 0}\lim_{R\to \infty}\int_{\R^N}\nabla w(p)\cdot \Big[DY(x)\cdot p\Big](-\Delta)^sa_{\mu,R}(p)dp. \label{combibi}
\end{align}
First we observe that 
\begin{align}
   & \lim_{\mu\to 0}\lim_{R\to\infty}\int_{\R^N}\nabla a_{\mu,R}(p)\cdot \Big[DY(x)\cdot p\Big](-\Delta)^sw(p)dp\nonumber\\
    &\quad = - (N-2s)\int_{\R^N}\frac{[DY(x)\cdot p]\cdot p}{|p|^{N-2s+2}}(-\Delta)^sw(p)dp . \label{combibi2}
\end{align}
Next we claim that 
\begin{align} \label{claim2prove}
 \lim_{\mu\to 0}\lim_{R\to \infty}\int_{\R^N}\nabla w(p)\cdot \Big[DY(x)\cdot p\Big](-\Delta)^sa_{\mu,R}(p)dp=0.   
\end{align}
To see this, we apply twice  the product rule \eqref{pl} for the fractional Laplacian to obtain
\begin{align}\label{expans-a-R-mu}
    (-\Delta)^sa_{\mu,R}&=(-\Delta)^s\Big[\psi_\mu |\cdot|^{2s-N}\widetilde{\psi}_R\Big]\nonumber\\
    &=|\cdot|^{2s-N}\psi_\mu(-\Delta)^s \widetilde{\psi}_R+|\cdot|^{2s-N}\widetilde{\psi}_R(-\Delta)^s\psi_\mu-|\cdot|^{2s-N}\mathcal I_s\big[\psi_\mu,\widetilde{\psi}_R\big] \nonumber\\
    &\qquad-\mathcal I_s\Big[|\cdot|^{2s-N}, \psi_\mu \widetilde{\psi}_R\Big],
\end{align}
where we used 
$$
(-\Delta)^s(|\cdot|^{2s-N})=0\qquad\text{in}\qquad \R^N\setminus \overline{B_\mu(0)}.
$$
By the scaling property of the fractional Laplacian we have 
\begin{align}
  (-\Delta)^s\psi_\mu(p)=-\mu^{-2s}(-\Delta)^s\Big(\rho\circ |\cdot|^2\Big)\big(\frac{p}{\mu}\big)\quad\text{for $\mu\in (0,1)$}, \\
  (-\Delta)^s \widetilde{\psi}_R(p)=R^{-2s}(-\Delta)^s\Big(\rho\circ |\cdot|^2\Big)\big(\frac{p}{R}\big)\quad\text{for $R>1$}.
\end{align}
Since $\rho\in C^\infty_c(\R^N)$, we have for all $p\in\R^N$, 
$$
\Big|(-\Delta)^s\Big(\rho\circ |\cdot|^2\Big)\big(\frac{p}{R}\big)\Big|\leq \frac{C(\rho)}{1+|p/R|^{N+2s}}\leq C(\rho).
$$
Using this, we estimate 
\begin{align}\label{flim-a-R-mu}
  &\Big|\int_{\R^N}\nabla w(p)\cdot \Big[DY(x)\cdot p\Big]|p|^{2s-N}\psi_\mu(p)(-\Delta)^s \widetilde{\psi}_R (p)dp \Big|\nonumber\\
  &\leq \frac{C}{R^{2s}}\int_{\R^N}\Big|\nabla w(p)\cdot \Big[DY(x)\cdot p\Big]|p|^{2s-N}\Big|dp\to 0\qquad\text{as}\qquad R\to\infty.
\end{align}
On the other hand, using the change of variable $\overline{p}=p/\mu$, we have 
\begin{align*}
    &\int_{\R^N}\nabla w(p)\cdot \Big[DY(x)\cdot p\Big]|p|^{2s-N}\widetilde{\psi}_R(p)(-\Delta)^s\psi_\mu(p)dp\nonumber\\
    &=-\int_{\R^N}\nabla w(\mu p)\cdot \Big[DY(x)\cdot \mu p\Big]\mu^{N-2s}|\mu p|^{2s-N}\widetilde{\psi}_R(\mu p)(-\Delta)^s\Big(\rho\circ |\cdot|^2\Big)(p)dp\nonumber\\
    &=-\mu\int_{\R^N}\nabla w(\mu p)\cdot \Big[DY(x)\cdot p\Big]| p|^{2s-N}\widetilde{\psi}_R(\mu p)(-\Delta)^s\Big(\rho\circ |\cdot|^2\Big)(p)dp.
\end{align*}
From this we deduce that 
\begin{align}\label{slim-a-R-mu}
    \lim_{\mu\to 0^+}\lim_{R\to\infty}\int_{\R^N}\nabla w(p)\cdot \Big[DY(x)\cdot p\Big]|p|^{2s-N}\widetilde{\psi}_R(p)(-\Delta)^s\psi_\mu(p)dp=0.
\end{align}
Next we claim that 
\begin{align}\label{lim-int-I-s}
    \lim_{\mu\to 0}\lim_{R\to \infty}\int_{\R^N}\nabla w(p)\cdot \Big[DY(x)\cdot p\Big]|p|^{2s-N}\mathcal I_s\big[\psi_\mu,\widetilde{\psi}_R\big](p)dp=0,
\end{align}
and 
    \begin{align}\label{lim-int-I-s-fund}
    \lim_{\mu\to 0}\lim_{R\to \infty}\int_{\R^N}\nabla w(p)\cdot \Big[DY(x)\cdot p\Big]\mathcal I_s\big[|\cdot|^{2s-N},\psi_\mu\widetilde{\psi}_R\big](p)dp=0.
\end{align}
We start with the proof of \eqref{lim-int-I-s}. 
We choose three compact sets $K_1\Subset K_2\Subset K_3\subset \R^N$ such that $\supp(w)\Subset K_1$ and $\supp(\rho)\Subset K_2$ and write 
\begin{align}
    &\frac{1}{c_{N,s}}\int_{\R^N}\nabla w(p)\cdot \Big[DY(x)\cdot p\Big]|p|^{2s-N}\mathcal I_s\big[\psi_\mu,\widetilde{\psi}_R\big](p)dp\nonumber\\
    &=\small\text{$\int_{K_1}\nabla w(p)\cdot \Big[DY(x)\cdot p\Big]|p|^{2s-N}\int_{\R^N}\frac{ \big(\rho(|p|^2/\mu^2)-\rho(|q|^2/\mu^2)\big)\big(\rho(|p|^2/R^2)-\rho(|q|^2/R^2)\big) }{ |p-q|^{N+2s} }\, dp\, dq$} \nonumber\\
    &=\small\text{$\int_{K_1}\nabla w(p)\cdot \Big[DY(x)\cdot p\Big]|p|^{2s-N}\int_{K_3}\frac{ \big(\rho(|p|^2/\mu^2)-\rho(|q|^2/\mu^2)\big)\big(\rho(|p|^2/R^2)-\rho(|q|^2/R^2)\big) }{ |p-q|^{N+2s} }\, dp\, dq$} \nonumber\\
    &+\small\text{$\int_{K_1}\nabla w(p)\cdot \Big[DY(x)\cdot p\Big]|p|^{2s-N}\int_{\R^N\setminus K_3}\frac{ \big(\rho(|p|^2/\mu^2)-\rho(|q|^2/\mu^2)\big)\big(\rho(|p|^2/R^2)-\rho(|q|^2/R^2)\big) }{ |p-q|^{N+2s} }\, dp\, dq$} .
\end{align}
On the one hand, since $\rho\in C^\infty_c(\R^N)$, we have 
\begin{align}\label{lim-int-I-s-1}
    &\small\text{$\Bigg|\int_{K_1}\nabla w(p)\cdot \Big[DY(x)\cdot p\Big]|p|^{2s-N}\int_{K_3}\frac{ \big(\rho(|p|^2/\mu^2)-\rho(|q|^2/\mu^2)\big)\big(\rho(|p|^2/R^2)-\rho(|q|^2/R^2)\big) }{ |p-q|^{N+2s} }\, dp\, dq \Bigg|$}\nonumber\\
    &\leq \small\text{$\frac{C}{\mu^2 R^2}\iint_{K_1\times K_3} |p|^{2s-N}\nabla w(p)\cdot \Big[DY(x)\cdot p\Big]\frac{\, dp\, dq }{|p-q|^{N+2s-2}}\to 0\qquad\text{as}\qquad R\to\infty.$}
\end{align}
On the other hand, we have 
\begin{align}
 &\small\text{$\Bigg|\int_{K_1}\nabla w(p)\cdot \Big[DY(x)\cdot p\Big]|p|^{2s-N}\int_{\R^N\setminus K_3}\frac{ \big(\rho(|p|^2/\mu^2)-\rho(|q|^2/\mu^2)\big)\big(\rho(|p|^2/R^2)-\rho(|q|^2/R^2)\big) }{ |p-q|^{N+2s} }\, dq\, dp \Bigg|$}\nonumber\\
 &\leq C\small\text{$\Bigg|\int_{K_1}|p|^{2s-N}\nabla w(p)\cdot \Big[DY(x)\cdot p\Big]\rho\big(\frac{|p|^2}{\mu^2}\big)\int_{\R^N\setminus K_3}\frac{\, dq\, dp}{|p-q|^{N+2s}}\Bigg|$}\nonumber\\
 &\leq C(K_1,K_3)\int_{K_1}|p|^{2s-N}\Big|\nabla w(p)\cdot \Big[DY(x)\cdot p\Big]\rho\big(\frac{|p|^2}{\mu^2}\big)\Big|dp.
\end{align}
By the Lebesgue dominated convergence theorem, it is clear that  
\begin{align}\label{lim-int-I-s-2}
    \lim_{\mu\to 0^+}\int_{K_1}|p|^{2s-N}\Big|\nabla w(p)\cdot \Big[DY(x)\cdot p\Big]\rho\big(\frac{|p|^2}{\mu^2}\big)\Big|dp=0.
\end{align}
The identity \eqref{lim-int-I-s} follows by combining \eqref{lim-int-I-s-1} and \eqref{lim-int-I-s-2}. \par \, 

Finally we prove \eqref{slim-a-R-mu}. For this we write 
\begin{align*}
    &\psi_\mu(p)\widetilde{\psi}_R(p)-\psi_\mu(q)\widetilde{\psi}_R(q)= \psi_{\mu}(p)\Big(\widetilde{\psi}_R(p)-\widetilde{\psi}_R(q)\Big)+\widetilde{\psi}_R(q)\Big(\psi_{\mu}(p)-\psi_{\mu}(q)\Big),
    \end{align*}
so that 
\begin{align}\label{slim-a-R-mu-1}
    &\mathcal I_s\Big[|\cdot|^{2s-N},\psi_\mu\widetilde{\psi}_R\Big](p)\nonumber\\
    &=\psi_\mu(p)\mathcal I_s\Big[|\cdot|^{2s-N},\widetilde{\psi}_R\Big](p)+c_{N,s}\int_{\R^N}\widetilde{\psi}_R(q)\frac{\Big(|p|^{2s-N}-|q|^{2s-N}\Big)\Big(\psi_{\mu}(p)-\psi_{\mu}(q)\Big)}{|p-q|^{N+2s}}dq.
\end{align}
Using the estimate \eqref{elem-id}, a similar argument as above gives also
\begin{align}\label{slim-a-R-mu-2}
    \lim_{\mu\to 0^+}\lim_{R\to\infty}\int_{\R^N}\nabla w(p)\cdot \Big[DY(x)\cdot p\Big]\psi_\mu(p)\mathcal I_s\Big[|\cdot|^{2s-N},\widetilde{\psi}_R\Big](p)dp=0.
\end{align}
Next, for every fixed $\varepsilon>0$, a similar argument as in \eqref{sec7-11} gives
\begin{align}\label{slim-a-R-mu-3}
    \lim_{\mu\to 0^+}&\lim_{R\to\infty}\int_{\R^N}\nabla w(p)\cdot \Big[DY(x)\cdot p\Big]\int_{\R^N}\widetilde{\psi}_R(q)\frac{\Big(|p|^{2s-N}-|q|^{2s-N}\Big)\Big(\psi_{\mu}(p)-\psi_{\mu}(q)\Big)}{|p-q|^{N+2s}}\, dq\, dp\nonumber\\
    &= \lim_{\mu\to 0^+}\int_{\R^N}\nabla w(p)\cdot \Big[DY(x)\cdot p\Big]\int_{\R^N}\frac{\Big(|p|^{2s-N}-|q|^{2s-N}\Big)\Big(\psi_{\mu}(p)-\psi_{\mu}(q)\Big)}{|p-q|^{N+2s}}\, dq\, dp\nonumber\\
    &=\lim_{\mu\to 0^+}\int_{B_{\frac{\varepsilon}{\mu}}(0)}\nabla w(\mu p)\cdot \Big[DY(x)\cdot \mu p\Big]\int_{B_{\frac{2\varepsilon}{\mu}}(0)}\frac{\Big(|p|^{2s-N}-|p|^{2s-N}\Big)\Big(\rho(|p|^2)-\rho(|q|^2)\Big)}{|p-q|^{N+2s}}\, dp\, dq \nonumber\\
    &=0.
\end{align}
by the Lebesgue dominated convergence theorem. In view of \eqref{slim-a-R-mu-1}, \eqref{slim-a-R-mu-2} and \eqref{slim-a-R-mu-3}, the proof of \eqref{slim-a-R-mu} is finished.

Recalling \eqref{expans-a-R-mu} and combining \eqref{flim-a-R-mu}, \eqref{slim-a-R-mu}, \eqref{lim-int-I-s} and \eqref{lim-int-I-s-fund} we arrive at \eqref{claim2prove}.

The combination of \eqref{combibi}, \eqref{combibi2} and \eqref{claim2prove} concludes the proof of Lemma \ref{appen}.

\end{proof}
\subsection{Proof of  \eqref{C2}.}
To proceed with the proofs of \eqref{C2}, we need some preliminary results. We start with the following results which are consequences of the estimates in \cite[Lemma 6.7, Lemma 6.8]{DFW}.
\begin{Lemma}\label{lem-crucial} Let $\Omega$ be a bounded open set of class $C^{1,1}$ and fix $h\in C^s(\R^N)$. Then for all $p\in \Omega$ there holds:
\begin{align}\label{77}
    |(-\Delta)^s\eta_k(p)|\leq C(N,s)\delta^{-2s}_\Omega(p)\quad\text{and}\quad (-\Delta)^s\eta_k(p)\to 0\quad\text{as}\quad k\to\infty, 
\end{align}
\begin{align}\label{79}
&\int_{\R^N}\frac{|\eta_k(p)-\eta_k(q)||h(p)-h(q)|}{|p-q|^{N+2s}}dq \leq C(N,s)\delta^{-s}_\Omega(p),\\
&\int_{\R^N}\frac{|\eta_k(p)-\eta_k(q)||h(p)-h(q)|}{|p-q|^{N+2s}}dq\to 0\quad \text{as}\quad k\to\infty\label{80},\\
&\Big|\mathcal I_s[\eta_k,\eta_k](p)\Big|\leq C(N,s)\delta^{-s}_\Omega(p)\quad\text{and}\quad  \mathcal I_s[\eta_k,\eta_k](p)\to 0\quad \text{as}\quad k\to\infty.\label{78}
\end{align}
\end{Lemma}
\begin{proof} We briefly sketch the proof and for details we refer to \cite{DFW}.  We recall from \cite[Proposition 6.3]{DFW} that there exists $\varepsilon'>0$ with the property that
$$
\Big|\Big[(-\Delta)^s\eta_k\Big]\big(\Psi(\sigma,r/k)\big)
\Big| \leq C\frac{k^{2s}}{1+r^{1+2s}}\leq C\frac{k^{2s}}{1+r^{2s}}\quad  \text{for}\;\; k\in \N,\;\; 0\leq r \leq k\varepsilon',\; \sigma\in\partial\Omega
$$
where $\Psi: \partial\Omega\times (0,\varepsilon)\to \Omega^\varepsilon_+$, $(\sigma,r)\mapsto \sigma+r\nu(\sigma)$.  Here $\nu$ denotes the interior unit normal to the boundary. Letting $p = \Psi(\sigma, r/k)\in\Omega^{\varepsilon'}_+
:= \{x\in \Omega : 0< \dist(x,\partial\Omega)<\varepsilon'\}
$ we deduce that

\begin{equation}\label{Ds-etak-1}
\Big|(-\Delta)^s\eta_k(p)\Big|\leq C\frac{r^{2s}}{1+r^{2s}}(\frac{k}{r})^{2s}\leq C\delta^{-2s}_\Omega(p).
\end{equation}
We also have that
\begin{equation}\label{Ds-etak-2}
\Big|(-\Delta)^s\eta_k(p)\Big|\leq C\frac{k^{2s}}{1+k^{1+2s}(r/k)^{1+2s}}\leq \frac{C}{k}\delta^{-1-2s}_\Omega(p)\to 0\quad\text{as}\quad k\to\infty.
\end{equation}
On the other hand, if $p\in\Omega\setminus \Omega_+^{\varepsilon'}$ one easily checks (see \cite[Page 19]{DFW}) that 
\begin{equation}\label{Ds-etak-3}
 |(-\Delta)^s\eta_k(p)|\leq C(N,\varepsilon)\qquad\&\qquad (-\Delta)^s\eta_k(p)\to 0\quad\text{as}\quad k\to\infty.
\end{equation}
We get \eqref{77} by combining \eqref{Ds-etak-1}, \eqref{Ds-etak-2} and \eqref{Ds-etak-3}. The other estimates are proved similarly. We refer to \cite[Lemma 6.8]{DFW} for more details.
\end{proof}
The following convergence results will also be needed for the proofs.
\begin{Lemma}\label{lem-lim-G-k-H-k} Let $\eta_k$, $\psi_{\mu,x}$, $\psi_{\gamma,y}$ and $\mathcal K_Y(\cdot,\cdot)$ be defined as above.  Let us set 
\begin{align} \label{def-G-la}
\mathcal G_{k,\mu,\gamma}(x,y)&:=\iint_{\R^{2N}}\eta_k(q)u(q)\Big(\psi_{\mu,x}(p)-\psi_{\mu,x}(q)\Big)\Big(w(p)-w(q)\Big)\mathcal K_Y(p,q)\, dp\, dq , \\
\label{def-H-la}
 \mathcal H^0_{k,\mu,\gamma}(x,y)&:=\iint_{\R^{2N}}\eta_k^2(q)w(q)\Big(\psi_{\mu,x}(p)-\psi_{\mu,x}(q)\Big)\Big(a_{y,\gamma}(p)-a_{y,\gamma}(q)\Big)
\mathcal K_Y(p,q)\, dp\, dq ,
\end{align}
   with 
   $$w=G_s(x,\cdot),\quad  u =\eta_kG_s(y,\cdot)\psi_{y,\gamma} \quad \text{ and } \quad a_{y,\gamma}=G_s(y,\cdot)\psi_{y,\gamma}.$$
   Then we have 
   \begin{align}\label{lem-lim-h-k}
    &\lim_{\mu\to 0^+}\lim_{k\to\infty}\mathcal H^0_{k,\mu,\gamma}(x,y)=0,
    \end{align}
    and in the case $2s<1$, there also holds:
    \begin{align}\label{lem-lim-g-k}
       &\lim_{\gamma\to 0^+}\lim_{\mu\to 0^+}\lim_{k\to\infty}\mathcal G_{k,\mu,\gamma}(x,y)\nonumber\\
       &=G_s(y,x)b_{N,s}\iint_{\R^N\times\R^N}\frac{\big(\rho(|q|^2)-\rho(|p|^2)\big(|p|^{2s-N}-|q|^{2s-N}\big)}{|p-q|^{N+2s}}\mathcal W_{Y,x}(p,q)\, dp\, dq ,
   \end{align}
   where we recall the notation    \eqref{defWsam} for $\mathcal W_{Y,x}(p,q)$.
\end{Lemma}
\begin{proof} We start with the proof of \eqref{lem-lim-g-k}. 
\par\;
By a similar argument as in the proof of Lemma \ref{ess-convergence-result} we have 
\begin{align*}
    &\lim_{\mu\to 0^+}\lim_{k\to\infty}\mathcal G_{k,\mu,x}(x,y)\nonumber\\
    &=\small\text{$\lim_{\mu\to 0^+}\int_{B_{\frac{\varepsilon}{\widetilde\mu}}(0)}a_{y,\gamma}(x-\widetilde\mu p)\int_{B_{\frac{2\varepsilon}{\widetilde\mu}}(0)}\widetilde\mu ^{N-2s}\big(\rho(|q|^2)-\rho(|p|^2)\big)\big(G_s(x,x-\widetilde\mu p)-G_s(x,x-\widetilde\mu q)\big)\widetilde{\mathcal K_Y}(p,q)$},
\end{align*}
where 
\begin{align*}
\widetilde{\mathcal K_Y}(p,q) :=|p-q|^{-2s-N}\mathcal W_Y(x-\widetilde\mu p,x-\widetilde\mu q),
\end{align*}
and 
\begin{align*}
&\mathcal W_Y(x-\widetilde\mu p,x-\widetilde\mu q)\nonumber\\
&=\frac{c_{N,s}}{2}\Bigg\{\div Y(x-\widetilde\mu p)+\div Y(x-\widetilde\mu q)+(N+2s)\frac{(Y(x-\widetilde\mu p)-Y(x-\widetilde\mu q))\cdot (p-q)}{\widetilde\mu|p-q|^2}\Bigg\}.    
\end{align*}
Since $$\big|a_{y,\gamma}(x-\widetilde\mu p,x-\widetilde\mu q)\big|\leq C, \qquad\big|\mathcal W_Y(x-\widetilde\mu p,x-\widetilde\mu q)\big|\leq C,$$
and \begin{align*}
    \lim_{\mu\to 0^+}\mathcal W_Y(x-\widetilde\mu p,x-\widetilde\mu q)&=\frac{c_{N,s}}{2}\Bigg\{ 2\div Y(x)-(N+2s)\frac{[DY(x)\cdot(p-q)]\cdot(p-q)}{|p-q|^2}\Bigg\},
\end{align*}
(here we used $\div Y\in C(\R^N)$); it follows by the Lebesgue dominated convergence theorem (see e.g proof of Lemma \ref{ess-convergence-result}) that 
\begin{align*}
    &\lim_{\mu\to 0^+}\lim_{k\to\infty}\mathcal G_{k,\mu,x}(x,y)\nonumber\\
    &=\small\text{$\lim_{\mu\to 0^+}\int_{B_{\frac{\varepsilon}{\widetilde\mu}}(0)}a_{y,\gamma}(x-\widetilde\mu p)\int_{B_{\frac{2\varepsilon}{\widetilde\mu}}(0)}\widetilde\mu ^{N-2s}\big(\rho(|q|^2)-\rho(|p|^2)\big)\big(G_s(x,x-\widetilde\mu p)-G_s(x,x-\widetilde\mu q)\big)\widetilde{\mathcal K_Y}(p,q)$}\nonumber\\
    &=G_s(y,x)\psi_{y,\gamma}(x)b_{N,s}\iint_{\R^N\times\R^N}\frac{\big(\rho(|q|^2)-\rho(|p|^2)\big(|p|^{2s-N}-|q|^{2s-N}\big)}{|p-q|^{N+2s}}\mathcal W_{Y,x}(p,q)\, dp\, dq .
\end{align*}
We conclude the proof of \eqref{lem-lim-g-k} by taking the limit w.r.t $\gamma$ into the identity above.
\par\;

Similarly, we have 
\begin{align}\label{lim-h-prime}
    &\lim_{\mu\to 0^+}\lim_{k\to\infty}\mathcal H^0_{k,\mu,\gamma}(x,y)\nonumber\\
    &=\lim_{\mu\to 0^+}\lim_{k\to\infty}\int_{\Omega}\eta_k^2(q)G_s(x,q)\int_{\R^N}\Big(\psi_{\mu,x}(p)-\psi_{\mu,x}(q)\Big)\Big(a_{y,\gamma}(p)-a_{y,\gamma}(q)\Big)
\mathcal K_Y(p,q)\, dp\, dq \nonumber\\
&=\small\text{$\lim_{\mu\to 0^+}\int_{B_{\frac{\varepsilon}{\widetilde\mu}}(0)}b_x(x-\widetilde\mu q)\int_{B_{\frac{2\varepsilon}{\widetilde\mu}}(0)}\big(\rho(|q|^2)-\rho(|p|^2)\big)\big(a_{y,\gamma}(x-\widetilde\mu p)-a_{y,\gamma}(x-\widetilde\mu q)\big)\widetilde{\mathcal K_Y}(p,q).$}
\end{align}
with $$b_x(x-\widetilde\mu q)=\widetilde\mu ^{N-2s}G_s(x,x-\widetilde\mu q).$$
 Moreover,
\begin{align}\label{h-est-1}
    \Big|b_x(x-\widetilde\mu q)\big(a_{y,\gamma}(x-\widetilde\mu p)-a_{y,\gamma}(x-\widetilde\mu q)\big)\widetilde{\mathcal K_Y}(p,q)\Big|\leq C|q|^{2s-N}|p-q|^{-2s-N},
\end{align}
and 
\begin{align}\label{h-est-2}
    \iint_{\R^{2N}}\frac{\big|\rho(|p|^2)-\rho(|q|^2)\big|}{|q|^{N-2s}|p-q|^{N+2s}}<\infty\qquad\text{(provided that $2s<1$)}.
\end{align}
In view of \eqref{lim-h-prime},\eqref{h-est-1} and \eqref{h-est-1} and the Lebesgue dominated convergence theorem, we deduce that 
\begin{align*}
    \lim_{\mu\to 0^+}\lim_{k\to\infty}\mathcal H^0_{k,\mu,\gamma}(x,y)=0,
\end{align*}
where we used that 
\begin{align*}
    &\lim_{\mu\to 0^+}\Big[a_{y,\gamma}(x-\widetilde\mu p)-a_{y,\gamma}(x-\widetilde\mu q)\Big]\nonumber\\
    &=\lim_{\mu\to 0^+}\Big[G_s(y,x-\widetilde\mu p)\psi_{y,\gamma}(x-\widetilde\mu p)-G_s(y,x-\widetilde\mu q)\psi_{y,\gamma}(x-\widetilde\mu q)\Big]\nonumber\\
    &=0.
\end{align*}
The proof of Lemma \ref{lem-lim-G-k-H-k} is therefore finished.
\end{proof}

With Lemma \ref{lem-crucial} and Lemma \ref{lem-lim-G-k-H-k} in hands we can now proceed with the proofs of \eqref{C1} and \eqref{C2}. We first start with the  proof of \eqref{C2}.

\begin{proof}[Proof of \eqref{C2}]
Recall
\begin{align} \label{defdeRam}
    \mathcal R_{k,\mu,\gamma}(x,y)=\pv\iint_{\R^{2N}}\Big(v(p)-v(q)\Big)\Big(u(q)w(p)-u(p)w(q)\Big)\mathcal K_Y(p,q)\, dp\, dq ,
\end{align}
where 
$$
w=G_s(x,\cdot), \quad v=\eta_k\psi_{\mu,x},\quad \& \quad u =\eta_kG_s(y,\cdot)\psi_{y,\gamma},
$$
and 
\begin{align*}
   \mathcal{K}_Y(x,y)=\frac{c_{N, s}}{2}\left[\operatorname{div} Y(x)+\operatorname{div} Y(y)-(N+2 s) \frac{(Y(x)-Y(y)) \cdot(x-y)}{|x-y|^{2}}\right]|x-y|^{-N-2 s}.
\end{align*}

On the one hand,
\begin{align}
    &v(p)-v(q)= \psi_{\mu,x}(p)\Big(\eta_k(p)-\eta_k(q)\Big)+\eta_k(q)\Big(\psi_{\mu,x}(p)-\psi_{\mu,x}(q)\Big),\\
    &u(q)w(p)-u(p)w(q)=u(q)\Big(w(p)-w(q)\Big)+w(q)\Big(u(q)-u(p)\Big),
    \end{align}
    so that 
\begin{align}
 &\Big( v(p)-v(q)\Big)\Big(u(q)w(p)-u(p)w(q)\Big)   \nonumber\\
 &=u(q)\Big(\eta_k(p)-\eta_k(q)\Big)\psi_{\mu,x}(p)\Big(w(p)-w(q)\Big)\nonumber\\
 &+\psi_{\mu,x}(p)w(q)\Big(\eta_k(p)-\eta_k(q)\Big)\Big(u(q)-u(p)\Big)\nonumber\\
 &+\eta_k(q)u(q)\Big(\psi_{\mu,x}(p)-\psi_{\mu,x}(q)\Big)\Big(w(p)-w(q)\Big)\nonumber\\
 &+\eta_k(q)w(q)\Big(\psi_{\mu,x}(p)-\psi_{\mu,x}(q)\Big)\Big(u(q)-u(p)\Big).
\end{align}
    
    On the other hand
    \begin{align}
    &\psi_{\mu,x}(p)\Big(w(p)-w(q)\Big)\nonumber\\
    &=\psi_{\mu,x}(q)G_s(x,q)-\psi_{\mu,x}(p)G_s(x,p)-G_s(x,q)\Big(\psi_{\mu,x}(q)-\psi_{\mu,x}(p)\Big), 
\end{align}
so that, replacing in the first term of the left hand side above, we arrive at 
\begin{align}
 &\Big( v(p)-v(q)\Big)\Big(u(q)w(p)-u(p)w(q)\Big)   \nonumber\\
 &=u(q)\Big(\eta_k(p)-\eta_k(q)\Big)\Big( \psi_{\mu,x}(q)G_s(x,q)-\psi_{\mu,x}(p)G_s(x,p)\Big)\nonumber\\
 &\quad -u(q)G_s(x,q)\Big(\eta_k(p)-\eta_k(q)\Big)\Big(\psi_{\mu,x}(q)-\psi_{\mu,x}(p)\Big)\nonumber\\
 &\quad\quad+\psi_{\mu,x}(p)w(q)\Big(\eta_k(p)-\eta_k(q)\Big)\Big(u(q)-u(p)\Big)\nonumber\\
 &\quad\quad\quad+\eta_k(q)u(q)\Big(\psi_{\mu,x}(p)-\psi_{\mu,x}(q)\Big)\Big(w(p)-w(q)\Big)\nonumber\\
 &\quad\quad\quad\quad+\eta_k(q)w(q)\Big(\psi_{\mu,x}(p)-\psi_{\mu,x}(q)\Big)\Big(u(q)-u(p)\Big).
\end{align}
Substituting this into \eqref{defdeRam} yields 
the following decomposition: 
\begin{align*}
    &\mathcal R_{k,\mu,\gamma}(x,y)= \mathcal  D_{k,\mu,\gamma}(x,y)-\mathcal E_{k,\mu,\gamma}(x,y)+\mathcal F_{k,\mu,\gamma}(x,y)+\mathcal G_{k,\mu,\gamma}(x,y)+\mathcal H_{k,\mu,\gamma}(x,y),
\end{align*}
with 
\begin{align*}
   \mathcal D_{k,\mu,\gamma}(x,y) &:= \iint_{\R^{2N}}u(q)\Big(\eta_k(p)-\eta_k(q)\Big)\Big(\psi_{\mu,x}(p)w(p)-\psi_{\mu,x}(q)w(q)\Big)\mathcal K_Y(p,q)\, dp\, dq ,
   \\
   \mathcal E_{k,\mu,\gamma}(x,y) &:=\iint_{\R^{2N}}-u(q)G_s(x,q)\Big(\eta_k(p)-\eta_k(q)\Big)\Big(\psi_{\mu,x}(q)-\psi_{\mu,x}(p)\Big)\mathcal K_Y(p,q)\, dp\, dq ,\\
  \mathcal F_{k,\mu,\gamma}(x,y) &:=\iint_{\R^{2N}}\psi_{\mu,x}(p)w(q)\Big(\eta_k(p)-\eta_k(q)\Big)\Big(u(q)-u(p)\Big)\mathcal K_Y(p,q)\, dp\, dq  ,
   \end{align*}
   where 
 $ \mathcal  G_{k,\mu,\gamma}(x,y)$ is given by    \eqref{def-G-la}, and 
\begin{align*}
  \mathcal  H_{k,\mu,\gamma}(x,y) :=\iint_{\R^{2N}}\eta_k(q)w(q)\Big(\psi_{\mu,x}(p)-\psi_{\mu,x}(q)\Big)\Big(u(q)-u(p)\Big)\mathcal K_Y(p,q)\, dp\, dq .
\end{align*}
Since $\Big|\mathcal K_Y(x,y)\Big|\leq C|x-y|^{-2s-N}$, we have: 
\begin{align*} 
    &\big|\mathcal D_{k,\mu,\gamma}(x,y)\big|\leq  C\int_{\Omega}|u(q)|\int_{\R^N}\frac{\big|\eta_k(p)-\eta_k(q)\big|\big| \psi_{\mu,x}(q)G_s(x,q)-\psi_{\mu,x}(p)G_s(x,p)\big|}{|p-q|^{N+2s}}\, dp\, dq ,\\ 
   & \big|\mathcal E_{k,\mu,\gamma}(x,y)\big|\leq C\int_{\Omega}G_s(x,q)\int_{\R^N}\frac{\big|\eta_k(p)-\eta_k(q)\big|\big| \psi_{\mu,x}(q)-\psi_{\mu,x}(p)\big|}{|p-q|^{N+2s}}\, dp\, dq ,
\end{align*}
Using Lemma \ref{lem-crucial}  and the estimates above, we easily check that
\begin{align}
\lim_{k\to\infty}\mathcal  D_{k,\mu,\gamma}(x,y)=0 
\quad \text{ and } \quad 
\lim_{k\to\infty}\mathcal E_{k,\mu,\gamma}(x,y)=0.\label{lim-E}
\end{align}

It remains to compute the limits of $\mathcal F_{k,\mu,\gamma}(x,y)$, $\mathcal G_{k,\mu,\gamma}(x,y)$ and $\mathcal H_{k,\mu,\gamma}(x,y)$. For this, we write 
\begin{align}\label{split-u}
&u(p)-u(q)\nonumber\\
&=\big(\eta_k(p)-\eta_k(q)\big)G_s(y,p)\psi_{\gamma,y}(p)+\eta_k(q)\big(G_s(y,p)\psi_{\gamma,y}(p)-G_s(y,q)\psi_{\gamma,y}(q)\big).
\end{align}
By \eqref{split-u}, we have the estimate: 
\begin{align*}
    |\mathcal F_{k,\mu,\gamma}|\leq &C\int_{\Omega}\psi_{\mu,x}(p)G_s(x,p)\int_{\R^N}\frac{(\eta_k(p)-\eta_k(q))^2}{|p-q|^{N+2s}}\, dq\, dp\\
    &+C\int_{\Omega}\psi_{\mu,x}(p)G_s(x,p)\int_{\R^N}\frac{\big|\eta_k(p)-\eta_k(q)\big|\big|G_s(y,p)\psi_{\gamma,y}(p)-G_s(y,q)\psi_{\gamma,y}(q)\big|}{|p-q|^{N+2s}}\, dq\, dp.
   \end{align*}
In view of this and Lemma \ref{lem-crucial} we get 
\begin{align}\label{lim-F}
\lim_{k\to\infty}F_{k,\mu,\gamma}(x,y)=0.    
\end{align}
Again using the decomposition \eqref{split-u} and 
recalling Lemma \ref{lem-crucial}, we check that

\begin{align}\label{lim-GH}
    \lim_{\gamma\to 0^+}\lim_{\mu\to 0^+}\lim_{k\to \infty} \mathcal H_{k,\mu,\gamma}(x,y)=\lim_{\gamma\to 0^+}\lim_{\mu\to 0^+}\lim_{k\to \infty}\mathcal H^0_{k,\mu,\gamma}(x,y)
\end{align}
with $\mathcal H^0_{k,\mu,\gamma}(x,y)$ given as in Lemma \ref{lem-lim-G-k-H-k}.
\par\;

We obtain \eqref{C2} by combining  \eqref{lim-E}, \eqref{lim-F}, \eqref{lim-GH} and Lemma \ref{lem-lim-G-k-H-k}.
\end{proof}

\subsection{Proof of  \eqref{C1}.}

Finally we prove \eqref{C1}.
 By the product rule \eqref{pl} for the fractional Laplacian, we have 
 \begin{align}\label{C1-1}
(-\Delta)^s\Big(\eta_k^2 G_s(y,\cdot)\psi_{\mu,x}\psi_{\gamma,y}\Big)&=\eta_k^2\psi_{\mu,x}(-\Delta)^s\Big(G_s(y,\cdot)\psi_{\gamma,y}\Big)+G_s(y,\cdot)\psi_{\gamma,y} (-\Delta)^s\Big( \eta_k^2\psi_{\mu,x}\Big)  \nonumber\\
&\quad\quad-\mathcal I_s\Big[\eta_k^2\psi_{\mu,x};G_s(y,\cdot)\psi_{\gamma,y} \Big].
 \end{align}
Firstly, since by assumption $\partial_tG_{\Omega_t}\Big(\phi_t(x),\phi_t(\cdot)\Big)\Big|_{t=0}\in C\big(\Omega\setminus\{x\}\big)$, a similar argument as in Lemma \ref{Lem4.2} , gives 
 \begin{align}\label{C1-2}
     \lim_{\gamma\to 0^+}\lim_{\mu\to 0^+}\lim_{k\to\infty}&\int_{\Omega}\eta_k^2(z)\psi_{\mu,x}(z)\partial_tG_{\Omega_t}\Big(\phi_t(x),\phi_t(z)\Big)\Big|_{t=0}(-\Delta)^s\Big(G_s(y,\cdot)\psi_{\gamma,y}\Big)dz\nonumber\\
     &=\partial_tG_{\Omega_t}\Big(\phi_t(x),\phi_t(y)\Big)\Big|_{t=0} .
 \end{align} 
 Indeed, we note that the assumption $w\in C(\Omega)$ in Lemma \ref{ess-convergence-result} is not strictly required. It can be released to $w$ is continuous at $x$ as stated in Remark \ref{rem-lem-2.6}.
 \par\;
 
Next we claim that 
\begin{align}\label{probleme}
\lim_{k\to \infty}\int_{\Omega}\partial_tG_{\Omega_t}\Big(\phi_t(x),\phi_t(z)\Big)\Big|_{t=0}\mathcal I_s\Big[\eta_k^2\psi_{\mu,x};\;G_s(y,\cdot)\psi_{\gamma,y} \Big](z)dz=0.
\end{align}
To see this, we write 
\begin{align}\label{final-spliting}
\mathcal I_s\Big[\eta_k^2\psi_{\mu,x}; \;&G_s(y,\cdot)\psi_{\gamma,y} \Big](z)=\psi_{\mu,x}(z)\mathcal I_s\Big[\eta_k^2;\;G_s(y,\cdot)\psi_{\gamma,y} \Big](z)\nonumber\\
&+c_{N,s}\int_{\R^N}\eta_k^2(h)\frac{\big(\psi_{x,\mu}(z)-\psi_{x,\mu}(h)\big)\big(G_s(y,z)\psi_{\gamma,y}(z)-G_s(y,h)\psi_{\gamma,y}(h)\big)}{|z-h|^{N+2s}}dh\nonumber\\
&=:\psi_{\mu,x}(z)\mathcal I_s\Big[\eta_k^2;\;G_s(y,\cdot)\psi_{\gamma,y} \Big](z)+\mathcal M_{k,\mu,\gamma}(x,z).
\end{align}
In the one hand, using the decomposition \eqref{spli-shap-deriv-green} and arguing as in the proof of [\eqref{lem-lim-h-k}, Lemma \ref{lem-lim-G-k-H-k}], it is not difficult to check that 
    \begin{align}\label{C1-4}
     \lim_{\gamma\to 0^+} \lim_{\mu\to 0^+}\lim_{k\to\infty}\int_{\Omega}\partial_tG_{\Omega_t}\Big(\phi_t(x),\phi_t(z)\Big)\Big|_{t=0}\mathcal M_{k,\mu,\gamma}(x,z)dz=0.
\end{align}
Indeed, by \eqref{spli-shap-deriv-green}, we may write
\begin{align}\label{justif1-C1-4}
    &\int_{\Omega}\partial_tG_{\Omega_t}\Big(\phi_t(x),\phi_t(z)\Big)\Big|_{t=0}\mathcal M_{k,\mu,\gamma}(x,z)dz\nonumber\\
    &=-b_{N,s}(N-2s)\int_{\Omega}\frac{(Y(x)-Y(z))\cdot(x-z)}{|x-z|^{N-2s+2}}\mathcal M_{k,\mu,\gamma}(x,z)dz\nonumber\\
   &\quad\quad- \int_{\Omega}\partial_tH_{\Omega_t}\Big(\phi_t(x),\phi_t(z)\Big)\Big|_{t=0}\mathcal M_{k,\mu,\gamma}(x,z)dz
\end{align}
As above, we know that when computing the limits, the integrals above can be reduced to an arbitrary small neighborhood of $x$. That is to say, we have  
\begin{align}\label{justif2-C1-4}
&\lim_{\gamma\to 0^+}\lim_{\mu\to 0^+}\lim_{k\to\infty}   \int_{\Omega}\frac{(Y(x)-Y(z))\cdot(x-z)}{|x-z|^{N-2s+2}}\mathcal M_{k,\mu,\gamma}(z)dz \nonumber\\
&=\lim_{\gamma\to 0^+}\lim_{\mu\to 0^+}\lim_{k\to\infty}   \int_{B_\varepsilon(x)}\frac{(Y(x)-Y(z))\cdot(x-z)}{|x-z|^{N-2s+2}}\mathcal M_{k,\mu,\gamma}(z)dz \nonumber\\
&=\lim_{\gamma\to 0^+}\lim_{\mu\to 0^+}  \int_{B_\varepsilon(x)}\frac{(Y(x)-Y(z))\cdot(x-z)}{|x-z|^{N-2s+2}}\mathcal I_s\big[\psi_{\mu,x},G_s(y,\cdot)\psi_{\mu,y}\big](z)dz\quad\text{for all $\varepsilon>0$} .
\end{align}
From here, we argue as in the proof of [\eqref{lem-lim-h-k}, Lemma \ref{lem-lim-G-k-H-k}] to get that 
\begin{align}\label{justif3-C1-4}
    \lim_{\gamma\to 0^+}\lim_{\mu\to 0^+}  \int_{B_\varepsilon(x)}\frac{(Y(x)-Y(z))\cdot(x-z)}{|x-z|^{N-2s+2}}\mathcal I_s\big[\psi_{\mu,x},G_s(y,\cdot)\psi_{\mu,y}\big](z)dz=0.
\end{align}
By a similar argument we also check that 
\begin{align}\label{justif4-C1-4}
&\lim_{\gamma\to 0^+}\lim_{\mu\to 0^+}  \int_{\Omega}\partial_tH_{\Omega_t}\Big(\phi_t(x),\phi_t(z)\Big)\Big|_{t=0}\mathcal I_s\big[\psi_{\mu,x},G_s(y,\cdot)\psi_{\mu,y}\big](z)dz\nonumber\\
&=\lim_{\gamma\to 0^+}\lim_{\mu\to 0^+}  \int_{B_\varepsilon(x)}\partial_tH_{\Omega_t}\Big(\phi_t(x),\phi_t(z)\Big)\Big|_{t=0}\mathcal I_s\big[\psi_{\mu,x},G_s(y,\cdot)\psi_{\mu,y}\big](z)dz=0.
\end{align}
Combining \eqref{justif1-C1-4}, \eqref{justif2-C1-4}, \eqref{justif3-C1-4} and \eqref{justif4-C1-4} we arrive at \eqref{C1-4}.
\par\;

In the other hand, denoting by $\Psi$ the transformations 
\begin{equation*}
    \Psi:\partial\Omega\times (0,\varepsilon)\to \Omega^\varepsilon_+, \;(\sigma,r)\mapsto \sigma+r\nu(\sigma)
\end{equation*}
 and arguing as in the proof of \cite[Lemma 6.8]{DFW} we get that there exists $\varepsilon'>0$ such that 
\begin{align}\label{i-s-eta-k-square}
   \Bigg| k^{-s}\mathcal I_s\Big[\eta_k^2;\;G_s(y,\cdot)\psi_{\gamma,y} \Big]\Big(\Psi\big(\frac{r}{k},\sigma\big)\Big)\Bigg|\leq \frac{C}{1+r^{1+s}}\quad\text{for}\;\;k\in \N^*, \;0\leq r<k\varepsilon',\;\sigma\in\partial\Omega.
\end{align}
We also know that, for all $\varepsilon>0$, there holds (see e.g \cite[Page 16]{DFW}) :
\begin{align}\label{Eq-pb}
&\lim_{k\to\infty}\int_{\Omega}\partial_tG_{\Omega_t}\Big(\phi_t(x),\phi_t(z)\Big)\Big|_{t=0}\psi_{\mu,x}(z)\mathcal I_s\Big[\eta_k^2;\;G_s(y,\cdot)\psi_{\gamma,y} \Big](z)dz\nonumber\\
&=\lim_{k\to\infty}\int_{\Omega^\varepsilon_+}\partial_tG_{\Omega_t}\Big(\phi_t(x),\phi_t(z)\Big)\Big|_{t=0}\psi_{\mu,x}(z)\mathcal I_s\Big[\eta_k^2;\;G_s(y,\cdot)\psi_{\gamma,y} \Big](z)dz .
\end{align}
In particular, \eqref{Eq-pb} holds also for $\varepsilon=\varepsilon'$ with $\varepsilon'$ given as in \eqref{i-s-eta-k-square}. Using once again the transformations $\Psi$ and changing variables, we get 
\begin{align}\label{lim-i-s-etak-square}
    &\Bigg|\int_{\Omega^\varepsilon_+}\partial_tG_{\Omega_t}\Big(\phi_t(x),\phi_t(z)\Big)\Big|_{t=0}\psi_{\mu,x}(z)\mathcal I_s\Big[\eta_k^2;\;G_s(y,\cdot)\psi_{\gamma,y} \Big](z)dz\Bigg|\nonumber\\
    &=\Bigg|\int_{\partial\Omega}\int_{0}^{\varepsilon'}\partial_tG_{\Omega_t}\Big(\phi_t(x),\phi_t(\Psi(r,\sigma))\Big)\Big|_{t=0}\psi_{\mu,x}(\Psi(r,\sigma))\mathcal I_s\Big[\eta_k^2;\;G_s(y,\cdot)\psi_{\gamma,y} \Big](\Psi(r,\sigma))j_{\Psi}(r,\sigma)drd\sigma\Bigg|\nonumber\\
    &=\frac{1}{k}\Bigg|\int_{\partial\Omega}\int_{0}^{k\varepsilon'}\partial_tG_{\Omega_t}\Big(\phi_t(x),\phi_t(\Psi(\frac{r}{k},\sigma))\Big)\Big|_{t=0}\psi_{\mu,x}(\Psi(\frac{r}{k},\sigma))\mathcal I_s\Big[\eta_k^2;\;G_s(y,\cdot)\psi_{\gamma,y} \Big](\Psi(\frac{r}{k},\sigma))j_{\Psi}(\frac{r}{k},\sigma)\Bigg|\nonumber\\
    &\leq \frac{C}{k^{1-s}}\int_{\partial\Omega}\int_{0}^{k\varepsilon'}\Bigg|\partial_tG_{\Omega_t}\Big(\phi_t(x),\phi_t(\Psi(\frac{r}{k},\sigma))\Big)\Big|_{t=0}\Bigg|drd\sigma\nonumber\\
    &\leq \frac{C}{k^{1-s}}\int_{\partial\Omega}\int_{0}^{\infty}\Bigg|\partial_tG_{\Omega_t}\Big(\phi_t(x),\phi_t(\Psi(\frac{r}{k},\sigma))\Big)\Big|_{t=0}\Bigg|drd\sigma\nonumber\\
    &\leq \frac{C}{k^{1-s}}\int_{\R^N}\Bigg|\partial_tG_{\Omega_t}\Big(\phi_t(x),\phi_t(z)\Big)\Big|_{t=0}\Bigg|dz\to 0\quad\text{as}\quad k\to\infty.
\end{align}
Note that the integral in \eqref{lim-i-s-etak-square} is finite since $\partial_tG_{\Omega_t}\Big(\phi_t(x),\phi_t(\cdot)\Big)\Big|_{t=0}=0$ in $\R^N\setminus\Omega$ and

$\partial_tG_{\Omega_t}\Big(\phi_t(x),\phi_t(\cdot)\Big)\Big|_{t=0}\in \mathcal L^1_s(\R^N)$ by assumption. The claim \eqref{probleme} follows from \eqref{final-spliting}, \eqref{C1-4} and \eqref{lim-i-s-etak-square}. In view of \eqref{C1-1}, \eqref{C1-2},  and \eqref{probleme}, to finish the proof of \eqref{C1} it is sufficient to prove that
 \begin{align}  \label{eq-15-01}
    &\lim_{\gamma\to 0^+} \lim_{\mu\to 0^+}\lim_{k\to\infty}\int_{\Omega}\partial_tG_{\Omega_t}\Big(\phi_t(x),\phi_t(z)\Big)\Big|_{t=0}G_s(y,z)\psi_{\gamma,y} (-\Delta)^s\Big( \eta_k^2\psi_{\mu,x}\Big)(z)dz\\
    &=-G_s(x,y)\int_{\R^N}\frac{[DY(x)\cdot z]\cdot z}{|z|^{N+2s-2}}(-\Delta)^s(\rho\circ |\cdot|^2)(z)dz.     \nonumber     
 \end{align}
 To see this, we use the fractional product rule once more to split: 
$$
(-\Delta)^s\Big( \eta_k^2\psi_{\mu,x}\Big)=
\eta_k^2(-\Delta)^s\psi_{\mu,x}
+\psi_{\mu,x}(-\Delta)^s\eta_k^2
-\mathcal I_s[\eta_k^2,\psi_{x,\mu}].
$$
Using Lemma \ref{lem-crucial} and that $\delta^{-s}_\Omega(\cdot)G_s(y,\cdot)\psi_{\gamma,y}(\cdot)\in L^\infty(\Omega)$;  one easily checks by the Lebesgue dominated convergence theorem that
\begin{align}\label{C1-3}
    &\lim_{\gamma\to 0^+} \lim_{\mu\to 0^+}\lim_{k\to\infty}\int_{\Omega}\partial_tG_{\Omega_t}\Big(\phi_t(x),\phi_t(z)\Big)\Big|_{t=0}G_s(y,z)\psi_{\gamma,y} \Big[\psi_{\mu,x}(-\Delta)^s\eta_k^2-\mathcal I_s[\eta_k^2,\psi_{x,\mu}]\Big](z)dz\nonumber\\
    &\qquad\qquad\qquad\qquad=0.
    \end{align}

Next, it follows from the decomposition \eqref{Eq-spliting-of-Green} of the fractional Green function $G_s$ that
\begin{align*}
G_{\Omega_t}\Big(\phi_t(x),\phi_t(z)\Big)=\frac{b_{N,s}}{|\phi_t(x)-\phi_t(z)|^{N-2s}}-H_{\Omega_t}\big(\phi_t(x),\phi_t(z)\big).
\end{align*} 
Consequently, 
\begin{align}\label{spli-shap-deriv-green}
&\partial_tG_{\Omega_t}\Big(\phi_t(x),\phi_t(z)\Big)\Big|_{t=0}\nonumber\\
&= -b_{N,s}(N-2s)\frac{(Y(x)-Y(z))\cdot(x-z)}{|x-z|^{N-2s+2}}-\partial_tH_{\Omega_t}\Big(\phi_t(x),\phi_t(z)\Big)\Big|_{t=0}.
\end{align}
Using this, we get:
\begin{align}
&\int_{\Omega}\eta_k^2(z)\partial_tG_{\Omega_t}\Big(\phi_t(x),\phi_t(z)\Big)\Big|_{t=0}G_s(y,z)\psi_{\gamma,y}(z) (-\Delta)^s \psi_{\mu,x}(z)dz  \nonumber\\
  &=-b_{N,s}(N-2s)\int_{\Omega}\eta_k^2(z)\frac{(Y(x)-Y(z))\cdot(x-z)}{|x-z|^{N-2s+2}}G_s(y,z)\psi_{\gamma,y}(z) (-\Delta)^s \psi_{\mu,x}(z)dz  \nonumber\\
  &\quad\quad\quad-\int_{\Omega}\eta_k^2(z)\partial_tH_{\Omega_t}\Big(\phi_t(x),\phi_t(z)\Big)\Big|_{t=0}G_s(y,z)\psi_{\gamma,y}(z) (-\Delta)^s \psi_{\mu,x}(z)dz. \label{inquiet}
\end{align}
As already remarked, when computing the limits, the integrals above can be reduced to an arbitrary small neighborhood of $x$, i.e, 
\begin{align}
    &\lim_{\gamma\to 0^+} \lim_{\mu\to 0^+}\lim_{k\to\infty}\int_{\Omega}\eta_k^2(z)\frac{(Y(x)-Y(z))\cdot(x-z)}{|x-z|^{N-2s+2}}G_s(y,z)\psi_{\gamma,y}(z) (-\Delta)^s \psi_{\mu,x}(z)dz \nonumber\\
    &=\lim_{\gamma\to 0^+} \lim_{\mu\to 0^+}\lim_{k\to\infty}\int_{B_\varepsilon(x)}\eta_k^2(z)\frac{(Y(x)-Y(z))\cdot(x-z)}{|x-z|^{N-2s+2}}G_s(y,z)\psi_{\gamma,y}(z) (-\Delta)^s \psi_{\mu,x}(z)dz ,  \label{tart1}
\end{align}
and 
\begin{align}
    &\lim_{\gamma\to 0^+} \lim_{\mu\to 0^+}\lim_{k\to\infty}\int_{\Omega}\eta_k^2(z)\partial_tH_{\Omega_t}\Big(\phi_t(x),\phi_t(z)\Big)\Big|_{t=0}G_s(y,z)\psi_{\gamma,y}(z) (-\Delta)^s \psi_{\mu,x}(z)dz\nonumber\\
    &=\lim_{\gamma\to 0^+} \lim_{\mu\to 0^+}\lim_{k\to\infty}\int_{B_\varepsilon(x)}\eta_k^2(z)\partial_tH_{\Omega_t}\Big(\phi_t(x),\phi_t(z)\Big)\Big|_{t=0}G_s(y,z)\psi_{\gamma,y}(z) (-\Delta)^s \psi_{\mu,x}(z)dz, \label{tart2}
\end{align}
for all $\varepsilon>0$. Next using \eqref{rescapsi}
and changing variables, we get
\begin{align*}
    &\lim_{\gamma\to 0^+}\lim_{\mu\to 0^+}\lim_{k\to\infty}\int_{\Omega}\eta_k^2(z)\frac{(Y(x)-Y(z))\cdot(x-z)}{|x-z|^{N-2s+2}}a_{y,\gamma}(z) (-\Delta)^s \psi_{\mu,x}(z)dz\nonumber\\
    &=-\lim_{\gamma\to 0^+}\lim_{\mu\to 0^+}\int_{B_{\frac{\varepsilon}{\widetilde\mu}}(0)}\frac{\widetilde\mu^{N-2s+1}(Y(x)-Y(x-\widetilde\mu z))\cdot z}{ |\widetilde\mu z|^{N-2s+2}}a_{y,\gamma}(x-\widetilde\mu z) (-\Delta)^s\big(\rho\circ |\cdot|^2\big)(z)dz ,
\end{align*}
where we recall that $a_{y,\gamma}(z)=G_s(y,z)\psi_{\gamma,y}(z)\in L^\infty(\Omega)$. Now since 

$$
\Big|\frac{\widetilde\mu^{N-2s+1}(Y(x)-Y(x-\widetilde\mu z))\cdot z}{ |\widetilde\mu z|^{N-2s+2}}a_{y,\gamma}(x-\widetilde\mu z)\Big|\leq C|z|^{2s-N},
$$
and
$$
\int_{\R^N}|z|^{2s-N}\big|(-\Delta)^s\big(\rho\circ |\cdot|^2\big)(z)\big|dz<\infty,
$$
it follows from the Lebesgue dominated convergence theorem that 

\begin{align} 
    &\lim_{\gamma\to 0^+}\lim_{\mu\to 0^+}\lim_{k\to\infty}\int_{\Omega}\eta_k^2(z)\frac{(Y(x)-Y(z))\cdot(x-z)}{|x-z|^{N-2s+2}}a_{y,\gamma}(z) (-\Delta)^s \psi_{\mu,x}(z)dz\nonumber\\
    &=-\lim_{\gamma\to 0^+}\lim_{\mu\to 0^+}\int_{B_{\frac{\varepsilon}{\widetilde\mu}}(0)}\frac{\widetilde\mu^{N-2s+1}(Y(x)-Y(x-\widetilde\mu z))\cdot z}{ |\widetilde\mu z|^{N-2s+2}}a_{y,\gamma}(x-\widetilde\mu z) (-\Delta)^s\big(\rho\circ |\cdot|^2\big)(z)dz \nonumber \\
    &=-G_s(x,y)\int_{\R^N}\frac{[DY(x)\cdot z]\cdot z}{|z|^{N+2s-2}}(-\Delta)^s(\rho\circ |\cdot|^2)(z)dz. \label{star}
\end{align}
Proceeding as above and since $\partial_tH_{\Omega_t}\Big(\phi_t(x),\phi_t(\cdot)\Big)\Big|_{t=0}\in C(\Omega)$ (by assumption), we easily check that
 \begin{align}\label{star2}
     \lim_{\gamma\to 0^+} \lim_{\mu\to 0^+}\lim_{k\to\infty}&\int_{B_\varepsilon(x)}\eta_k^2(z)\partial_tH_{\Omega_t}\Big(\phi_t(x),\phi_t(z)\Big)\Big|_{t=0}G_s(y,z)\psi_{\gamma,y}(z) (-\Delta)^s \psi_{\mu,x}(z)dz= 0   .
 \end{align}
By combining  \eqref{inquiet}, \eqref{tart1}, \eqref{tart2}, \eqref{star} and \eqref{star2}, we conclude that 
 \begin{align} 
    &\lim_{\gamma\to 0^+} \lim_{\mu\to 0^+}\lim_{k\to\infty}
\int_{\Omega}\eta_k^2(z)\partial_tG_{\Omega_t}\Big(\phi_t(x),\phi_t(z)\Big)\Big|_{t=0}G_s(y,z)\psi_{\gamma,y}(z) (-\Delta)^s \psi_{\mu,x}(z)dz \nonumber\\
 &=  - G_s(x,y)\int_{\R^N}\frac{[DY(x)\cdot z]\cdot z}{|z|^{N+2s-2}}(-\Delta)^s(\rho\circ |\cdot|^2)(z)dz . \label{tart5}
 \end{align}
Combining  \eqref{C1-3} and \eqref{tart5}, we 
 obtain     \eqref{eq-15-01} and therefore 
the proof of \eqref{C1}  is finished.

\section{Appendix}
\label{subsec-9.1}

The following properties of the Green function were used all along the manuscript. 
\begin{itemize}
    \item 
Let $\Omega$ be a bounded open set of $\R^N$ of class $C^{1,1}$. Then for all $x,y\in\Omega$ with $x\neq y$, there holds (see \cites{Tedeuz, CS})
\begin{equation}\label{eq:greenfunct-estimate}
    C_1\min\Big(\frac{1}{|x-y|^{N-2s}}, \frac{\delta^s_\Omega(x)\delta^s_\Omega(y)}{|x-y|^N}\Big)\leq \frac{G_s(x,y)}{\Lambda_{N,s}}\leq \min\Big(\frac{1}{|x-y|^{N-2s}}, C_2\frac{\delta^s_\Omega(x)\delta^s_\Omega(y)}{|x-y|^N}\Big)
\end{equation}
for some constant $C_1, C_2>0$ and an explicit constant $\Lambda_{N,s}$. 
\item If $\Omega$ is an arbitrary bounded open set of $\R^N$ , then (see e.g \cite[Corollary 3.3]{Bogdan2002}):
\begin{equation}\label{gradient-estimate-green-func}
    \big|\nabla _y G_s(x,y)\big|\leq N\frac{G_s(x,y)}{\min\{|x-y|,\delta_\Omega(y)\}}, \quad \text{for all $x,y\in \Omega$ with $x\neq y$}.
\end{equation}
It follows from \eqref{eq:greenfunct-estimate} and \eqref{gradient-estimate-green-func} that if $\Omega$ is of class $C^{1,1}$ and $x\in\Omega$, then $\nabla _y G_s(x,y)$ is uniformly in $x$ $L^1$-integrable in $\Omega$  provided that $s\in (1/2,1)$, see e.g \cite[Lemma 9]{Bogdan2011}.
\end{itemize}

\bigskip 
\bigskip \ \par \noindent {\bf Acknowledgements.}  This work was initiated when F. S. was visiting the  FAU DCN-AvH during some snowy days. He warmly thanks Enrique Zuazua and his team  for their kind hospitality. 


\end{document}